%% file: main.tex
\documentclass{amsart}

\usepackage{amsmath,amsthm,amsfonts,amscd,amssymb} 
\usepackage{verbatim}      	% Allows quoting source with commands.
\usepackage{epsfig}         	% Allows inclusion of eps files.
\usepackage{url}		% Allows good typesetting of web URLs.
\usepackage{epic,eepic,latexsym}
\usepackage{graphicx}
\usepackage{color}
\usepackage{lcd}
\usepackage{mathrsfs}
\usepackage[centering,text={15.5cm,22cm},
        marginparwidth=20mm,dvips]{geometry}
\usepackage[linktocpage]{hyperref}
\usepackage{lscape}
\usepackage{tabularx}
\usepackage{marginnote}
\usepackage[all]{xy}
\usepackage{enumitem}
\usepackage{stmaryrd}
\usepackage{upgreek}
\usepackage{transparent}

\DeclareMathOperator{\Hom}{Hom}

\DeclareMathOperator{\codim}{codim}

\DeclareMathOperator{\Tr}{Tr}

\DeclareMathOperator{\val}{val}

\DeclareMathOperator{\sign}{sgn}

\DeclareMathOperator{\trop}{trop}
\DeclareMathOperator{\prin}{prin}
\DeclareMathOperator{\out}{out}
\DeclareMathOperator{\In}{in}
\DeclareMathOperator{\Mult}{Mult}

\DeclareMathOperator{\Aut}{Aut}
\DeclareMathOperator{\Ends}{Ends}

\DeclareMathOperator{\scat}{Scat}

\DeclareMathOperator{\Joints}{Joints}

\DeclareMathOperator{\Supp}{Supp}

\DeclareMathOperator{\Li}{Li}

\DeclareMathOperator{\as}{as}

\DeclareMathOperator{\N}{N}
\DeclareMathOperator{\Parents}{Parents}
\DeclareMathOperator{\Ancestors}{Ancestors}
\DeclareMathOperator{\Leaves}{Leaves}

\let\?=\overline
\let\:=\colon
\let\llb=\llbracket
\let\rrb=\rrbracket
\let\bb=\mathbb

\let\rar=\rightarrow
\let\f=\mathfrak
\let\s=\mathcal

\let\wh=\widehat
\let\wt=\widetilde

\newcommand {\A} {{\bf B}}

\newcommand {\ww} {{\bf w}}

\newcommand {\pp} {{\bf p}}
\newcommand {\JJ} {{\bf J}}

\theoremstyle{plain}% default
 \newtheorem{thm}{Theorem}[section]
 
 \newtheorem{lem}[thm]{Lemma}
  
  \newtheorem{prop}[thm]{Proposition}
  
   \newtheorem{cor}[thm]{Corollary}

\theoremstyle{definition}
 \newtheorem{dfn}[thm]{Definition}
  
 \newtheorem{ntn}[thm]{Notation}
 \newtheorem{eg}[thm]{Example}
 \newtheorem{egs}[thm]{Examples}
 
\theoremstyle{remark} 
 \newtheorem{rmk}[thm]{Remark}

  \newtheorem{asss}[thm]{Assumptions}

\title{Scattering diagrams, theta functions, and refined tropical curve counts}
\author{Travis Mandel}
\address{School of Mathematics\\
University of Edinburgh\\
Edinburgh EH9 3FD\\
UK}
\email{Travis.Mandel{\char'100}ed.ac.uk}
\thanks{The author was supported by the Center of Excellence Grant ``Centre for Quantum Geometry of Moduli Spaces'' from the Danish National Research Foundation (DNRF95), and later by the National Science Foundation RTG Grant DMS-1246989, and then by the Starter Grant ``Categorified Donaldson-Thomas Theory'' no. 759967 of the European Research Council. \\ 
2010 Mathematics Subject Classification 14T05 (primary), 13F60, 14N10 (secondary).}

\begin{document}

\begin{abstract}
Working over various graded Lie algebras and in arbitrary dimension, we express scattering diagrams and theta functions in terms of counts of tropical curves/disks, weighted by multiplicities given in terms of iterated Lie brackets.  Over the tropical vertex group, our tropical curve counts are known to give certain descendant log Gromov-Witten invariants.   Working over the quantum torus algebra yields theta functions for quantum cluster varieties, and our tropical description sets up for geometric interpretations of these.  As an immediate application, we prove the quantum Frobenius conjecture of \cite{FG1}.  We also prove a refined version of the \cite{CPS} result on consistency of theta functions, and we prove the non-degeneracy of the trace-pairing for the Gross-Hacking-Keel Frobenius structure conjecture.
\end{abstract}

\maketitle

\setcounter{tocdepth}{1}
\tableofcontents

\section{Introduction}\label{Intro}
Mirror symmetry predicts a close relationship between certain counts of holomorphic disks in one space and the data of certain sections of vector bundles on a mirror space.  Motivated by this, various combinations of Gross, Hacking, Keel, Kontsevich, and Siebert \cite{CPS,GHK1,GHKK,GHS} have defined canonical ``theta functions'' in terms of combinatorial objects called scattering diagrams and broken lines which, at least heuristically, capture the data of mirror holomorphic disk counts.  In this article, building off ideas from \cite{GPS}, we show how the scattering diagrams and theta functions (along with certain refinements!) can be expressed in terms of certain counts of tropical curves and tropical disks, cf. Corollary \ref{ScatTropDisksCor} and Theorem \ref{MainTheorem}.  The original motivation is that these tropical curve counts can be related to holomorphic curve counts \cite{Mi,NS,AGr,MRud}, and this should lead to proofs of the expected mirror symmetry correspondences.  Indeed, separate work of the author \cite{ManFrob} uses this to prove that the Frobenius structure conjecture of Gross-Hacking-Keel \cite[arXiv v1, \S 0.4]{GHK1} holds for cluster varieties.

The non-refined (i.e., classical) versions of scattering diagrams and theta functions mentioned above are defined over the module of log derivations (as in \cite{GPS}), or for cluster varieties, a sub Lie algebra called the Poisson torus algebra.   But, as described in \cite{WCS}, scattering diagrams can be defined over other monoid-graded Lie algebras, e.g., the quantum torus algebra.  We extend the construction of theta functions to this more general context, and we show that the descriptions of the scattering diagrams and theta functions in terms of tropical counts apply in this general setup.  Here, the tropical curves are counted using a new version of tropical multiplicities defined in terms of iterated brackets in the Lie algebra, cf. \S \ref{MultSection}.

In the classical setting, the tropical description of the scattering diagram in dimension $2$ is \cite[Thm. 2.8]{GPS}, while a version for the higher-dimensional cases appears as part of \cite[Prop. 5.14]{CPS}.  Similarly, \cite[Prop. 5.15]{CPS} gives a description of certain classical broken lines in terms of a version of tropical disk counts.  Joint work of the author and H. Ruddat \cite{MRudMult} proves that the iterated bracket description of multiplicities introduced in \S \ref{MultSection} does in fact yield the correct multiplicities for relating the tropical curve counts here to (descendant) log Gromov-Witten invariants.

When working in dimension $2$ over the quantum torus algebra, our multiplicities are essentially genus $0$ Block-G\"ottsche multiplicities, extended to allow for some $\psi$-class conditions (the refined descendant multiplicities of \cite{BlSh} are a symmetrization of ours).  Various Block-G\"ottsche invariants (with slightly different conditions imposed than in our setup) have been interpreted in terms of refined Severi degrees \cite{BG,GSh}, and a motivic interpretation has been investigated in \cite{NPS}.

The description of the two-dimensional scattering diagrams in terms of Block-G\"ottsche invariants was previously found in \cite[Corollary 4.8]{FS}, where they noted a relationship to Poincar\'e polynomials of certain moduli of quiver representations (refining the results of \cite{GP} in terms of the corresponding Euler characteristics).  It follows from the ideas of \cite{Bridge} that this Poincar\'e polynomial description holds in higher dimensions as well.  This is explained in joint work of the author with M.-W. Cheung \cite{CM}, where we will also express the tropical multiplicities appearing here in terms of certain moduli of composition series.  New Donaldson-Thomas/Gromov-Witten correspondence theorems (and quantum refinements) will follow from comparing these equivalent descriptions of scattering diagrams.

Alternatively, these Block-G\"ottsche invariants for two-dimensional scattering diagrams have been related to certain higher-genus invariants with lambda classes by Bousseau \cite{Bou}, and to counts of real curves by Mikhalkin \cite{Mikq}.  Bousseau has used his correspondence result and the tropical description of two-dimensional quantum scattering diagrams to express these scattering diagrams, hence the mirror construction of \cite{GHK1}, in terms of these higher-genus invariants, cf. \cite{Bou2,Bou3}. Future work of the author with Bousseau will use our results to prove a refined version of the Frobenius structure conjecture in dimension $2$, relating quantum theta functions to these higher-genus invariants.  Other upcoming work of the author will extend Mikhalkin's ideas to more general conditions and higher dimensions.  This will relate our quantum tropical counts (those appearing in the quantum torus algebra cases of Corollary \ref{ScatTropDisksCor} and Theorem \ref{MainTheorem}) to certain signed counts of holomorphic disks with boundary on the positive real locus, the power of the quantum parameter $q$ giving certain areas of the disks.  The goal here is a quantum version of the Frobenius structure conjecture which enumerates open strings in the presence of a $B$-field.

As outlined above, the primary and motivating value of our main results is that the expressions in terms of tropical curve counts will lead to nice geometric interpretations.  That said, we do present two direct applications of the tropical description to understanding properties of the theta functions.  First, in \S \ref{CPSproof}, we sketch how the tropical interpretation yields a proof of Theorem \ref{CPS}, which extends to our refined setting a foundational result of Carl-Pumperla-Siebert \cite{CPS} on the global well-definedness of theta functions. Then in \S \ref{Application}, we use the tropical description to prove Fock and Goncharov's quantum Frobenius conjecture \cite[Conj. 4.8.6]{FG1} (after proving their classical Frobenius conjecture via other techniques).\footnote{Fock and Goncharov's Frobenius conjectures should not be confused with the quite different Frobenius \textit{structure} conjecture of Gross-Hacking-Keel.}  This conjecture describes the behavior of quantum theta functions at roots of unity under the action of the quantum Frobenius map.\footnote{For quantum cluster varieties from surfaces, the quantum Frobenius conjecture is \cite[Thm. 1.2.6]{AK}, although the canonical bases there are not yet known to equal to theta bases.}

\subsection{Outline of the paper}

In \S \ref{scatter} we cover the basic definitions and properties of scattering diagrams, and in \S \ref{Broken Lines} we define our general version of broken lines and theta functions.  In \S \ref{TrNonDeg}, we prove that, under certain mild assumptions, the multiplication of the theta functions is determined by a certain trace operator which maps functions to their degree $0$ terms.  This is \textbf{Theorem \ref{TrThm}}, and it is an important step in proving the Frobenius structure conjecture in \cite{ManFrob}.  Degree $0$ terms in products of theta functions (i.e., the traces) can be understood in terms of tropical curves, as opposed to just tropical disks, so the correspondence to holomorphic curves is far better understood (although it should be possible to express the tropical disk counts in terms of the punctured invariants of \cite{GSInt}).

In \S \ref{TropBasics}, we first give the basic definitions of tropical curves and tropical disks, and we describe the types of conditions we will impose on these objects.  Our definitions of the multiplicities of our tropical curves/disks are given in \S \ref{MultSection}, with simplifications for the various setups discussed in Remark \ref{SignSkewCases} and Examples \ref{MultExamples}.  In particular, our version of the Block-G\"ottsche multiplicities is explained in Example \ref{MultExamples}(iii).  The invariance of the resulting tropical counts is \textbf{Proposition \ref{TropInvariant}}.  Our main results relating scattering diagrams to tropical curve counts are \textbf{Theorem \ref{ScatTropDisks}} and \textbf{Corollary \ref{ScatTropDisksCor}}, and then the description of theta functions in terms of tropical curve counts is \textbf{Theorem \ref{MainTheorem}}. We prove these results and the refined \cite{CPS} result (\textbf{Theorem \ref{CPS}}) in \S \ref{MainProofs}. 

In \S \ref{ClusterSection} we focus on the setup relevant to cluster varieties.  We review the basic definitions of seeds in \S \ref{SeedsSection}, and in \S \ref{InClusterScat} we give the initial scattering diagrams associated to seeds when constructing theta functions on the corresponding cluster varieties and their quantizations.  Finally, in \S \ref{Application}, we prove the classical and quantum Fock-Goncharov Frobenius Conjectures (\textbf{Theorems \ref{FrobConjClass} and \ref{FrobConjQuant}}, respectively).

\subsection*{Acknowledgements}
The author thanks Tom Bridgeland, Man-Wai Cheung, Mark Gross, Sean Keel, Greg Muller, Helge Ruddat, Bernd Siebert, and Jacopo Stoppa for helpful discussions.

\section{Scattering diagrams and theta functions}\label{CanBroken}

\subsubsection*{Notation}
Given a finite-rank lattice $L$, we write $L_{\bb{Q}}:=L\otimes \bb{Q}$ and $L_{\bb{R}} := L\otimes \bb{R}$.  We denote the dual pairing between $L$ and $\Hom(L,\bb{Z})$ by $\langle \cdot,\cdot\rangle$. 

\subsection{Scattering diagrams}\label{scatter}

Let $N$ denote a lattice of finite rank $r$, and let $M$ denote the dual lattice $\Hom(N,\bb{Z})$.  Fix a strictly convex rational polyhedral cone $\sigma_{N^{\oplus}}\subset N_{\bb{R}}$.  Let $N^{\oplus}:=\sigma_{N^{\oplus}}\cap N$, and let $N^+:=N^{\oplus}\setminus \{0\}$.  Let $\f{g}:=\bigoplus_{n\in N^+} \f{g}_n$ be a Lie algebra graded by $N^+$, meaning that $[\f{g}_{n_1},\f{g}_{n_2}] \subseteq \f{g}_{n_1+n_2}$.

For each $k\in \bb{Z}_{\geq 1}$, let 
\begin{align}\label{kNplus}
kN^+:=\{n_1+\ldots+n_k\in N|n_i\in N^+ \mbox{ for each } i=1,\ldots,k\}.
\end{align} 
Let $$\f{g}^{\geq k}:=\bigoplus_{n\in kN^+} \f{g}_n.$$  Note that $\f{g}^{\geq k}$ is a Lie subalgebra of $\f{g}$.  Let $\f{g}_k$ denote the nilpotent Lie algebra $\f{g}/\f{g}^{\geq k}$, and let $\wh{\f{g}}:=\varprojlim \f{g}_k$.  We have corresponding Lie groups $G:=\exp \f{g}$, $G_k:=\exp \f{g}_k$, and $\wh{G}:=\exp \wh{\f{g}} = \varprojlim G_k$.

Let $\sigma_P$ denote a convex (but not necessarily strictly convex) rational polyhedral cone in $N_{\bb{R}}$ containing $\sigma_{N^{\oplus}}$, and let $P:=\sigma_P\cap N$.  Let $A$ be a $P$-graded algebra over $A_0$ (the degree $0$ part) on which $\f{g}$ acts via $A_0$-derivations respecting the grading, so $\f{g}_n \cdot A_p \subset A_{n+p}$.  It will occasionally be useful to think of $\f{g}\oplus A$ as a Lie algebra under the bracket 
\begin{align}\label{Ag}
    [g_1+a_1,g_2+a_2]=[g_1,g_2]+g_1\cdot a_2 - g_2\cdot a_2 +[a_1,a_2].
\end{align}
Let $A^{\oplus}$ denote the subring $\bigoplus_{n\in N^{\oplus}} A_n$ of $A$, and for each $k\in \bb{Z}_{\geq 1}$, let $A^{\geq k}$ denote the ideal $\bigoplus_{n\in kN^+} A_n$ of $A^{\oplus}$.  Then let $\wh{A}^{\oplus}:=\varprojlim_k (A^{\oplus}/A^{\geq k})$, and let $\wh{A}:=A\otimes_{A^{\oplus}} \wh{A}^{\oplus}$.  I.e., $\wh{A}$ is the $N^+$-adic completion of $A$.  Then $\wh{G}$ inherits an action on $\wh{A}$, $\wh{G}\rar \Aut(\wh{A})$. 

For any sublattice $L\subset N$ or subspace $L\subset N_{\bb{R}}$, let $\f{g}_L:=\bigoplus_{n\in L\cap N^+} \f{g}_n$ be the corresponding sub Lie algebra of $\f{g}$, and let $\wh{\f{g}}_L$ denote the $(L\cap N^+)$-adic completion.  Similarly, define $A_L:= \bigoplus_{n\in L\cap P} A_n$ and let $\wh{A}_L$ be its $(L\cap N^+)$-adic completion.

Fix a saturated sublattice $K\subset N$ such that $[\f{g},\f{g}_K]=0$ and $\f{g}\cdot A_K=0$, i.e., such that $\f{g}_K$ is central in $\f{g}$, and the action of $\f{g}$ on $A$ is via $A_K$-derivations.\footnote{The reader can safely take $K=0$ (so $N=\?{N}$) and ignore this extra generality, but in certain applications it is useful to view $A_K$ as the coefficient ring for $A$.  Similarly, a reader who is interested only in the scattering diagram, not in broken lines and theta functions, may take $A=0$ and $P=N$, so then all conditions on $A$ and $P$ become trivial.}  Let
\begin{align*}
    \pi_K:N\rar \?{N}:=N/K
\end{align*}
denote the projection, and let $\?{M}$ be the dual lattice to $\?{N}$, canonically identified with $K^{\perp}\cap M\subset M$.

We assume that $P+K$ is a lattice, i.e., that $\?{P}:=\pi_K(P)$ is a lattice, and we fix a piecewise-linear section $\varphi:\?{P}\rar P$ of $\pi_K|_P$ such that $P=\varphi(\?{P})+(K\cap P)$.  For each $p\in \?{P}$, we designate a special element $z^{\varphi(p)}\in A_{\varphi(p)}$.  We assume $\varphi(0)=0$, and $z^0=1$.  In our examples it will typically be obvious from the notation what these elements $z^{\varphi(p)}$ are.

For each $n \in N^+$, we have a Lie subalgebra $$\f{g}_{n}^{\parallel}:= \wh{\f{g}}_{\bb{Z} n} \subset \wh{\f{g}}.$$  For $n\in N^+$ and $m\in n^{\perp}\setminus \{0\}$, let $\f{g}_{n,m}^{\parallel}$ denote the sub Lie algebra of $\f{g}_n^{\parallel}$ consisting of those $g$ such that $[g,\f{g}_{m^{\perp}}]=0$ and $g\cdot A_{m^{\perp}}=0$.  For each $(n_1,m_1)$ and $(n_2,m_2)$ with $n_i\in N^+$ and nonzero $m_i\in n_i^{\perp}$ for $i=1,2$, we require that
\begin{align}\label{gnm}
    [\f{g}_{n_1,m_1}^{\parallel},\f{g}_{n_2,m_2}^{\parallel}]\subset \f{g}_{n_1+n_2,\mu((n_1,m_1),(n_2,m_2))}^{\parallel},
\end{align}
where
\begin{align}\label{m}
    \mu((n_1,m_1),(n_2,m_2)):=\langle n_2,m_1\rangle m_2-\langle n_1,m_2\rangle m_1 \in (n_1+n_2)^{\perp}.
\end{align}
This setup is motivated by the following examples, which will be built upon throughout the paper.

\begin{egs}\label{gegs} For use in these examples, we fix a commutative ring $R$.

\begin{enumerate}[label=(\roman*)]\setlength\itemsep{1em}
\item {\bf The tropical vertex group:} Let $\Theta_K(R[N^{\oplus}])$ denote the module of log derivations of $R[N^{\oplus}]$ over $R[K\cap N^{\oplus}]$:
\begin{align*}
    \Theta_K(R[N^{\oplus}]):=R[N^{\oplus}]\otimes_{\bb{Z}} \?{M} 
\end{align*}
with action on $R[N^{\oplus}]$ defined by
\begin{align*}
    f\otimes m(z^n):=f\langle n,m\rangle z^n.
\end{align*}
We write $f\otimes m$ as $f\partial_m$.  $\Theta_K(R[N^{\oplus}])$ forms a Lie algebra with bracket $[a,b]:=ab-ba$, where multiplication means composition of derivations.  In particular, one computes
\begin{align}\label{hCommutator}
    [z^{n_1}\partial_{m_1},z^{n_2}\partial_{m_2}]=z^{n_1+n_2}\partial_{\mu((n_1,m_1),(n_2,m_2))}
\end{align}
for $\mu$ as defined in \eqref{m}. 
Let
\begin{align*}
    \f{h}:=\bigoplus_{n\in N^+} \f{h}_n,
\end{align*}
where $\f{h}_n$ is the submodule of $\Theta_K(R[N^{\oplus}])$ spanned by elements of the form $z^n\partial_m$ with $\langle n,m\rangle = 0$.  One easily checks that $\f{h}$ is closed under the bracket and hence is a Lie subalgebra, clearly graded by $N^+$.  We take $\f{g}:=\f{h}$.  The corresponding pronilpotent group $\wh{G}=\wh{H}$ constructed from this $\f{g}$ as above is called the tropical vertex group.

For the algebra $A$, we take $A:=R[P]$, so $\wh{A}=:R\llb N^{\oplus} \rrb_P$ is the corresponding Laurent series ring.  One checks that an element of the form $\exp(\log(f)\partial_m)\in \wh{G}$ acts on a monomial $z^p\in \wh{A}$ via 
 \[\exp(\log(f)\partial_m)\cdot z^p = z^pf^{\langle p,m\rangle}.
 \]
 Note that $\f{g}_{n,m}^{\parallel}$ is generated by $z^n\partial_m$.  Condition \eqref{gnm} now follows from \eqref{hCommutator}.
 
\item {\bf Poisson torus algebras:}  This is a special case of the tropical vertex group example and is particularly important for cluster algebras.  For this and Example (iii) below, we assume $N$ is equipped with a $\bb{Q}$-valued skew-symmetric form $\omega=\{\cdot,\cdot\}$.  Each $\f{g}$ will be skew-symmetric with respect to $\omega$ in the sense that if $\{n_1,n_2\}=0$, then $[\f{g}_{n_1},\f{g}_{n_2}]=0$.  Similarly, the actions on $A$ will be skew-symmetric, meaning that $\f{g}_{n_1}\cdot A_{n_2}=0$ whenever $\{n_1,n_2\}=0$.  Note that these skew-symmetry conditions imply that $\f{g}_n^{\parallel}=\f{g}_{n,\{n,\cdot\}}^{\parallel}$, and that Condition \eqref{gnm} also follows.

For simplicity, we also assume in this example that either $\{\cdot,\cdot\}$ is $\bb{Z}$-valued, or that $R$ contains a copy of $\bb{Q}$ (in which case we identify $\Theta_K(R[N^{\oplus}])$ with $R[N^{\oplus}]\otimes_{\bb{Q}} \?{M}_{\bb{Q}}$).  We define a map $\omega_1:N\rar M_{\bb{Q}}$ via $\omega_1(n)=\{n,\cdot\}$.  A natural choice for $K$ in this and the next example is $K:=\ker(\omega_1)$.

Now, let $\f{h}$ be the Lie algebra of the previous example.  The elements of the form
\begin{align}\label{PoissonTorusElements}
    z^{n} \partial_{\omega_1(n)}
\end{align} generate a Lie subalgebra $\f{g}^{\omega}\subset \f{h}$ which we take as our $\f{g}$.  We denote the corresponding prounipotent  Lie group $\wh{G}$ by $\wh{G}^{\omega}$.  $A$ and $\wh{A}$ are as before, and the action of $\f{g}^{\omega}$ on them is via restriction from that of $\f{h}$.

With this setup, $\wh{A}=R\llb N^{\oplus} \rrb_P$ forms a Poisson algebra with Poisson bracket defined by \begin{align}\label{PoissonBracket}
[z^{p_1},z^{p_2}]:=\{p_1,p_2\}z^{p_1+p_2}.
\end{align}
Then $\iota:z^{e_i}\partial_{\omega_1(n)} \mapsto z^{e_i}$ identifies $\f{g}^{\omega}$ (respectively, $\wh{\f{g}}^{\omega}$) with the $R$-span (respectively, the topological $R$-span\footnote{By the topological $R$-span of a set $S$ in the $N^+$-adic completion $\wh{A}$ of $A$, we mean the set of all possibly-infinite sums of elements in $S$ with coefficients in $R$ such that, for each $k>0$, all but finitely many terms vanish modulo $A^{\geq k}$.}) of the elements $z^n\in R\llb N^{\oplus} \rrb_P$ with $n\in N^+$, the Lie bracket being identified with the Poisson bracket.  The action of $\f{g}^{\omega}$ on $A$ is then just the restriction to $\f{g}^{\omega}$ of the adjoint action of $A$ on itself (with the Poisson bracket as the Lie bracket), and similarly for the action of $\wh{\f{g}}^{\omega}$ on $\wh{A}$.

\item {\bf Quantum torus algebras:} The previous example admits a quantization (important for quantum cluster algebras) as follows: Fix some $D\in \bb{Q}_{>0}$ such that $D\{\cdot,\cdot\}$ is $\bb{Z}$-valued.  For any $a \in \frac{1}{D}\bb{Z}_{\geq 0}$, we have a corresponding ``quantum number''
\[[a]_q:=q^a-q^{-a}\in R[q^{\pm 1/D}].
\]
Note that $\lim_{q^{1/D}\rar 1} \frac{[a]_q}{q-q^{-1}}=a$. 
Define $R_q\subset R(q^{1/D})$ by adjoining $[a]_q^{-1}$ to $R[q^{\pm 1/D}]$ for each $a\in \frac{1}{D}\bb{Z}_{> 0}$.

Now, let $A=R_q[P]$ be the quantum torus algebra:
\[R_q[P]:=R_q[z^p|p\in P]/\langle z^{p_1}z^{p_2}=q^{\{p_1,p_2\}}z^{p_1+p_2}\rangle.\] 
The $N^+$-adic completion is $\wh{A}=:R_q\llb N^{\oplus}\rrb_P$. 

$R_q[P]$ forms a Lie algebra under the usual commutator, which one easily checks is given by
\begin{align*}
 [z^{p_1},z^{p_2}]=[\{p_1,p_2\}]_q z^{p_1+p_2}.
\end{align*}
We take $\f{g}=\f{g}^{\omega}_q$ to be the sub-Lie algebra $R_q[N^+]$, spanned over $R_q$ by $z^n$ with $n\in N^+$.  The action of $\f{g}^{\omega}_q$ on $A$ is just the restriction of the adjoint action.  One checks that this specializes\footnote{Technically, making $\f{g}_q^{\omega}\rar \f{g}^{\omega}$ into a well-defined Lie algebra homomorphism requires more care with the coefficients in $\f{g}_q^{\omega}$.  In \S \ref{Application} we will use the classical limit map for $A$, but not for $\f{g}$, so we do not take the time to make this precise.} to the previous example in the $q^{1/D}\mapsto 1$ limit, taking $z^p\mapsto z^p$ for $A$ and $\frac{z^n}{q-q^{-1}}\mapsto z^n$ for $\f{g}$.

The $N^+$-adic completion of $\f{g}$ is $\wh{g}^{\omega}_q = R_q\llb N^{\oplus}\rrb$, and exponentiation yields $\wh{G}:=\wh{G}^{\omega}_q$ in the multiplicative group of $R_q\llb N^{\oplus}\rrb\subset \wh{A}$.  The action of this $\wh{G}$ on $\wh{A}$ is then via conjugation, $g\cdot a = gag^{-1}$.

\item {\bf Hall algebras:} The Hall algebra scattering diagrams of \cite{Bridge} provide additional interesting examples which further refine the Poisson and quantum torus algebra examples above.  However, the Hall algebra does not satisfy the condition of \eqref{gnm}.  To apply the results of this paper then, including the crucial refined \cite{CPS} result (Theorem \ref{CPS}), one must mod out by an ideal in order to obtain a skew-symmetric Lie algebra.\footnote{Actually, our proof of Theorem \ref{MainTheorem} does not use \eqref{gnm} and so applies more generally, but without Theorem \ref{CPS}, theta functions become less meaningful.}   This setup will not be further discussed here, but it is investigated in \cite{CM}.
\end{enumerate}
\end{egs}

\begin{dfn}\label{WallDef}
For the above data, a wall in $N_{\bb{R}}$ over $\f{g}$ is a triple $(m_{\f{d}},\f{d},g_{\f{d}})$ such that:
\begin{itemize}
\item $m_{\f{d}}$ is an element of $\?{M}$ (which we recall is identified with $K^{\perp}\cap M$), determined up to positive scaling (we could require $m_{\f{d}}$ to be primitive, but it will often be convenient to allow it to be non-primitive).
\item $\f{d}$ is a closed, convex (but not necessarily strictly convex), rational-polyhedral, codimension-one affine cone in $N_{\bb{R}}$ which is parallel to $m_{\f{d}}^{\perp}$.  This is called the support of the wall.
\item $g_{\f{d}} \in \f{g}_{n_{\f{d}},m_{\f{d}}}^{\parallel}$ for some primitive $n_{\f{d}}\in m_{\f{d}}^{\perp}\cap N$.  $-n_{\f{d}}$ is called the  direction of the wall.
\end{itemize}
A scattering diagram $\f{D}$ over $\f{g}$ is a set of walls over $\f{g}$ such that for each $k >0$, there are only finitely many $(m_{\f{d}},\f{d},g_{\f{d}})\in \f{D}$ with $g_{\f{d}}$ not projecting to $0$ in $\f{g}_k$. 
\end{dfn}
A wall with direction $-n_{\f{d}}$ is called incoming if it is closed under addition by $n_{\f{d}}$.  Otherwise, the wall is called outgoing.  Note that, given $\f{d}$, the additional data of $m_{\f{d}}$ is equivalent to choosing a side of $\f{d}$ to be the positive side of the wall (i.e., the side where $m_{\f{d}}$ is positive).

We will sometimes denote a wall $(m_{\f{d}},\f{d},g_{\f{d}})$ by just $\f{d}$.  Denote $\Supp(\f{D}):= \bigcup_{\f{d}\in \f{D}} \f{d}$, and \begin{align*}
\Joints(\f{D}):= \bigcup_{\f{d}\in \f{D}} \partial \f{d} \cup \bigcup_{\substack{\f{d}_1,\f{d}_2\in \f{D}\\
		               \dim \f{d}_1\cap \f{d}_2 = r-2}} \f{d}_1\cap \f{d}_2. 
\end{align*}

\begin{rmk}[Other conventions]\label{RmkScatBar}
We briefly discuss how our definition of a scattering diagram relates to other definitions which have appeared in the literature.
\begin{enumerate}[label=(\roman*)]
    \item In practice, walls of scattering diagrams are closed under addition by $K_{\bb{R}}$.  Thus, it is reasonable (though more notationally cumbersome for our purposes here) to view the scattering diagram as living in $\?{N}_{\bb{R}}$, replacing each $\f{d}$ above with $\pi_K(\f{d})$ and viewing $m_{\f{d}}^{\perp}$ as living in $\?{N}_{\bb{R}}$ instead of $N_{\bb{R}}$.  This is essentially the approach implicitly used in \cite{GPS} and \cite{GHK1}. 
    The modifications for this viewpoint are fairly straightforward: The direction of a wall is then $-\pi_K(n_{\f{d}})$ instead of $-n_{\f{d}}$, and incoming walls are then closed under addition by $\pi_K(n_{\f{d}})$.  In the definition of broken lines in Def. \ref{broken line} below, the only modification is that $Q$ should live in $\?{N}_{\bb{R}}$ instead of $N_{\bb{R}}$, and $\gamma'(t)$ should be $-\pi_K(v_i)$ in place of $-v_i$.  Similarly, when using this viewpoint, our counts of tropical curves and tropical disks in $N_{\bb{R}}$ can be replaced with the analogous counts in $\?{N}_{\bb{R}}$ obtained by applying $\pi_K$ to each value of the tropical degree and to each incidence condition.
    \item In some setups, e.g., the Hall algebra setup of \cite{Bridge}, it is more natural to view the walls of the scattering diagram as living in $M_{\bb{R}}$, with $\f{d}$ being parallel to $n_{\f{d}}^{\perp}$.  These cases come with a skew-symmetric form $\{\cdot,\cdot\}$ on $N$ and a map $\omega_1:N\rar M$ as mentioned in Example \ref{gegs}(ii), and broken lines have $\gamma'(t)=-\omega_1(v_i)$ in place of $-v_i$.  These scattering diagrams in $M_{\bb{R}}$ yield scattering diagrams in $N_{\bb{R}}$ as in our setup by taking $\omega_1^{-1}$ of the supports of the walls.  If $\f{g}$ is skew-symmetric with respect to $\{\cdot,\cdot\}$ and we take $K=\ker \omega_1$, then $\omega_1(N_{\bb{R}})$ is identified with $\?{N}_{\bb{R}}$, and so intersecting the walls in $M_{\bb{R}}$ with $\omega_1(N_{\bb{R}})$ recovers the viewpoint of (i) above.
\end{enumerate}
\end{rmk}

Note that for each $k>0$, a scattering diagram $\f{D}$ over $\f{g}$ induces a finite scattering diagram $\f{D}^k$ over $\f{g}_k$ with walls corresponding to the $\f{d}\in \f{D}$ for which $g_{\f{d}}$ is nontrivial in $\f{g}_k$.

Consider a smooth immersion $\gamma:[0,1]\rar N_{\bb{R}}\setminus \Joints(\f{D})$ with endpoints not in $\Supp(\f{D})$ which is transverse to each wall of $\f{D}$ it crosses.  Let $(m_{\f{d}_i},\f{d}_i,g_{\f{d}_i})$, $i=1,\ldots, s$, denote the walls of $\f{D}^{k}$ crossed by $\gamma$, and say they are crossed at times $0<t_1\leq \ldots \leq t_s<1$, respectively (if $t_i=t_{i+1}$, then the requirement that each $\f{g}_{\f{d}}$ is in $\f{g}_{n_{\f{d}},m_{\f{d}}}^{\parallel}$ implies that the ordering of these two walls does not affect \eqref{WallCross} and therefore does not matter).  Define 
\begin{align}\label{WallCross}
\theta_{\f{d}_i}:=\exp(g_{\f{d}_i})^{\sign \langle -\gamma'(t_i),m_{\f{d}_i}\rangle} \in G_k.
\end{align}
Let $\theta_{\gamma,\f{D}}^k:=\theta_{\f{d}_s} \cdots \theta_{\f{d}_1}\in G_k$, and define the path-ordered product:
\begin{align*}
\theta_{\gamma,\f{D}}:= \varprojlim_k \theta_{\gamma,\f{D}}^k \in \wh{G}.
\end{align*}

\begin{dfn}
Two scattering diagrams $\f{D}$ and $\f{D}'$ are  equivalent if $\theta_{\gamma,\f{D}} = \theta_{\gamma,\f{D}'}$ for each smooth immersion $\gamma$ as above.  $\f{D}$ is consistent if each $\theta_{\gamma,\f{D}}$ depends only on the endpoints of $\gamma$.
\end{dfn}

\begin{egs}\label{equivD}
~

\begin{enumerate}[label=(\roman*)]
\item Replacing a wall $(m_{\f{d}},\f{d},g_{\f{d}})\in \f{D}$ with the wall $(-m_{\f{d}},\f{d},-g_{\f{d}})$ produces an equivalent scattering diagram.
\item Consider a collection of walls $\{(m_{\f{d}},\f{d},g_{\f{d}_i}\in \f{g}_{n_{\f{d}},m_{\f{d}}}^{\parallel})\in \f{D}| i\in S\}$, where $S$ is some countable index set and $n_{\f{d}}$, $m_{\f{d}}$, and $\f{d}$ are independent of $i$.  Replacing this collection of walls with a single wall $(m_{\f{d}},\f{d},\sum_{i\in S} g_{\f{d}_i})$ produces an equivalent scattering diagram.
\item Replacing a wall $(m_{\f{d}},\f{d},g_{\f{d}})\in \f{D}$ with a pair of walls $(m_{\f{d}},\f{d}_i,g_{\f{d}})$, $i=1,2$, such that $\f{d}_1\cup \f{d}_2=\f{d}$ and $\codim(\f{d}_1\cap \f{d}_2)=2$ produces an equivalent scattering diagram.
\end{enumerate}
\end{egs}

The following theorem on scattering diagrams is fundamental to the theory.  The two-dimensional tropical vertex group cases were first proved in \cite{KS}.  The tropical vertex cases for higher-dimensional spaces (including more general affine manifolds than just $N_{\bb{R}}$) were proved in \cite{GS11}, and the result for more general $\f{g}$ follows from \cite[Thm. 2.1.6]{WCS} (cf. \cite[Thm. 1.21]{GHKK}).  Alternatively, we note that the existence part of the result follows from the construction of $\f{D}_k^{\infty}$ in \S \ref{Dkinfty} (a generalization of the construction from \cite[\S 1.4]{GPS}), while a separate uniqueness argument is given in \cite[\S 3.1]{CM}.

\begin{thm}\label{KSGS}
Let $\f{g}$ be an $N^+$-graded Lie algebra, and let $\f{D}_{\In}$ be a finite scattering diagram over $\f{g}$ whose only walls have full affine hyperplanes as their supports.  Then there is a unique-up-to-equivalence scattering diagram $\f{D}$, also denoted $\scat(\f{D}_{\In})$, such that $\f{D}$ is consistent, $\f{D} \supset \f{D}_{\In}$, and $\f{D}\setminus \f{D}_{\In}$ consists only of outgoing walls.
\end{thm}

We next give several important examples of initial scattering diagrams $\f{D}_{\In}$.  For a more specific example of a possible $\f{D}_{\In}$ and the corresponding $\scat(\f{D}_{\In})$, cf. Example \ref{A2Example}.

\begin{egs}\label{InExamples}
We present some important examples of initial scattering diagrams which will be examined more in \S \ref{ClusterSection}.  These examples build off those of Examples \ref{gegs}.  First though, we fix some additional data:

We fix a multiset (i.e., a set possibly with repetition) $E:=\{e_i\}_{i\in I}$ of vectors in $N$, indexed over a finite set $I$.  Let $F$ be a subset of $I$ such that $E_{I\setminus F}:=\{e_i\}_{i\in I\setminus F} \subset N^{\oplus}$ (typically, one would be given $N$, $E$, $I$, and $F$, and would then choose $N^{\oplus}\subset N$ to contain $E$).

For the skew-symmetric examples, we also we fix numbers $\{d_i\in \bb{Q}_{>0}\}_{i\in I}$ and define a bilinear form $(\cdot,\cdot)$ on $N$ satisfying 
\begin{align*}
    (e_i,e_j)=d_j\{e_i,e_j\}.
\end{align*}
We require that $(n_1,n_2)\in \bb{Z}$ whenever $n_1,n_2\in N$ with at least one of $n_1$ or $n_2$ being in $N^{\oplus}$.

The form $(\cdot,\cdot)$ determines maps $\pi_1,\pi_2:N^{\oplus}\rar M$, $\pi_1(n):=(n,\cdot)$ and $\pi_2(n):=(\cdot,n)$.  In all our skew-symmetric examples, we will have $\f{g}_n^{\parallel}=\f{g}_{n,\pi_1(n)}^{\parallel}$.  A natural choice for $K$ is $K:=\ker(\pi_1)$, which if $E$ spans $N_{\bb{Q}}$ is the same as $\ker(\omega_1)$.

\begin{enumerate}[label=(\roman*)]\setlength\itemsep{1em}
    \item Take $\f{g}=\f{h}$ as in Example \ref{gegs}(i).  In addition to $E$, suppose we are given a multiset $U=\{u_i\}_{i\in I\setminus F}$, this time with vectors $u_i\in \?{M}\setminus \{0\}$, such that $\langle e_i,u_i\rangle=0$ for each $i\in I\setminus F$. 
    Then we take the initial scattering diagram to be   
    \begin{align*}
    \f{D}_{\In}:=\{(u_i, u_i^{\perp},\log(1+z^{e_i})\partial_{u_i})|i\in I\setminus F\}.
\end{align*}
The wall-crossing automorphism for crossing from the side of $(u_i, u_i^{\perp},\log(1+z^{e_i})\partial_{u_i})$ containing some $p\in P$ to the other side then acts by 
\begin{align}\label{TropicalVertexWallCross}
    z^p\mapsto z^p(1+z^{e_i})^{|\langle p,u_i\rangle|}.
\end{align}  Such walls are commonly (e.g., in \cite{GPS} and \cite{GHKK}) denoted as simply $(u_i^{\perp},(1+z^{e_i})^{|u_i|})$.  

\item Now take $\f{g}=\f{g}^{\omega}\subset \f{h}$ as in Example \ref{gegs}(ii).  We take $\f{D}_{\In}$ to be the special case of $\f{D}_{\In}$ from the previous example in which $u_i$ is taken to be $-\pi_2(e_i) = d_i\omega_1(e_i)$ for each $i\in I\setminus F$.

Using the embedding $\iota:z^n\partial_{\omega_1(n)}\mapsto z^n$ of $\wh{\f{g}}^{\omega}$ into the Poisson algebra $\wh{A}=R\llb N^{\oplus} \rrb_P$, the initial scattering functions $\log(1+z^{e_i})\partial_{d_i\omega_1(e_i)}$ become dilogarithms:
\begin{align}\label{Li2}
    \iota\left(\log(1+z^{e_i})\partial_{d_i\omega_1(e_i)}\right) &= \iota\left(\sum_{k=1}^{\infty} (-1)^{k+1}\frac{1}{k^2} d_iz^{ke_i}\partial_{\omega_1(ke_i)} \right) \nonumber\\
    &= d_i\sum_{k=1}^{\infty} (-1)^{k+1}\frac{z^{ke_i}}{k^2} \nonumber\\
    &= -d_i\Li_2(-z^{e_i}),
\end{align}
where $\Li_2$ is the dilogarithm function defined by
\[\Li_2(x):=\sum_{k= 1}^{\infty} \frac{x^k}{k^2}.
\]
Thus, we can write the initial scattering diagram as
\begin{align}\label{DInPoisson}
    \f{D}_{\In} = \{(\omega_1(e_i),\omega_1(e_i)^{\perp},-d_i\Li_2(-z^{e_i}))|i\in I\setminus F\}.
\end{align}

\item Consider the quantization $\f{g}=\f{g}^{\omega}_q$ as in Example \ref{gegs}(iii), that is, $\wh{\f{g}}=R_q\llb N^+ \rrb \subset R_q \llb N^{\oplus} \rrb_P=\wh{A}$.  Similarly to in the previous example, we take the initial scattering diagram to be
\begin{align}\label{qIn}
\f{D}_{\In}:=\{(\omega_1(e_i), \omega_1(e_i)^{\perp},-\Li_2(-z^{e_i};q^{1/d_i})\},
\end{align}
where the scattering functions are now defined in terms of quantum dilogarithms:
\begin{align*}
    \Li_2(x;q):=\sum_{k=1}^{\infty} \frac{x^k}{k[k]_{q}}.
\end{align*}
Here, we use our notation $[k]_q=q^k-q^{-k}$, so $[k]_{q^{1/d_i}}=[k/d_i]_q$.  Note that the $q\mapsto 1$ limit of $\Li_2(x;q)$ is $\Li_2(x)$ (with $\frac{x}{q-q^{-1}}$ mapping to $x$), so this $\f{D}_{\In}$ does indeed specialize to the one from the previous example in the $q^{1/D}\mapsto 1$ limit (with $\frac{z^n}{q-q^{-1}}$ mapping to $z^n$).  Let 
\begin{align*}
    \Psi_{q^{1/d_i}}(z^{e_i}):=\exp(-\Li_2(-z^{e_i};q^{1/d_i}))=\prod_{k=1}^{\infty} \frac{1}{1+q^{(2k-1)/d_i}z^{e_i}} \in \wh{G}.
\end{align*}
Then for any $p\in P$, crossing a wall as above from the side containing $p$ to the other side acts on $z^p$ via
\begin{align}\label{qWallCross}
    \Psi_{q^{1/d_i}}(z^{e_i})^{\sign \{e_i,p\}} \cdot z^p &= \Psi_{q^{1/d_i}}(z^{e_i})^{\sign \{e_i,p\}} z^p \Psi_{q^{1/d_i}}(z^{e_i})^{-\sign \{e_i,p\}} \nonumber\\
    &= z^p \prod_{k=1}^{d_i|\{e_i,p\}|} (1+q^{\sign(\{e_i,p\})(2k-1)/d_i}z^{e_i}).
\end{align}
\end{enumerate}
\end{egs}

Given an $N^+$-graded Lie algebra $\f{g}$ as above and any commutative, associative algebra $T$, we can obtain another $N^+$-graded Lie algebra $\f{g}\otimes T$ with bracket defined by $[g_1\otimes t_1,g_2\otimes t_2]:=[g_1,g_2]_{\f{g}}\otimes (t_1t_2)$ (when it is possibly not clear from context, we will use subscripts after brackets to indicate the Lie algebra in which the bracket is performed).  We will denote elements $g\otimes t$ as simply $tg$.  We denote $N^+$-adic completion of $\f{g}\otimes T$ by $\f{g}\wh{\otimes} T$, and we similarly denote the corresponding Lie group as $G\wh{\otimes} T$.  These act on the algebra $A\wh{\otimes} T$ obtained by taking the $N^+$-adic completion of $A\otimes T$.  Here, the action of $\f{g}\wh{\otimes} T$ on $A\wh{\otimes} T$ is given by $(tg)\cdot a=(g\cdot a)\otimes t$, also denoted $t(g\cdot a)$.  We will often use this construction to adjoin nilpotent elements.  The following lemma is straightforward.

\begin{lem}\label{SimpleBreak}
Let $T$ be a commutative, associative algebra with $t\in T$, $t^2=0$.  Let $g\in \wh{\f{g}}$, $a\in \wh{A}$.  Then
\begin{align*}
\exp(tg)\cdot a = a+t(g\cdot a).
\end{align*}
Here, $\cdot$ on the left-hand side is the action of $G\wh{\otimes} T$ on $A\wh{\otimes} T$, while $\cdot$ on the right-hand side is the action of $\f{g}$ on $\wh{A}$.
\end{lem}

In \S \ref{ScatterFactor}, the construction of $\scat(\f{D}_{\In})$ from $\f{D}_{\In}$ will depend on repeatedly applying the following computation:\footnote{Lemma \ref{PentagonScatter} in the cases where $\dim N=2$ and $\f{g}=\f{h}$ is essentially \cite[Lemma 1.9]{GPS}.  In the cases with $\dim N=2$ and $\f{g}=\f{g}^{\omega}_q$, it is \cite[Lemma 4.3]{FS}.}

\begin{lem}\label{PentagonScatter}
Suppose we have an $N^+$-graded Lie algebra $\f{g}$ and a commutative associative algebra $T$ with $t_1,t_2\in T$, $t_1^2=t_2^2=0$.  Fix $n_1,n_2\in N^+$, and fix primitive $m_1,m_2\in \?{M}$ such that $\langle n_i,m_i\rangle =0$ for $i=1,2$.  Also, fix some $g_i\in \f{g}_{n_i}$ for $i=1,2$.  Let 
\begin{align*}
    \f{D}_{\In}:= \{(m_1,m_1^{\perp},t_1g_1),(m_2,m_2^{\perp},t_2g_2)\}
\end{align*}
be a scattering diagram over $\f{g}\otimes T$.  Then $\scat(\f{D}_{\In})=\f{D}_{\In} \cup \{(m_3,\f{d}_3,g_3)\}$, where
\begin{align*}
    m_3&:=\mu((n_1,m_1),(n_2,m_2)),\\
    \f{d}_3&:=(m_1^{\perp}\cap m_2^{\perp})+\bb{R}_{\leq 0}(n_1+n_2), \mbox{~~~~~and}\\
    g_3&:=t_1t_2[g_1,g_2]_{\wh{\f{g}}}.
\end{align*}
\end{lem}

\begin{proof} First, recall from \eqref{m} that $\mu((n_1,m_1),(n_2,m_2)):=\langle n_2,m_1\rangle m_2-\langle n_1,m_2\rangle m_1$.  One easily checks now that $(m_1\cap m_2)^{\perp}\subset m_3^{\perp}$ and $n_1+n_2\in m_3^{\perp}$, so  $$m_3^{\perp}\supset (m_1^{\perp}\cap m_2^{\perp})+\bb{R}(n_1+n_2).$$ Hence, $m_3^{\perp}$ does contain $\f{d}_3$.  

~

\noindent\begin{minipage}{0.58\textwidth}
\setlength{\parindent}{15pt}
 Now, let $\gamma$ be a path as in the figure to the right, going from the region with $m_1,m_2<0$ to the region with $m_2>0$ and $m_1< 0$, then to $m_1,m_2> 0$, then to $m_2< 0$ and $m_1> 0$, and then back to $m_1,m_2 < 0$.  Then 
\begin{align*}
    \theta_{\gamma}&=\exp(t_1g_1)\exp(t_2g_2)\exp(-t_1g_1)\exp(-t_2g_2) \\
    &= [\exp(t_1g_1),\exp(t_2g_2)]_{G\wh{\otimes}T},
\end{align*}
where $[a,b]_{G\wh{\otimes}T}:=aba^{-1}b^{-1}$ for any $a,b\in G\wh{\otimes} T$. 
\end{minipage}
\hfill
\begin{minipage}{.35\textwidth}
\def\svgwidth{150pt}
    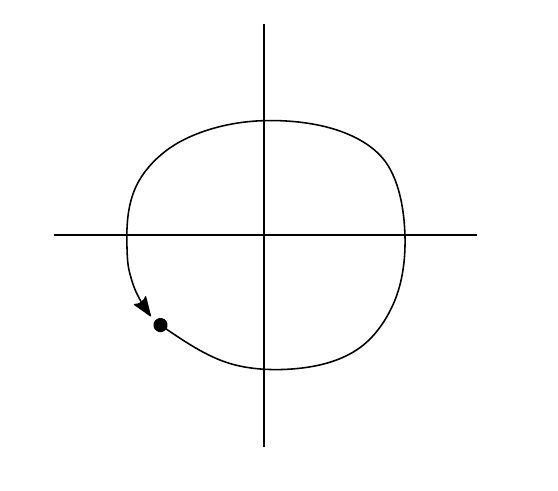\label{m1m2}
\end{minipage}

~

We claim that 
\begin{align}
[\exp(t_1g_1),\exp(t_2g_2)]_{G\wh{\otimes} T} = \exp([t_1g_1,t_2g_2]_{\f{g}\wh{\otimes} T}).
\end{align}
Indeed, the Baker-Campbell-Hausdorff formula tells us that for any $x,y\in \wh{\f{g}}$, we have \[\log(\exp(t_1x)\exp(t_2y))=t_1x+t_2y+\frac{1}{2}t_1t_2[x,y],\] and using this, we compute:
\begin{align*}
    \log([\exp(t_1g_1),\exp(t_2g_2)]) &= \log(\exp(t_1 g_1)\exp(t_2g_2)\exp(-t_1 g_1) \exp(-t_2 g_2)) && \\
                                &= \log(\exp(\log(\exp(t_1 g_1)\exp(t_2g_2)))\exp(\log(\exp(-t_1 g_1) \exp(-t_2 g_2))) && \\
                                &= \log(\exp(t_1g_1+t_2g_2+\frac{1}{2} t_1t_2 [g_1,g_2]) \exp(-t_1 g_1 -t_2 g_2 +\frac{1}{2} t_1t_2 [g_1,g_2])) && \\
                                &= (t_1g_1+t_2g_2+\frac{1}{2} t_1t_2 [g_1,g_2]) + (-t_1 g_1 -t_2 g_2 +\frac{1}{2} t_1t_2 [g_1,g_2]) && \\
                                 &~~ + \frac{1}{2} [t_1g_1+t_2g_2+\frac{1}{2} t_1t_2 [g_1,g_2],-t_1 g_1 -t_2 g_2 +\frac{1}{2} t_1t_2 [g_1,g_2]] && \\
                                 &= [t_1g_1,t_2g_2]. &&
\end{align*}
Thus, $\theta_{\gamma}=\exp(t_1t_2[g_1,g_2])=\exp(g_3)$.  Since $g_3=t_1t_2[g_1,g_2]$ is in $\f{g}_{n_1+n_2}$ and commutes with both $t_1g_1$ and $t_2g_2$, we just have to check that crossing $\f{d}_3$ along $\gamma$ induces the scattering automorphism $g_3^{-1}$.  That is, we just have to check that $\langle -\gamma'(t),m_3\rangle<0$, where $t$ is the time at which $\gamma$ passes $\f{d}_3$.

Suppose $\langle n_1,m_2\rangle \geq 0$ and $\langle n_2,m_1\rangle \geq 0$.  Then $\langle n_1,m_3\rangle 
 \geq 0$,  
 and when $\gamma$ passes through $\f{d}_3$, it comes from the side of $\f{d}_3$ which contains $-n_1$.  Hence, $\langle -\gamma'(t),m_3\rangle \leq 0$, as desired.  The cases where one or both of $\langle n_1,m_2\rangle$ and $\langle n_2,m_1\rangle$ are negative are similarly checked.

\end{proof}

\subsection{Broken lines and theta functions}\label{Broken Lines}

Fix a consistent scattering diagram $\f{D}$ over $\f{g}$, with $\wh{\f{g}}$ acting on $\wh{A}$ as in \S \ref{scatter}.  Recall that for each $p\in \?{P}$, we have designated an element $z^{\varphi(p)}\in A_{\varphi(p)}$.

\begin{dfn}\label{broken line}
Let $p \in \?{P}\setminus \{0\}$, $Q\in N_{\bb{R}}\setminus \Supp(\f{D})$.  A  broken line $\gamma$ with  ends $(p,Q)$ is the data of a continuous map $\gamma:(-\infty,0]\rar N_{\bb{R}}\setminus \Joints(\f{D})$, values $-\infty < t_0 \leq t_1 \leq \ldots \leq t_{\ell} = 0$, and for each $i=0,\ldots,\ell$, an associated homogeneous element $a_i \in A_{v_i}$ for some $v_{i} \in P\setminus \{0\}$, such that:
\begin{enumerate}[label=(\roman*), noitemsep]
\item $\gamma(0)=Q$.
\item For $i=1\ldots, \ell$, $\gamma'(t)=-v_i$ for all $t\in (t_{i-1},t_{i})$.  Similarly, $\gamma'(t)=-v_0$ for all $t\in (-\infty,t_0)$. 
\item $a_0=z^{\varphi(p)}$.
\item For $i=0,\ldots,\ell-1$, $\gamma(t_i)\in \Supp(\f{D})$.  Let 
\begin{align*}
g_i:=\prod_{\substack{(m_{\f{d}},\f{d},g_{\f{d}})\in \f{D} \\ \f{d}\ni \gamma(t_i)}} \exp(g_{\f{d}})^{\sign(\langle v_i,m_{\f{d}} \rangle)} \in \wh{G}.
\end{align*}
I.e., $g_i$ is the $\epsilon\rar 0$ limit of the wall-crossing automorphism $\theta_{\gamma|_{(t_i-\epsilon,t_i+\epsilon)}}$ defined in \eqref{WallCross} (using a smoothing of $\gamma$). Then $a_{i+1}$ is a homogeneous term of $g_i\cdot a_i$.
\end{enumerate}
  We will call $v_{i+1}-v_i\in N^{\oplus}$ a bend of  $\gamma$.  We assume all bends are nonzero, so we cannot get new broken lines by just inserting new values of $t$ as trivial bends.  A straight broken line is a broken line with no bends.  By the type of a broken line $\gamma$ as above, we mean the data of the elements $a_i\in A_{v_i}$, $i=0,\ldots,\ell$.
\end{dfn}

Fix a generic point $Q \in N_{\bb{R}}\setminus \Supp(\f{D})$.  For any $p\in \?{P} \setminus \{0\}$, we define a theta function
\begin{align}\label{vartheta-dfn}
\vartheta_{p,Q}:=\sum_{\Ends(\gamma)=(p,Q)} a_{\gamma} \in \wh{A}.
\end{align}
Here, the sum is over all broken lines $\gamma$ with ends $(p,Q)$, and $a_{\gamma}$ denotes the homogeneous element of $\wh{A}$ attached to the final straight segment of $\gamma$.  That this is well-defined will be proven shortly.  For the case $p=0$, we define $\vartheta_{0,Q}=1$.

\begin{eg}\label{A2Example}
\begin{figure}[htb]
\def\svgwidth{250pt}
    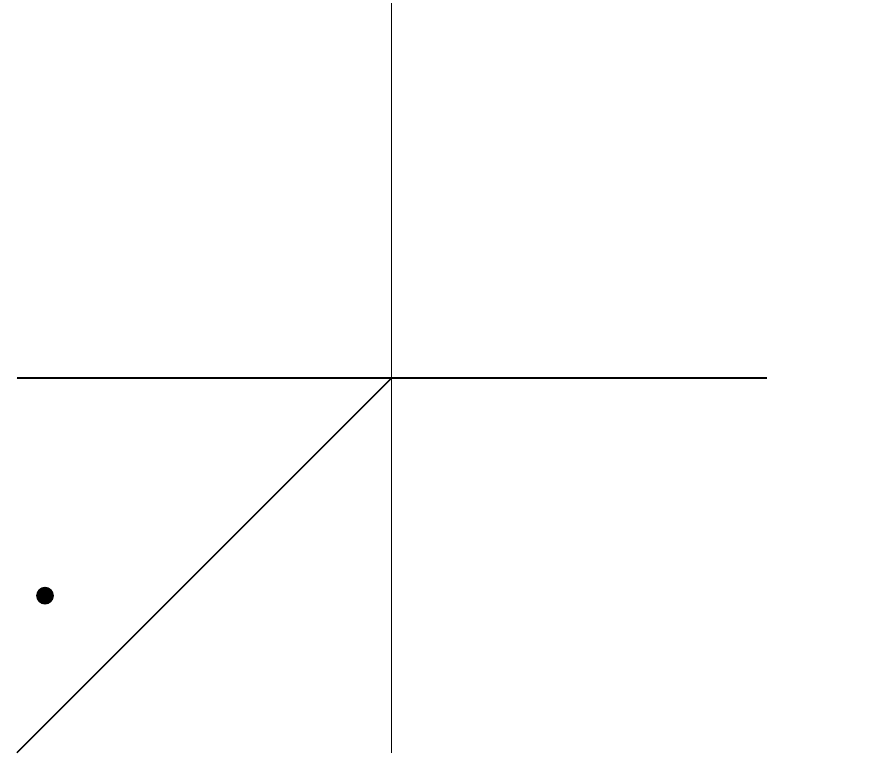
    \caption{\label{A2}}
\end{figure}

Let $N=\bb{Z}^2$, equipped with the standard skew-symmetric form, and consider the quantum torus algebra setup as in Example \ref{gegs}(iii).  Consider the scattering diagram $\f{D}_{\In}$ with walls $\f{d}_1:=(e_2^*,(e_2^*)^{\perp},-\Li_2(-z^{e_1};q))$ and  $\f{d}_2:=(-e_1^*,(-e_1^*)^{\perp},-\Li_2(-z^{e_2};q))$.  Then $\f{D}:=\scat(\f{D}_{\In})$ consists of one additional wall $\f{d}_3:=(e_2^*-e_1^*,(e_2^*-e_1^*)^{\perp}\cap \bb{R}_{\leq 0}^2,-\Li_2(-(q-q^{-1})z^{e_1+e_2};q))$.  The supports of these walls are illustrated in Figure \ref{A2} as solid lines.  The consistency can be written as the expression $\Psi_q(z^{e_1})\Psi_q(z^{e_2}) = \Psi_q(z^{e_2}) \Psi_q((q-q^{-1})z^{e_1+e_2})\Psi_q(z^{e_1})$ (the two sides of this equation corresponding to the two paths from the bottom-right quadrant to the top-left), which is a modified version of the quantum pentagon identity of \cite{FK}.  The dashed lines in Figure \ref{A2} are the broken lines for $\vartheta_{e_1,Q}$.  See \eqref{qWallCross} for the formula used for computing the wall-crossings.  From bottom to top, the final monomials attached to these broken lines are $z^{e_1}$, $(q-q^{-1})^2z^{2e_1+e_2}$, and $(q-q^{-1})z^{e_1+e_2}$, so $\vartheta_{e_1,Q}$ is the sum of these three terms.
\end{eg}

We will now prove several facts about these theta functions, beginning with showing that they are well-defined.  Given $n\in N^{\oplus}$, let 
\begin{align}\label{dn}
d(n) \in \bb{Z}_{\geq 0}
\end{align} denote the largest number $k$ such that $n\in kN^+$ (as defined in \eqref{kNplus}), taking $d(0)$ to be $0$.  Note that $d(n_1+n_2) \geq d(n_1)+d(n_2)$ for all $n_1,n_2\in N^{\oplus}$.  Now, note that for $a\in \wh{A}_p$ and $g\in \f{g}_n$,
\begin{align}\label{action-grading}
    \exp(g)\cdot a \in a+\bigoplus_{k\in \bb{Z}_{>0}} \wh{A}_{p+kn}.
\end{align} 
That is, $\exp(g)\cdot a$ is equal to $a$ plus terms of degree equal to $p$ plus a positive multiple of $n$.  Hence, for any broken line $\gamma$, we always have $d(v_{i+1})>d(v_i)$ (notation as in Definition \ref{broken line}).  That is, bends always increase $d$ of the degree of the elements attached to the straight segments of $\gamma$, so a broken line $\gamma$ with ends $(p,Q)$ contributing $a_{\gamma}\in A_{v_{\gamma}}$ to \eqref{vartheta-dfn} has at most $d(v_{\gamma}-\varphi(p))$ bends.  Recall from Definition \ref{WallDef} the requirement that for each $k>0$, $g_{\f{d}}$ projects to $0$ in $\f{g}_k$ for all but finitely many walls $\f{d}\in \f{D}$.  It now follows that for each $k>0$, there are only finitely many broken lines $\gamma$ with $\Ends(\gamma)=(p,Q)$ such that the projection of $a_{\gamma}$ to $\f{g}_k$ is non-trivial.  Hence, \eqref{vartheta-dfn} is indeed well-defined.

Furthermore, since bends shift the degree of the attached element by an element of $N^+$, we see that the term in \eqref{vartheta-dfn} of minimal degree is the one associated to the unbroken line, i.e., $z^{\varphi(p)}$.  That is,
\begin{align}\label{zp}
\vartheta_{p,Q}\in z^{\varphi(p)}+\wh{A}_{\varphi(p)+N^+},
\end{align}
where $\wh{A}_{\varphi(p)+N^+}$ is the ideal of $\wh{A}$ consisting of the topological span of terms with grading equal to $\varphi(p)+n$ for some $n\in N^+$.  Let $$\?{P}^{\circ}\subset \?{P}$$ be the subset consisting of the elements $p$ such that $a z^{\varphi(p)} \neq 0$ for any nonzero $a\in \wh{A}_K$.  It follows from \eqref{zp} that the set $\{\vartheta_{p,Q}\in \wh{A}|p\in \?{P}^{\circ}\}$ (with fixed $Q$) is linearly independent over $\wh{A}_K$.

Recall that $P=\varphi(\?{P})+K\cap P$.  We will frequently want to make the following assumptions:
\begin{asss}\label{zpass} ~

\begin{enumerate}[label=(\roman*)]
    \item $\?{P}^{\circ}=\?{P}$ (e.g., $\wh{A}$ is an integral domain and each $z^{\varphi(p)}$ is nonzero).
    \item For every $p\in \?{P}$, $A_{\varphi(p)+P\cap K}=z^{\varphi(p)}A_K$.
\end{enumerate}  
\end{asss}
These assumptions are indeed satisfied in Examples \ref{gegs}(i)-(iii).  

Assumption \ref{zpass}(ii) implies that $\wh{A}$ is topologically spanned over $\wh{A}_K$ by $\{z^{\varphi(p)}|p\in \?{P}\}$.  It follows from this and \eqref{zp} that $\{\vartheta_{p,Q}|p\in \?{P}\}$ spans $\wh{A}$ topologically over $\wh{A}_K$.  In summary, we have the following:

\begin{prop}\label{ThetaProp}
For fixed generic $Q \in N_{\bb{R}}\setminus \Supp(\f{D})$ and any $p\in \?{P}$, \eqref{vartheta-dfn} gives a well-defined element $\vartheta_{p,Q} \in z^{\varphi(p)}+\wh{A}_{p+N^+} \subset \wh{A}$.  Under Assumptions \ref{zpass}, the theta functions $\Theta_Q:=\{\vartheta_{p,Q}|p\in \?{P}\}$ form an additive topological basis for $\wh{A}$ over $\wh{A}_K$, hence also (at least topologically) span the subalgebra $A_{\Theta,Q}\subset \wh{A}$ generated by $\Theta_Q$.
\end{prop}

The following is a fundamental feature of theta functions.\footnote{When working over $\f{h}$, chambers of the scattering diagram give charts for the mirror manifold, and path-ordered products give the transition functions.  In this context, Theorem \ref{CPS} can roughly be interpreted as saying that the locally defined theta functions $\vartheta_{p,Q}$ patch together correctly to form global functions on the mirror.}

\begin{thm}[Refined \cite{CPS} result]\label{CPS}
Consider $\f{D}=\scat(\f{D}_{\In})$ as in Theorem \ref{KSGS}.  Fix two generic points $Q_1,Q_2\in N_{\bb{R}}\setminus \Supp(\f{D})$.  Let $\gamma$ be a smooth path in $N_{\bb{R}}\setminus \Joints(\f{D})$ from $Q_1$ to $Q_2$.  Then for any $p\in \?{P}$,
\begin{align*}
    \vartheta_{p,Q_2}=\theta_{\gamma,\f{D}} (\vartheta_{p,Q_1}).
\end{align*}
\end{thm}
When working over the module of log derivations as in Example \ref{gegs}(i) or (ii) (but for more general consistent scattering structures on more general integral affine manifolds than just $N_{\bb{R}}$), Theorem \ref{CPS} is due to \cite{CPS} (their Lemmas 4.7 and 4.9).  The author imagines that the arguments of \cite{CPS} can be generalized to any $\f{g}$ and $\wh{A}$ as above, but in \S \ref{CPSproof} we will sketch a new argument in terms of counts of tropical disks.

Theorem \ref{CPS} implies in particular that, as an abstract $\wh{A}_K$-algebra, $A_{\Theta,Q}$ is independent of the choice of $Q$ (although the embedding into $\wh{A}$ does depend on $Q$).  We will denote this abstract algebra by $A_{\Theta}$, and we let $\vartheta_p\in A_{\Theta}$ denote the element $\vartheta_{p,Q}\in A_{\Theta,Q}$ under this identification $A_{\Theta,Q}\cong A_{\Theta}$.

Under Assumptions \ref{zpass}, one sees that $A_{\Theta}$ and the theta functions are determined by the structure constants $\alpha(p_1,\ldots,p_s;p)\in \wh{A}_K$, $p_1,\ldots,p_s,p\in \?{P}$, defined by
\begin{align*}
    \vartheta_{p_1}\cdots \vartheta_{p_s} = \sum_{p\in \?{P}:z^p\neq 0} \alpha(p_1,\ldots,p_s;p) \vartheta_p.
\end{align*}
Even when Assumptions \ref{zpass} do not hold, each generic $Q\in N_{\bb{R}}\setminus \Supp(\f{D})$ determines an embedding $A_{\Theta} \cong A_{\Theta,Q} \subset \wh{A}$, hence a $\?{P}$-grading on $A_{\Theta}$, and we define\footnote{We note that $\alpha_Q(p_1,\ldots,p_s;p)$ is independent of $Q$ if $p=0$.} \begin{align}\label{alphaQ}
\alpha_Q(p_1,\ldots,p_s;p) \in A_p  
\end{align}
to be the degree $p$ part of $\vartheta_{p_1,Q}\cdots \vartheta_{p_s,Q}$.  The next proposition (generalizing \cite[Prop. 6.4(3)]{GHKK} and following the same argument) tells us how to compute the $\alpha$'s.
\begin{prop}\label{alphaprop}
For $p_1,\ldots,p_s,p\in \?{P}$ and generic $Q\in N_{\bb{R}}\setminus \Supp(\f{D})$,
\begin{align}\label{alphaform}
\alpha_Q(p_1,\ldots,p_s;p)=\sum_{\substack{\gamma_1,\ldots,\gamma_s \\
    	               \Ends(\gamma_i)=(p_i,Q), i=1,\ldots,s \\
		                \pi_K(v_{\gamma_1}+\ldots+v_{\gamma_s}) = p}}
	                a_{\gamma_1} \cdots a_{\gamma_s},
\end{align}
where the sum is over all ordered $s$-tuples of broken lines $(\gamma_i)_{i=1,\ldots,s}$ with $\Ends(\gamma_i)=(p_i,Q)$, and $a_{\gamma_i}\in A_{v_{\gamma_i}}$ is the element attached to the final straight segment of $\gamma_i$.

Now suppose Assumptions \ref{zpass} hold.  Then $\alpha(p_1,\ldots,p_s;0)=\alpha_{Q}(p_1,\ldots,p_s;0)$ for each generic $Q\in N_{\bb{R}}\setminus \Supp(\f{D})$.  More generally, for any $p\in \?{P}$, we have
\begin{align}\label{alphaKform}
\alpha(p_1,\ldots,p_s;p)z^{\varphi(p)} = \sum_{\ell \in \bb{Z}_{\geq 0}} \sum_{\substack{\gamma_1,\ldots,\gamma_s \\
    	               \Ends(\gamma_i)=(\varphi(p_i),Q_{\ell}), i=1,\ldots,s \\
		                \pi_K(v_{\gamma_1}+\ldots+v_{\gamma_s}) = p \\
		                d(v_{\gamma_1}+\ldots+v_{\gamma_s}-\varphi(p))=\ell}}
	                a_{\gamma_1} \cdots a_{\gamma_s} \in \wh{A}_K,
\end{align}
where $d$ is as in \eqref{dn} and $Q_{\ell}$ shares a maximal cell of $\f{D}^{\ell}$ with $\varphi(p)$.
\end{prop}
\begin{proof}
This first claim is straightforward from the definitions, with the finiteness of the sum in \eqref{alphaform} following from the well-definedness of $\vartheta_{p_i,Q} \in \wh{A}$ (Prop. \ref{ThetaProp}).

For the second claim, we observe that the straight broken line with attached element $z^{\varphi(p)}$ is the only broken line $\gamma$ over $\f{D}^{\ell}$ with $\pi_K(v_{\gamma})=p$ and end at a point $Q_{\ell}$ which shares a maximal cell of $\f{D}^{\ell}$ with $\varphi(p)$.  To see this, note that if we start at $Q_{\ell}$ and move in the $v_{\gamma}$-direction, then  we will never hit a wall of $\f{D}^{\ell}$ and so $\gamma$ cannot contain any bends.  Hence, the only $q\in \?{P}$ such that $\vartheta_{q,Q_{\ell}}$ has a $z^{\varphi(p)}$-term is $q=p$.  On the other hand, \eqref{zp} says that $\vartheta_{p,Q}$ equals $z^{\varphi(p)}$ plus higher degree terms.  Thus, for any $f\in \wh{A}$, the $z^{\varphi(p)}$-coefficient of $f$ expanded in the topological $\wh{A}_K$-module basis $\{z^{\varphi(n)}\}_{n\in \?{P}}$ of $\wh{A}$ must agree, modulo the topological span of $A^{\geq \ell}$, with the $\vartheta_{p,Q_{\ell}}$-coefficient of $f$ expanded in the topological basis $\{\vartheta_{p,Q_{\ell}}\}_{p\in P}$.  The claim now follows from considering the case $f=\vartheta_{p_1,Q_{\ell}}\cdots \vartheta_{p_s,Q_{\ell}}$.
\end{proof}

\subsection{A non-degenerate trace pairing}\label{TrNonDeg}
The Frobenius Structure Conjecture of \cite[arXiv v1, \S 0.4]{GHK1} predicts the existence of a certain associative algebra associated to any log Calabi-Yau variety with maximal boundary $(Y,D)$.  More precisely, the algebra has an additive (topological) basis of ``theta functions,''  and the multiplication rule is determined by a ``trace'' function which is defined in terms of certain descendant log Gromov-Witten invariants of $(Y,D)$.  In this subsection we will consider a certain trace function on $A_{\Theta}$ and prove that it is non-degenerate, hence is sufficient to completely determine the structure constants for the theta function multiplication.  Separate work of the author uses Theorem \ref{MainTheorem} and some tropical correspondence results to prove that this trace really is given by the desired GW invariants, and the combination of these results proves the Frobenius structure conjecture for cluster varieties. 

Viewing $\wh{A}$ as a topological $\?{P}$-graded $\wh{A}_K$-algebra, we have a map of $\wh{A}_K$-modules
\begin{align*}
\Tr:\wh{A} \rar \wh{A}_K
\end{align*}
taking an element $f\in \wh{A}$ to its degree $0$ part (using the $\?{P}$-grading).  Since we assumed that $\f{g}\cdot A_K=0$, all wall-crossing automorphisms act trivially on $\wh{A}_K$, and so $\Tr$ induces a map $\Tr:A_{\Theta}\rar \wh{A}_K$ as well (no dependence on $Q$). 
$\Tr$ also induces an ``s-point function''
\[
\begin{array}{c c}
\Tr^s: \wh{A}^{\otimes s} \rar \wh{A}_K, 
&f_1\otimes \cdots \otimes f_s \mapsto \Tr(f_1\cdots f_s),
\end{array}
\]
and similarly for $\Tr^s:A_{\Theta}^{\otimes s} \rar \wh{A}_K$ for each $s\geq 1$. The following theorem implies that these uniquely determines $A_{\Theta}$ and the theta functions.\footnote{An algebro-geometric proof for a version of Theorem 
\ref{TrThm} in the two-dimensional tropical vertex group situation has previously been found by Gross-Hacking-Keel \cite{GHK2}.}

\begin{thm}\label{TrThm}
Assume that $A$ is an integral domain and that $z^{\varphi(p)}$ is nonzero for each $p\in \?{P}$.  The map
\begin{align*}
\Tr^{\vee}:A_{\Theta} \rar \Hom_{\wh{A}_K}(A_{\Theta}, \wh{A}_K), ~ a\mapsto [b\mapsto \Tr(ab)]
\end{align*}
is injective.\footnote{In fact, the same proof shows the strong statement that the similarly defined map $\Tr^{\vee}:\wh{A} \rar \Hom_{\wh{A}_K}(A_{\Theta,Q}, \wh{A}_K)$ is injective for each $Q$.}  Hence, given the topological $\wh{A}_K$-module structure on $A_{\Theta}$, the $\wh{A}_K$-algebra structure (i.e., the multiplication rule) is uniquely determined by $\Tr^2$ and $\Tr^3$.  In particular, if Assumption \ref{zpass} holds, then all the structure constants $\alpha_K(p_1,\ldots,p_s;p)$ are determined by those of the form $\alpha(p_1,p_2;0)$ and $\alpha(p_1,p_2,p_3;0)$.
\end{thm}
\begin{proof}
To prove that $\Tr^{\vee}$ is injective, we will show that for any $f\in \wh{A}$, there exists some $p\in \?{P}$ such that $\Tr(f\vartheta_p)\neq 0$.  Pick some generic $Q\in N_{\bb{R}}\setminus \Supp(\f{D})$ so we can view $A_{\Theta}$ as a topological $P$-graded $A_0$-algebra $A_{\Theta,Q}$.  For nonzero $f\in A_{\Theta,Q}$, choose $p_0\in P$ such that $d(p_0)$, as defined in \eqref{dn}, is as small as possible subject to the condition that the degree $p_0$ part of $f$, denoted $f_{p_0}$, is nonzero.  Let $\?{p_0}=\pi_K(p_0)$.  By \eqref{zp}, $\vartheta_{-\?{p_0}}=z^{\varphi(-\?{p_0})}+$[terms with higher $d$].  So the degree $p_0+\varphi(-\?{p_0})$ part of $f$ is $f_{p_0}z^{\varphi(-\?{p_0})} \neq 0$.  Since $\pi_K(p_0+\varphi(-\?{p_0})) = 0\in \?{P}$, degree $p_0+\varphi(-\?{p_0})$ with respect to the $P$-grading implies degree $0$ with respect to the $\?{P}$-grading.  Hence, $\Tr(f\vartheta_{\?{p_0}})\neq 0$, as desired. 

For the remaining claims, suppose we want to determine the product of two elements $a,b\in \wh{A}$.  The above injectivity implies that it is enough to specify $\Tr(abc)=\Tr^2(ab,c)$ for each $c \in \wh{A}$, and this is equal to $\Tr^3(a,b,c)$.  The claim about the structure constants then follows because Assumption \ref{zpass} implies that the theta functions span (topologically), so knowing the multiplication rule for the theta functions determines the whole ring.
\end{proof}

\begin{rmk}[Frobenius algebras]
Recall that a Frobenius $R$-algebra is defined to be an $R$-algebra $A$, together with an $R$-algebra homomorphism $\Tr:A\rar R$, such that the map $\Tr^{\vee}:A\rar \Hom_R(A,R)$, $a\mapsto [b\mapsto \Tr(ab)]$, is an isomorphism.  This forces $A$ to be finite-dimensional.  If we allow $\Tr^{\vee}$ to instead be just injective, rather than an isomorphism, we could define infinite dimensional Frobenius algebras.  Such structures appear, for example, in \cite{BSS}.  Theorem \ref{TrThm} then says that $\Tr$ makes $A_{\Theta}$ into an infinite dimensional Frobenius $\wh{A}_K$-algebra.
\end{rmk}

\section{Tropical curves and the main results}\label{Main Section}

For use throughout this section, let us fix an initial scattering diagram $\f{D}_{\In}:=\{(m_{\f{d}_i},\f{d}_i,g_{\f{d}_i})|i\in I\}$ with $I$ a finite index-set  ($I$ here actually corresponds to $I\setminus F$ in the setup of Examples \ref{InExamples}) and $\f{d}_i=m_{\f{d}_i}^{\perp}$.  We can decompose $g_{\f{d}_i}$ as 
\begin{align}\label{gdi}
    g_{\f{d}_i}=\sum_{j\geq 1} g_{ij}\in \f{g}^{\parallel}_{n_{\f{d}_i},m_{\f{d}_i}}
\end{align}
with $g_{ij}\in \f{g}_{jn_{\f{d}_i}}$ ($j$ and $n_{\f{d}_i}$ being multiplied in this subscript).  For example, $\f{D}_{\In}$ could be any of the initial scattering diagrams from Examples \ref{InExamples}.  Let $\f{D}:=\scat(\f{D}_{\In})$ as in Theorem \ref{KSGS}.  We will describe $\f{D}$ and the associated theta functions in terms of counts of tropical curves and tropical disks.

\subsection{Tropical curves and tropical disks}\label{TropBasics}

\begin{ntn}
For any weighted graph $\Gamma$, possibly with some $1$-valent vertices removed, we let $\Gamma^{[0]}$, $\Gamma^{[1]}$, and $\Gamma^{[1]}_{\infty}$ denote the vertices, edges, and non-compact edges,\footnote{If $\Gamma$ consists of single edge and no vertices, we view $\Gamma_{\infty}^{[1]}$ as including two elements, one for each unbounded direction.  This case without vertices often requires special treatment.} respectively.  By ``weighted,'' we mean that $\Gamma$ is equipped with a function $w:\Gamma^{[1]}\rar \bb{Z}_{\geq 1}$.  
\end{ntn}

Let $\?{\Gamma}$ be a weighted, connected, finite tree without bivalent vertices, and let $\Gamma$ be the complement of the $1$-valent vertices.  We mark the non-compact edges via $\epsilon:S \overset{\sim}{\rar} \Gamma^{[1]}_{\infty}$ for some finite index set $S$.  Given $i\in S$, let $E_i$ denote $\epsilon(i)$.  Let $L$ be a finite-rank lattice.

\begin{dfn}
A  parameterized marked tropical curve in $L_{\bb{R}}$ is the data $(\Gamma, \epsilon)$ as above, along with a proper continuous map $h:\Gamma\rar L_{\bb{R}}$ such that:
\begin{itemize}
\item For each $E\in \Gamma^{[1]}$, $h|_{E}$ is an embedding with image contained in an affine line of rational slope.
\item The following ``balancing condition'' holds for every $V\in \Gamma^{[0]}$: For each edge $E\in \Gamma^{[1]}$ containing $V$, let $u_{(V,E)}\in L\setminus \{0\}$ denote the primitive integral vector emanating from $V$ in the direction $h(E)$.  Then
\begin{align}\label{balance}
\sum_{\substack{E\in \Gamma^{[1]} \\ E\ni V}} w(E)u_{(V,E)} = 0.
\end{align}
\end{itemize}
Two parameterized marked tropical curves $h_i:\Gamma_i\rar L_{\bb{R}}$, $i=1,2$, are isomorphic if there is a homeomorphism $\phi:\Gamma_1\rar \Gamma_2$ respecting the weights, markings, and maps $h_i$.  A (rational) tropical curve is an isomorphism class of parameterized marked tropical curves.

A tropical disk is defined in nearly the same way, except that $\Gamma$ is equipped with a marked vertex $Q_{\out}$ which is allowed to have any valence (including being univalent or bivalnet).  Furthermore, $Q_{\out}$ is not required to satisfy the balancing condition.
\end{dfn}

The type of a tropical curve or disk is the data of the weighted marked graph $(\Gamma,\epsilon)$, along with the vectors $u_{(V,E)}$ for each $V\in \Gamma^{[0]}$ and $E\in \Gamma^{[1]}$ with $E\ni V$.  If $\Gamma$ has no vertices, the type includes the data of the two unbounded directions.

For each $i\in S$, let $u_{E_i}$ denote the primitive vector pointing in the unbounded direction of $h(E_i)$.  The  degree $\Delta$ of a marked tropical curve/disk $(h,\Gamma,\epsilon)$ is the map $\Delta:S\rar L$ taking $i\in S$ to $w(E_i)u_{E_i}\in L\setminus \{0\}$.  

Let $\A$ denote a collection $\{B_i\subset L_{\bb{R}}|i\in S\}$ of affine subspaces of $L_{\bb{R}}$ indexed by $S$, plus an additional affine subspace $B_{\out}$ if we are considering tropical disks rather than tropical curves.  We say that a tropical curve $(h,\Gamma,\epsilon)$ matches the constraints $\A$ if $h(E_i)\subset B_i$ for each $i\in S$.  Similarly for a tropical disk with the additional requirement that $h(Q_{\out})\in B_{\out}$.  We call the conditions imposed by $\A$ incidence conditions.

For $s\geq 1$, we say\footnote{Higher-valence conditions as a tropical analog of $\psi$-class conditions first appeared in \cite{Mi2}, with proofs of various descendant correspondence theorems appearing in \cite{Mr,GrP2,Over,AGr,MRud}.  The last two of these apply in particular to the tropical curve counts which appear here when working over $\f{h}$.} that a tropical disk satisfies the $\psi$-class condition $\psi_{Q_{\out}}^{s-2}$ if $$\val(Q_{\out}) \geq s.$$  Note that we can have $s=-1$ if $Q$ is univalent.

Let $\f{T}_{\Delta}(\A)$ denote the set of tropical curves of degree $\Delta$ satisfying incidence conditions $\A$.  Let $\f{T}'_{\Delta}(\A,s-2)$ denote 
the set of tropical disks satisfying  incidence conditions $\A$ and the $\psi$-class condition $\psi_{Q_{\out}}^{s-2}$.

\subsubsection{Degrees and incidence conditions coming from scattering diagrams}\label{DegIncScat}

Let $\ww:=(\ww_i)_{i\in I}$ be a tuple of weight vectors $\ww_i:=(w_{i1},\ldots,w_{il_i})$ with $0< w_{i1} \leq \ldots \leq w_{il_i}$, $w_{ij}\in \bb{Z}$.  For $\Sigma_{l_i}$ denoting the group of permutations of $\{1,\ldots,l_i\}$, let \[\Aut(\ww)\subset \prod_{i\in I} \Sigma_{l_i}\] be the group of automorphisms of the second indices of the weights $\ww_i$ which act trivially on $\ww$.

Recall the lattice $\?{N}=N/K=\pi_K(N)$ from \S \ref{scatter}.  We will consider tropical curves in $\?{N}_{\bb{R}}$.  Let $\Delta_{\ww}$ denote the tropical curve degree $$\Delta_{\ww}:I_{\ww}\rar N\setminus \{0\}$$ with $I_{\ww}:=\{(i,j)|i\in I, j=1,\ldots,l_i\}\cup \{\out\}$, $\Delta((i,j))=w_{ij}n_{\f{d}_i}$, and $\Delta(\out):=-n_{\out}$, where \begin{align*}
    n_{\out}:=\sum_{i,j} w_{ij}n_{\f{d}_i}.
\end{align*} 
Here, $\out$ is the label for an unbounded edge $E_{\out}$.  We will typically write $E_{(i,j)}$ as simply $E_{ij}$.

Now let $\pp:=(p_1,\ldots,p_s)$ denote an $s$-tuple of vectors in $\?{P}\setminus \{0\}$, $s\geq 1$.  We let $\Delta_{\ww,\pp}$ denote the tropical disk degree $$\Delta_{\ww,\pp}:I_{\ww,\pp}\rar N\setminus \{0\}$$ with $I_{\ww,\pp}:=\{(i,j)|i\in I, j=1,\ldots,l_i\}\cup \{1,\ldots,s\}$, $\Delta((i,j)):=w_{ij}n_{\f{d}_i}$, and $\Delta(k)=\varphi(p_k)$ for $k=1,\ldots,s$.

Given $n\in N_{\bb{R}}\setminus \{0\}$, let $\A_{\ww,n}$ denote the incidence conditions $\{B_{ij}|(i,j)\in I_{\ww}\}\cup \{B_{\out}\}$ with each $B_{ij}$ a generic translate of $\f{d}_i$, and with $B_{\out}$ a generic translate of the span of $n$ and $n_{\out}=\sum_{i,j} w_{ij}n_{\f{d}_i}$.

Similarly, given a generic point $Q\in N_{\bb{R}}\setminus \Supp(\f{D})$, we define the incidence conditions $\A_{\ww,\pp,Q}$ as follows: take $B_{ij}$'s as before, take $B_k:=N_{\bb{R}}$ for each $k=1,\ldots,s$ (i.e., there are no incidence conditions on the $E_k$'s), and after fixing the generic $B_{ij}$'s, take $B_{\out}$ to be a single point $rQ$ for $r\gg 0$ ($r$ sufficiently large relative to the distance of the $B_{ij}$'s from the origin).

With these conditions, $\f{T}_{\Delta_{\ww}}(\A_{\ww,n})$ and $\f{T}'_{\Delta_{\ww,\pp}}(\A_{\ww,\pp,Q},s-2)$ are finite, so we can ``count'' the elements of these sets after assigning appropriate multiplicities to each tropical curve.

We note that for generic incidence conditions, every vertex of the tropical curves/disks in these sets will be trivalent except for possibly $Q_{\out}$, which will be $s$-valent.  Furthermore, for the tropical disks, each of the $s$ components of $\Gamma\setminus Q_{\out}$ will consist of exactly one of the edges of the form $E_k$, $k=1,\ldots,s$.

\subsubsection{Multiplicities}\label{MultSection} 
We next define the multiplicities of the tropical curves/disks in $\f{T}_{\Delta_{\ww}}(\A_{\ww,n})$ and $\f{T}'_{\Delta_{\ww,\pp}}(\A_{\ww,\pp,Q},s-2)$.  While the general definition here is somewhat complicated (particularly the issue of signs), we will see in Remark \ref{SignSkewCases} and Examples \ref{MultExamples} that the computation can be simplified significantly in all the examples we care about.

Let us begin with $\Gamma \in \f{T}_{\Delta_{\ww}}(\A_{\ww,n})$ for some $n\in N_{\bb{R}}\setminus \{0\}$.  By thinking of the edges $E_{ij}\in \Gamma^{[1]}_{\infty}$ as being incoming edges and the edge $E_{\out}\in \Gamma^{[1]}_{\infty}$ as being an outgoing edge, we obtain a flow on $\Gamma$.   For each $E\in \Gamma^{[1]}$, let $u_E\in N$ be the primitive vector tangent to $E$ pointing in the opposite direction of the flow of $\Gamma$.   To each incoming edge $E_{ij}$ we associate the element $m_{\f{d}_i}\in \?{M}$ and the element $g_{iw_{ij}}\in \f{g}_{w_{ij}u_i,m_{\f{d}_i}}$ from the expansion of $g_{\f{d}_i}$ given in \eqref{gdi}.

We now use the flow to recursively associate, up to signs, an element $m_E\in \?{M}\subset M$ and an element $g_E\in \f{g}^{\parallel}_{w(E)u_E,m_E}$ to every edge $E\in \Gamma^{[1]}$ as follows: Suppose two edges $E_1$ and $E_2$ flow into a common vertex $V$, with edge $E_3$ flowing out of $V$, such that $E_1$ and $E_2$ have associated elements $m_{E_1},m_{E_2}\in \?{M}$ and $g_{E_1}\in \f{g}_{n_1,m_{E_1}}$, $g_{E_2}\in \f{g}_{n_2,m_{E_2}}$ for $n_1=w(E_1)u_{E_1}$ and $n_2=w(E_2)u_{E_2}$.  Then we define
\begin{align}\label{mE3}
    m_{E_3}:=\mu((n_1,m_{E_1}),(n_2,m_{E_2})) \in \?{M}
\end{align}
for $\mu$ as defined in \eqref{m}, and we define
\begin{align}\label{g1g2}
g_{E_3}:=[g_{E_1},g_{E_2}]\in \f{g}_{n_1+n_2,m_{E_3}},
\end{align}
where the containment in $\f{g}_{n_1+n_2,m_{E_3}}$ utilizes Condition \eqref{gnm}. Note that reordering $E_1$ and $E_2$ will change the signs of both $m_{E_3}$ and $g_{E_3}$ above.  

This flow process determines (up to simultaneously changing both signs) elements
\begin{align}\label{mgGamma}
  m_{\Gamma}:=m_{E_{\out}}\in \?{M}  \hspace{.25 in} \mbox{and} \hspace{.25 in} g_{\Gamma}:=g_{E_{\out}}\in \f{g}_{n_{\out},m_{E_{\out}}}
\end{align} associated to the outgoing edge of $\Gamma$.

Now, suppose $n=\varphi(p)$ for some $p\in \?{P}$.  If $n\notin m_{\Gamma}^{\perp}$, we define
\begin{align}\label{MultGamma0}
    \Mult(\Gamma):=\sign \langle n,m_{\Gamma}\rangle (g_{\Gamma}\cdot z^n)\in A.
\end{align}
where $\cdot$ is the action of $\f{g}$ on $A$.  If $n$ is in $m_{\Gamma}^{\perp}$, we take $\Mult(\Gamma):=0$ (which in practice is typically equal to $g_{\Gamma}\cdot z^n$ in this case anyway).  Note that the factor $\sign \langle n,m_{\Gamma}\rangle$ makes up for the ambiguity in the ordering of the edges $E_1$ and $E_2$ above. 

Now suppose $\Gamma \in \f{T}'_{\Delta_{\ww,\pp}}(\A_{\ww,\pp,Q},s-2)$.  
We use a flow on $\Gamma$ like before, this time with all unbounded edges being sources and $Q_{\out}$ being the sink.  We again associate $g_{iw_{ij}}\in \f{g}$ and $m_{\f{d}_i}\in M$ to $E_{ij}$ for each $(i,j)$, and we associate $z^{\varphi(p_i)}\in A$ to $E_k\in \Gamma^{[1]}_{\infty}$ for $k=1,\ldots,s$.  

We now recursively assign to every edge $E$ either elements $m_E$ and $g_E$ as before, or an element $a_E\in A_{w(E)u_E}$.  When two edges with associated elements of $\?{M}$ and $\f{g}$ flow into a vertex, outgoing elements in $\?{M}$ and $\f{g}$ are determined as before.     On the other hand, if $E_1,E_2$ flow into a vertex, and $E_1$ has associated elements $m_E\in u_{E_1}^{\perp}$ and $g_{E_1}\in \f{g}_{w(E_1)u_{E_1},m_E}$, while $E_2$ has associated element $a_{E_2}\in A_{w(E_2)u_{E_2}}$, we associate to the outgoing edge $E_3$ the element 
\begin{align}\label{g1a2}
a_{E_3}:=\sign\langle u_{E_2},m_{E_1}\rangle (g_{E_1}\cdot a_{E_2})\in A_{w(E_3)u_{E_3}}.
\end{align}
We note that $g_{E_1}\cdot a_{E_2}$ above may be viewed as a bracket $[g_{E_1},a_{E_2}]$ as in \eqref{Ag}, so \eqref{g1a2} is indeed analogous to \eqref{g1g2}.

Now, for $k=1,\ldots,s$, let $E_{\out,k}$ denote the edge of $\Gamma$ containing $Q_{\out}$ which is in the same connected component of $\Gamma\setminus Q_{\out}$ as $E_k$.  Then
\begin{align}\label{MultDisk}
    \Mult(\Gamma):=a_{E_{\out,1}}a_{E_{\out,2}}\cdots a_{E_{\out,s}}\in A_{n_{\out}}\subset A,
\end{align}
 where now, $n_{\out}=\sum_{\ell \in I_{\ww,\pp}} \Delta_{\ww,\pp}(\ell) = \sum_{(i,j)} w_{ij}n_{\f{d}_i} + \sum_k \varphi(p_k)$.

\begin{rmk}[Signs of multiplicities in skew-symmetric cases]\label{SignSkewCases}
The sign issues in the multiplicity definitions above can be simplified when $\f{g}$ is skew-symmetric---there is a canonical choice of ordering for the commutators and an easy way to find each $m_E$ when using this choice.  Recall that we call $\f{g}$ skew-symmetric if there is a skew-symmetric form $\omega=\{\cdot,\cdot\}$ on $N$ such that $[\f{g}_{n_1},\f{g}_{n_2}]=0$ whenever $\{n_1,n_2\}=0$.  Using Example \ref{equivD}(i), we assume that $m_{\f{d}_i}=\omega_1(n_{\f{d}_i})$ for every wall $\f{d}_i\in \f{D}_{\In}$.  Then, when choosing an ordering for a commutator as in \eqref{g1g2} above, pick $[g_{E_1},g_{E_2}]$ if $\{u_{E_1},u_{E_2}\}\geq 0$ and take the reverse ordering otherwise.  With these choices, one checks that $m_E$ is always given by $\omega_1(w(E)u_E)$.  Hence, the factor $\sign \langle n,m_{\Gamma}\rangle$ of \eqref{MultGamma0} is simply $\sign(\{u_E,n\})$.  Similarly, the factor $\sign(\langle u_{E_2},m_{E_1}\rangle)$ from \eqref{g1a2} is simply $\sign(\{u_{E_1},u_{E_2}\})$. 
\end{rmk}

We now explain how $\Mult(\Gamma)$ can be computed more simply in the cases of Examples \ref{gegs}(i)-(iii).

\begin{egs}\label{MultExamples} ~

\begin{enumerate}[label=(\roman*)]\setlength\itemsep{1em}
\item An alternative approach to scattering diagrams over $\f{h}$ as in Example \ref{gegs}(i) (cf. \cite{GPS}) is to attach not an element of $\wh{\f{h}}$ to each wall, but rather, an element of $R\llb N^{\oplus}\rrb$.  In this perspective, a wall is expressed as $(\f{d},f_{\f{d}})$, $f_{\f{d}}\in R\llb N^{\oplus}\rrb$.  Letting $m_{\f{d}}$ be either primitive element of $\?{M}$ which vanishes on $\f{d}$, the wall $(\f{d},f_{\f{d}})$ would in our approach be written as $(m_{\f{d}},\f{d},f_{\f{d}}\partial_{m_{\f{d}}})$.  

Now, suppose that the elements of $\f{h}$ attached to the initial edges $E_{ij}$ have the form $g_{iw_{ij}}=a_{iw_{ij}}z^{w_{ij}e_i}\partial_{m_{\f{d_i}}}$ for some constants $a_{iw_{ij}}\in R$.  Let $a_{\ww}:=\prod_{i,j} a_{iw_{ij}}$.  If $\Gamma\in \f{T}_{\Delta_{\ww}}(\A_{\ww,n})$, then $$g_{\Gamma}=a_{\ww} z^{n_{\out}} \partial_{m_{\Gamma}},$$ and so if $n\in \varphi(P)$,
\begin{align}\label{hCurveMult}
    \Mult(\Gamma) = a_{\ww} |\langle n,m_{\Gamma} \rangle|z^{n+n_{\out}}.
\end{align}
Note that the computation of \eqref{hCurveMult} does not actually require knowing the sign of $m_{\Gamma}$.

Similarly, for $\Gamma\in \f{T}'_{\Delta_{\ww,\pp}}(\A_{\ww,\pp,Q},s-2)$, we can compute $\Mult(\Gamma)$ using iterated commutators and actions $g_{E_1}\cdot a_{E_2}$ as in \eqref{g1a2} without worrying about signs: the product in \eqref{MultDisk} will be an element of the form \begin{align}\label{hDiskMult}
    a_{\ww}kz^{n_{\out}}
\end{align} for some $k\in \bb{Z}$, and the correct sign choices will result in $k$ being non-negative.

It follows from joint work of the author and H. Ruddat \cite{MRudMult} that the factor $|\langle n,m_{\Gamma} \rangle|$ in \eqref{hCurveMult}, and the factor $|k|$ of \eqref{hDiskMult} (in the case $n_{\out}=0$ so we have honest tropical curves) are the correct multiplicities for counting tropical curves if one wishes for the counts to equal the appropriate corresponding Gromov-Witten invariants.  Furthermore, the factors $a_{\ww}$ are related to counts of multiple covers of certain $(-1)$-curves in a degeneration of a certain blowup of a toric variety.  This is used in the author's proof of the Frobenius structure conjecture for cluster varieties \cite{ManFrob}.

\item In the case that $\f{g}=\f{g}^{\omega}\subset \f{h}$ as in Example \ref{gegs}(ii), the multiplicity formula simplifies further.  For each vertex $V\neq Q_{\out}$ with edges $E_1,E_2$ containing $V$, $n_1:=w(E_1)u_{E_1}$ and $n_2:=w(E_2)u_{E_2}$, define \begin{align*}
\Mult(V):=|\{n_1,n_2\}|.
\end{align*}
For $\Gamma\in \f{T}_{\Delta_{\ww}}(\A_{\ww,n})$, using Remark \ref{SignSkewCases}, one finds that
\begin{align*}
    \Mult(\Gamma) = a_{\ww} \left(\prod_{V\in \Gamma^{[0]}} \Mult(V)\right)|\{n_{\out},n\}|z^{n+n_{\out}}.
    \end{align*}
For $\Gamma\in \f{T}'_{\Delta_{\ww,\pp}}(\A_{\ww,\pp,Q},s-2)$, define $\Mult(Q_{\out})=1$.  Then
\begin{align*}
    \Mult(\Gamma) = a_{\ww} \left(\prod_{V\in \Gamma^{[0]}} \Mult(V)\right)z^{n_{\out}}.
\end{align*}

\item Similarly for the quantization $\f{g}=\f{g}^{\omega}_q$: For $V\neq Q_{\out}$, take $\Mult_q(V):=[|\{n_1,n_2\}|]_q$, where $n_1$ and $n_2$ are weighted tangent vectors of edges containing $V$, and we recall $[a]_q$ denotes $q^a-q^{-a}$.  Then for $\Gamma\in \f{T}_{\Delta_{\ww}}(\A_{\ww,n})$ we have
\begin{align*}
    \Mult(\Gamma) = a_{\ww} \left(\prod_{V\in \Gamma^{[0]}} \Mult_q(V)\right)[|\{n_{\out},n\}|]_q z^n.
\end{align*}
For $\Gamma\in \f{T}'_{\Delta_{\ww,\pp}}(\A_{\ww,\pp,Q},s-2)$, define
\begin{align*}
    \Mult_q(Q_{\out})=q^{\sum \{n_i,n_j\}}
\end{align*}
where the sum is over all pairs $i,j\in \{1,\ldots,s\}$ with $i<j$, and $n_k:=w(E_{\out,k})u_{E_{\out,k}}$.  Equivalently, $\Mult_q(Q_{\out})$ is determined by $z^{n_1} z^{n_2} \cdots z^{n_s} = \Mult_q(Q_{\out}) z^{n_{\out}}$. Then
\begin{align}\label{Multq}
    \Mult(\Gamma) = a_{\ww} \left(\prod_{V\in \Gamma^{[0]}} \Mult_q(V)\right)z^{n_{\out}}.
\end{align}
We note that (after removing the $a_{\ww}$-factors and $z^n$ or $z^{n_{\out}}$ factors) these quantum multiplicities extend the Block-G\"ottsche multiplicities of \cite{BG} to allow for these higher-dimensional cases with $\psi$-class conditions.  
\end{enumerate}
\end{egs}

With these multiplicities,  we can define
\begin{align*}
    \N^{\trop}_{\ww}(p):=\sum_{\Gamma \in \f{T}_{\Delta_{\ww}}(\A_{\ww,\varphi(p)})} \Mult(\Gamma)
\end{align*}
for each $p\in \?{P}\setminus \{0\}$, and
\begin{align*}
    \N^{\trop}_{\ww,\pp}(Q):=\sum_{\Gamma \in \f{T}'_{\Delta_{\ww,\pp}}(\A_{\ww,\pp,Q},s-2)} \Mult(\Gamma)
\end{align*}
for each generic $Q\in N_{\bb{R}}\setminus \Supp(\f{D})$.  Also, for $n\in N^+$ primitive, $\ww\in \s{W}(n)$, $m\in n^{\perp}\cap M$, and $n_0\in N_{\bb{R}}\setminus \{0\}$, we define
\begin{align*}
    \N^{\trop}_{\ww}(n_0; m):=\sum_{\substack{\Gamma \in \f{T}_{\Delta_{\ww}}(\A_{\ww,n_0}) \\ m_{\Gamma} \in \bb{R} m}} \sign(m_{\Gamma}/m) g_{\Gamma}.
\end{align*}
Here $m_{\Gamma}$ and $g_{\Gamma}$ are given as in \eqref{mgGamma}, and $\sign(m_{\Gamma}/m)$ is defined to be $+1$ if $m_{\Gamma}$ is a positive multiple of $m$ and $-1$ otherwise.

\begin{prop}\label{TropInvariant}
The quantities $\N^{\trop}_{\ww}(n_0;m)$ and $\N^{\trop}_{\ww}(p)$ do not depend on the choices of generic representatives of the incidence conditions $\A$.  For fixed $Q$, $\N^{\trop}_{\ww,\pp}(Q)$ does not depend on the generic choices of representatives for the conditions $\{B_{ij}\}_{ij}$.  If $n_{\out}:=\sum_{(i,j)} w_{ij}n_{\f{d}_i} + \sum_k \varphi(p_k)$ is contained in $K$, then $\N^{\trop}_{\ww,\pp}(Q)$ is also independent of the generic choice of $Q$.
\end{prop}
\begin{proof}
The invariance of $\N^{\trop}_{\ww}(n_0;m)$, $\N^{\trop}_{\ww}(p)$, and $\N^{\trop}_{\ww,\pp}(Q)$ (for fixed $Q$) will follow as immediate corollaries of Theorem \ref{ScatTropDisks}, Corollary \ref{ScatTropDisksCor}, and Theorem \ref{MainTheorem}, respectively.  The final statement will follow once we prove Theorem \ref{CPS} since all wall-crossings act trivially on $A_K$.
\end{proof}

\begin{rmk}
An earlier version of this paper (arXiv v3) claimed a direct proof of Proposition \ref{TropInvariant} rather than realizing it as a corollary of the results below.  However, that argument had a flaw, namely, the claim of Footnote 11 in that version is nontrivial, and in fact is false without our Condition \eqref{gnm} which was not present in that version.  However, the key ideas of that argument, plus a proof of the flawed footnote for some cases, will still be used in \S \ref{CPSproof} to prove Theorem \ref{CPS}.
\end{rmk}

For each $n\in N$, let $\s{W}_{\pp}(n)$ be the set of weight vectors $\ww$ such that
\begin{align*}
n_{\out}:=\sum_{i,j} w_{ij} n_{\f{d}_i} + \sum_{k=1}^s \varphi(p_k) = n.
\end{align*}
We will write just $\s{W}(n)$ for the cases where $\pp$ is empty (i.e., when considering tropical curves in $\f{T}_{\Delta_{\ww}}(\A_{\ww,n_0})$ for some $n_0$), so $\s{W}(n)$ is the set of weight vectors such that $n_{\out}:=\sum_{i,j} w_{ij}n_{\f{d}_i} = n$.

We are now ready to state the main theorems.

\begin{thm}\label{ScatTropDisks}
For $n\in N^+$ primitive and $m\in (n^{\perp}\cap M)\setminus \{0\}$, let $\f{D}(n,m)$ be the set of walls in $\f{D}$ with direction $-n$ and support parallel to $m^{\perp}$, i.e., walls $(m_{\f{d}},\f{d},g_{\f{d}})$ with $m_{\f{d}}\in \bb{Q} m$ and $g_{\f{d}}\in \f{g}_{n}^{\parallel}$.  By applying the equivalence of Example \ref{equivD}(i), assume that each $m_{\f{d}}$ here is in fact a positive rational multiple of $m$.  Let $n_0\in N_{\bb{R}}\setminus m^{\perp}$.  Then
\begin{align}\label{ScatTropDisksFormula}
\sum_{\f{d}\in \f{D}(n,m)} g_{\f{d}} = \sum_{\substack{k>0  \\ \ww \in \s{W}(kn)}} \frac{N^{\trop}_{\ww}(n_0;m)}{ |\Aut(\ww)|}.
\end{align}
\end{thm}

If every wall in $\f{D}(n,m)$ has the same support, then the sum on the left-hand side of \eqref{ScatTropDisksFormula} appears when combining the walls into a single wall via the equivalence of Example \ref{equivD}(ii).  This is the motivation for considering such an expression.

From the definition of the multiplicity of tropical curves in $\f{T}_{\Delta_{\ww}}(\A_{\ww,n})$, we immediately obtain the following as a corollary of Theorem \ref{ScatTropDisks}.

\begin{cor}\label{ScatTropDisksCor}
For primitive $n\in N^+$ and any $p\in \?{P}$, let 
\begin{align*}
    g_{n,p}:=\sum_{\substack{(m_{\f{d}},\f{d},g_{\f{d}})\in \f{D}\\ n_{\f{d}}=n}} \sign(p,m_{\f{d}}) g_{\f{d}}\cdot z^{\varphi(p)}.
\end{align*}
The sum here is over all walls of $\f{D}$ with direction $-n$.  I.e., $\exp(g_{n,p})$ is the image of $z^{\varphi(p)}$ under the automorphism associated with crossing these walls while moving in the direction $-\varphi(p)$.  Then 
\begin{align*}
g_{n,p} = \sum_{\substack{k> 0 \\ \ww \in \s{W}(kn)}} \frac{\N^{\trop}_{\ww}(p)}{|\Aut(\ww)|}. 
\end{align*}
\end{cor}

Now recall the structure constants $\alpha_Q(p_1,\ldots,p_s;p)$ of \eqref{alphaQ}.

\begin{thm}[Main Theorem]\label{MainTheorem}
For $\pp=(p_1,\ldots,p_s)$ an $s$-tuple of elements of $\?{P}\setminus \{0\}$, $s\geq 1$, $p\in \?{P}$, and $Q\in N_{\bb{R}}$ generic, we have 
\begin{align}\label{MainTrop}
\alpha_Q(p_1,\ldots,p_s;p) = \sum_{r\in K\cap P} \left(\sum_{\ww\in  \s{W}_{\pp}(\varphi(p)+r)}   \frac{\N^{\trop}_{\ww,\pp}(Q)}{|\Aut(\ww)|}\right).
\end{align}
\end{thm}

\begin{rmk}[Tropical ribbons]\label{ribbon}
We note that each of the above results can be restated using tropical ribbons in place of tropical curves/disks.  By a tropical ribbon, we mean the data of a tropical curve or disk, plus the additional data of a cyclic ordering of the edges at each vertex (cf. the ribbon trees and ribbon graphs of \cite{GSTheta,Abo,Slaw} for applications of such objects in related contexts).  Let us view $\f{g}$ as part of the commutator Lie algebra of some associative algebra $U_{\f{g}}$ (e.g., the universal enveloping algebra of $\f{g}$).  Then, in the definition of the multiplicities in \S \ref{MultSection}, in place of the commutator $[g_{E_1},g_{E_2}]$ of \eqref{g1g2}, we take the product $g_{E_1}g_{E_2}$ if the ordering $E_1,E_2,E_3$ agrees with the ribbon structure at $V$ and the product $-g_{E_2}g_{E_1}$ if it does not (keeping $m_{E_3}$ defined as in \eqref{mE3}).  Similarly for $\eqref{g1a2}$, viewing $\f{g}\oplus A$ now as part of the commutator Lie algebra of some associative algebra $U_{\f{g}\oplus A}$.  That is, we take $a_{E_3}$ there to be $\sign(\langle u_{E_2},m_{E_1}\rangle) g_{E_1} a_{E_2}$ if the ribbon structure agrees with the ordering $E_1,E_2,E_3$, and $-\sign(\langle u_{E_2},m_{E_1}\rangle) a_{E_2}g_{E_1}$ otherwise, where the products here are in $U_{\f{g}\oplus A}$.  The ribbon structure at $Q$ is taken to be the one induced by the ordering of the theta functions being multiplied.  It is then clear that the multiplicity of a tropical curve as in \S \ref{MultSection} is the same as the sum of the multiplicities of all the associated tropical ribbons with these new ribbon multiplicities.  In some applications, e.g., over the quantum torus algebra and over the Hall algebra, it may be preferable to use this ribbon perspective because the ribbon multiplicities will have a more natural geometric interpretation than the tropical multiplicities.
\end{rmk}

\subsection{Factored, perturbed, and asymptotic scattering diagrams}\label{ScatterFactor}

\subsubsection{Factoring and perturbing the initial scattering diagram}\label{ScatterFactorInitial}

To prove Theorems \ref{ScatTropDisks} and \ref{MainTheorem}, we extend and build off ideas from \cite[\S 1.4-\S 2]{GPS}.
\begin{dfn}\label{as}
For any scattering diagram $\f{D}$, the asymptotic scattering diagram $\f{D}_{\as}$ of $\f{D}$ is defined as follows: Every wall $(m_{\f{d}},n+\f{d},g_{\f{d}})\in \f{D}$, with $\f{d}$ denoting a rational polyhedral cone (apex at the origin) and $n \in N_{\bb{R}}$ translating this cone, is replaced by the wall $(m_{\f{d}},\f{d},g_{\f{d}})$.  
\end{dfn}

Note that
\begin{align*}
\scat(\f{D}_{\as}) = (\scat(\f{D}))_{\as}.
\end{align*}
We will use the technique from \cite{GPS} in which one factors an initial scattering diagram $\f{D}_{\In}$, deforms the factored scattering diagram by moving the supports of the initial walls, constructs $\scat$ of the deformed scattering diagram, and then takes the asymptotic scattering diagram to obtain $\scat(\f{D}_{\In})$.

Let $T$ denote the commutative polynomial ring $\bb{Z}[t_i|i\in I]$, and let $T_k:= T/ \langle t_i^{k+1}|i\in I\rangle$.   
 Let $\f{D}_{\In,T_k}$ and $\f{D}_{\In,T}$ be the initial scattering diagrams over $\f{g} \otimes T_k$ and $\f{g}\otimes T$, respectively, given by replacing each $g_{\f{d}_i}=\sum_{j\geq 1} g_{ij}$ from $\f{D}_{\In}$ with $g'_{\f{d}_i}:=\sum_{j\geq 1} t_i^jg_{ij}$.

 We will show that Theorems \ref{ScatTropDisks} and \ref{MainTheorem} hold for $\f{D}_{T_k}:=\scat(\f{D}_{\In,T_k})$ for all $k$, hence for $\f{D}_T:=\scat(\f{D}_{\In,T})$.  Taking $t_i= 1$ for each $i$ then recovers the theorems for $\f{D}=\scat(\f{D}_{\In})$.

We have an inclusion of commutative rings
\begin{align*}
T_k&\hookrightarrow T'_k:=\bb{Z}[\?{u}_{ij}|i\in I, 1\leq j \leq k]/\langle \?{u}_{ij}^2|i\in I, 1\leq j \leq k\rangle\\
         t_i&\mapsto \sum_{j=1}^k \?{u}_{ij}.
\end{align*}
Using this inclusion to work in $\f{g}\otimes T'_k$, we have 
\begin{align}\label{Rwdiq}
g'_{\f{d}_i}= \sum_{w= 1}^k t_i^jg_{ij} = \sum_{w=1}^k \sum_{\#J=w} w! g_{iw} \?{u}_{iJ},
\end{align}
where the second sum is over all subsets $J\subset \{1,\ldots,k\}$ of size $w$, and
\begin{align*}
\?{u}_{iJ}:=\prod_{j\in J} \?{u}_{ij}.
\end{align*}
Applying the equivalence from Examples \ref{equivD}(ii) in reverse and then perturbing the walls (i.e., translating the walls by some generic amount), we obtain a scattering diagram 
\begin{align}\label{PertIn0}
\?{\f{D}}_k^0:= \{(m_{\f{d}_i},\f{d}_{iwJ},w! g_{iw} \?{u}_{iJ}) | 1\leq w \leq k, J\subset \{1,\ldots,k\}, \#J=w\},
\end{align}
where $\f{d}_{iwJ}$ is some generic translation of $\f{d}_i=m_{\f{d}_i}^{\perp}$.  Note that $\scat(\?{\f{D}}_k^0)_{\as} = \f{D}_{T_k}$.

It shall be useful for us to work over a new commutative ring $\wt{T}_k$, defined by
\begin{align*}
    \wt{T}_k:=\bb{Z}[u_{iJ}|i\in I, J\subset \{1,\ldots,k\}]/\langle u_{iJ_1}u_{iJ_2}|J_1 \cap J_2 \neq \emptyset \rangle.
\end{align*}
Note that we have a surjective homomorphisms 
\begin{align}\label{pi}
  \pi:\wt{T}_k \rar T'_k, \hspace{.25 in}  u_{iJ} \mapsto \?{u}_{iJ}.
\end{align}
Let $\f{D}_k^0$ denote the initial scattering diagram over $\f{g}\otimes \wt{T}_k$ defined as in \eqref{PertIn0}, but with the factors $\?{u}_{iJ}$ replaced by $u_{iJ}$, i.e.,
\begin{align}\label{PertIn}
\f{D}_k^0:= \{(m_{\f{d}_i},\f{d}_{iwJ},w! g_{iw} u_{iJ}) | 1\leq w \leq k, J\subset \{1,\ldots,k\}, \#J=w\}.
\end{align}

\subsubsection{Constructing the consistent scattering diagram $\f{D}_k^{\infty}$}\label{Dkinfty}

As in \cite[\S 1.4]{GPS}, we now produce a sequence of scattering diagrams $\f{D}_k^0,\f{D}_k^1,\f{D}_k^2,\ldots,\f{D}_k^{k\#I-1}=:\f{D}_k^{\infty}= \scat(\f{D}_k^0)$.  Assume inductively that:
\begin{itemize}
\item[(a)] Each wall in $\f{D}_k^i$ is of the form $(m_{\f{d}},\f{d},g_{\f{d}} u_{\JJ_{\f{d}}})$, where $g_{\f{d}} \in \f{g}_{n_{\f{d}}}$ for some $n_{\f{d}}\in N^+$, $\JJ_{\f{d}}$ is a collection of pairwise-disjoint subsets of $I\times \{1,\ldots,k\}$ of the form $(i,J)$ for various $i\in I$ and $J\subset \{1,\ldots,k\}$, and
\begin{align}\label{uJ}
u_{\JJ_{\f{d}}}:=\prod_{(i,J)\in \JJ_{\f{d}}} u_{iJ}.
\end{align}
\item[(b)] There is no set $W$ of walls in $\f{D}_k^i$ of cardinality $\geq 3$ such that $\bigcap_{\f{d}\in W} \f{d}$ has codimension $\leq 2$ and $u_{\JJ_{\f{d}_1}} u_{\JJ_{\f{d}_2}}\neq 0$ for each pair of distinct walls $\f{d}_1,\f{d}_2\in W$.  
\end{itemize}
These conditions clearly hold for $\f{D}_k^0$.  To get $\f{D}_k^l$ from $\f{D}_k^{l-1}$, consider each pair $\f{d}_1,\f{d}_2\in \f{D}_k^{l-1}$ which satisfies:
\begin{itemize}
\item[(i)] $\{\f{d}_1,\f{d}_2\} \nsubseteq  \f{D}_k^{l-2}$,
\item[(ii)] $\f{d}_1\cap \f{d}_2\neq \emptyset$ has codimension $2$ and is not contained in the boundary of either $\f{d}_1$ or $\f{d}_2$,
\item[(iii)] $u_{\JJ_{\f{d}_1}}u_{\JJ_{\f{d}_2}} \neq 0$.
\end{itemize}
Given such a pair, Lemma \ref{PentagonScatter} says that adding the following new wall will result in consistency around the joint $\f{d}_1\cap \f{d}_2$ (i.e., path-ordered products around this joint will be trivial):
\begin{align}\label{Parents}
\f{d}(\f{d}_1,\f{d}_2):=(\mu((n_{\f{d}_1},m_{\f{d}_1}),(n_{\f{d}_2},m_{\f{d}_2})), (\f{d}_1\cap\f{d}_2)+\bb{R}_{\leq 0} (n_{\f{d}_1}+n_{\f{d}_2}), [g_{\f{d}_1}u_{\JJ_{\f{d}_1}},g_{\f{d}_2}u_{\JJ_{\f{d}_2}}]).
\end{align}
We now define
\begin{align*}
\f{D}_k^l:= \f{D}_k^{l-1}\cup \{\f{d}(\f{d}_1,\f{d}_2)|\f{d}_1,\f{d}_2 \mbox{ satisfying (i)-(iii) above}\}.
\end{align*}

\begin{dfn}
If $\f{d}=\f{d}(\f{d}_1,\f{d}_2)$, define $\Parents(\f{d}):=\{\f{d}_1,\f{d}_2\}$, and if $\f{d}\in \f{D}_k^0$, define $\Parents(\f{d}):=\emptyset$.  Recursively define $\Ancestors(\f{d})$ by $\Ancestors(\f{d}):=\{\f{d}\} \cup \bigcup_{\f{d}'\in \Parents(\f{d})} \Ancestors(\f{d}')$.  Define
\begin{align*}
\Leaves(\f{d}):=\{\f{d}'\in \Ancestors(\f{d})|\f{d}' \mbox{~is the support of a wall in~} \f{D}_k^0 \}.
\end{align*}
\end{dfn}

It is clear that $\f{D}_k^l$ satisfies inductive hypothesis (a).  For hypothesis (b), suppose we do have such a bad set of walls $W$.  Since the products $u_{\f{d}_1}u_{\f{d}_2}$ are nonzero for each $\f{d}_1,\f{d}_2\in W$, the sets $\Leaves(\f{d})$ for $\f{d}\in W$ must be pairwise disjoint.  Thus, slightly shifting the initial walls' supports will shift the walls in $W$ independently, and so we can avoid having this bad set $W$ by choosing the walls $\f{d}_{iwJ}$ more generically.  

For each $\JJ=\{(i,J_{ij})\subset I\times \{1,\ldots,k\}\}_{ij}$, let 
\begin{align}\label{IJ}
    I_{\JJ}:=\bigcup_{(i,J)\in \JJ} (i,J) \subset I\times \{1,\ldots,k\}.
\end{align}
Since the cardinality of $I_{\JJ_{\f{d}}}$ for the new walls increases with each step and is bounded by $k\#I$, we see that the process stabilizes with the scattering diagram  $\f{D}_k^{k\#I-1}$, so we denote $\f{D}_k^{\infty}:=\f{D}_k^{k\#I-1}$.  We check the consistency of $\f{D}_k^{\infty}$ at the end of \S \ref{TropDkinfty}.

\subsubsection{The tropical description of $\f{D}_k^{\infty}$}\label{TropDkinfty}

We will continue to use $\JJ$ to denote collections of pairwise-disjoint sets $$\JJ=\{(i,J_{ij})\subset I\times \{1,\ldots,k\}\}_{ij}$$ as in inductive hypothesis (a) of \S \ref{Dkinfty} above.  We denote $u_{\JJ}$ as in \eqref{uJ} and $I_{\JJ}$ as in \eqref{IJ}.

Now, as in \S \ref{DegIncScat}, fix a weight vector $\ww:=(\ww_i)_{i\in I}$, $\ww_i:=(w_{i1},\ldots,w_{il_i})$ with $0< w_{i1} \leq \ldots \leq w_{il_i}$.  Let $\JJ_{\ww}$ denote the set of all possible collections $\JJ$ as above, subject to the requirement that $\#J_{ij}=w_{ij}$.  Note that each $\JJ\in \JJ_{\ww}$ corresponds to a set $\f{D}_{k,\JJ}^0=\{\f{d}_{iw_{ij}J_{ij}}\}_{ij}$ of walls of $\f{D}_k^0$, and two choices of $\JJ$ correspond to the same $\f{D}_{k,\JJ}^0$ exactly if they are related by an element of $\Aut(\ww)$.  Let $\f{D}_{k,\JJ}^{\infty}$ denote the set of walls in $\f{D}_k^{\infty}$ whose leaves are precisely the walls of $\f{D}_{k,\JJ}^0$.  Note that, for  $\JJ\in \JJ_{\ww}$ and $(m_{\f{d}},\f{d},g_{\f{d}}u_{\JJ}) \in \f{D}_{k,\JJ}^{\infty}$, we must have $g_{\f{d}} \in \f{g}_{n_{\out}}$ where $n_{\out}:=\sum_{i,j} w_{ij}n_{\f{d}_i}$.

We will write $\A_{\ww,n,\JJ}$ and $\A_{\ww,\pp,Q,\JJ}$ to indicate that we are choosing the representatives of the incidence conditions $\A_{\ww,n}$ and $\A_{\ww,\pp,Q}$ so that $B_{ij}=\f{d}_{iw_{ij}J_{ij}}$.  Recall that for $n_0\in N_{\bb{R}}\setminus \{0\}$, the condition $B_{\out}$ from $\A_{\ww,n_0}$ is a generic translate of the $\bb{R}$-span of $n_0$ and $n_{\out}$.  In particular, if $n_0\notin m_{\f{d}}^{\perp}\ni n_{\out}$, then $B_{\out}\cap \f{d}$ is a ray or a line.  The following is essentially a generalization of \cite[Thm 2.4]{GPS} (which used $\f{g}=\f{h}$ and dimension $2$) to higher dimensions and more general $\f{g}$ (the two-dimensional case over $\f{g}^{\omega}_q$ is \cite[Lemmas 4.5-4.6]{FS}).

\begin{lem}\label{ScatTropBij}
For every wall $(m_{\f{d}},\f{d},g_{\f{d}}u_{\JJ})\in \f{D}_{k,\JJ}^{\infty}$, there exists a unique tropical curve $h:\Gamma \rar N_{\bb{R}}$ in $\f{T}_{\Delta_{\ww}}(\A_{\ww,n_0,\JJ})$ with $h(E_{\out})\subset \f{d}$, where $n_0$ is any element of $N_{\bb{R}}\setminus m_{\f{d}}^{\perp}$.  Furthermore, $h(\partial E_{\out}) \in \partial \f{d}$, and up to an equivalence as in Example \ref{equivD}(i) (plus possibly a positive re-scaling of $m_{\f{d}}$), we have 
\begin{align}\label{BreakCoeff}
m_{\f{d}}=m_{\Gamma}\in \?{M}  \hspace{.25 in} \mbox{and} \hspace{.25 in} g_{\f{d}}:=g_{\Gamma}\prod_{ij} (w_{ij}!)
\end{align} for $m_{\Gamma}$ and $g_{\Gamma}$ as defined in \eqref{mgGamma}.  
\end{lem}
\begin{proof}
The proof of \cite[Theorem 2.4]{GPS} is easily modified to prove this Lemma.  The idea is to construct the tropical curve by starting with the ray $\f{d}\cap B_{\out}$ and considering the endpoint $p\in \f{d}_1\cap \f{d}_2$, where $\{\f{d}_1,\f{d}_2\}=\Parents(\f{d})$.  The resulting stratum is given weight $|n_{\f{d}}|$ (the index of $n_{\f{d}}$, i.e., $n_{\f{d}}$ equals $|n_{\f{d}}|>0$ times a primitive vector), where $g_{\f{d}}\in \f{g}_{n_{\f{d}}}$.  From $p$, extend the tropical curve in the directions $n_{\f{d}_1}$ and $n_{\f{d}_2}$ with weights $|n_{\f{d}_1}|$ and $|n_{\f{d}_2}|$, respectively, until reaching the boundaries of the walls $\f{d}_1$ and $\f{d}_2$.   The balancing condition at $p$ follows easily from \eqref{Parents} and the fact that commutators in $\f{g}$ respect the $N^+$-grading.  The process is repeated for each of these branches, and continues until every branch extends to infinity in some leaf.  This gives the desired tropical curve.  The formulas for $g_{\f{d}}$, and $m_{\f{d}}$ follow easily from \eqref{Parents} and the definitions of $g_{\Gamma}$ and $m_{\Gamma}$, noting that the $\prod w_{ij}!$ factor appears because of the fact that $g_{iw}$ is multiplied by $w!$ in the definition of $\f{D}_k^0$ in \eqref{PertIn}, and similarly for the $u_{\JJ}$ factor.
\end{proof}

We now check that $\f{D}_k^{\infty}$ is consistent.  The wall-crossing automorphisms $\theta_{\f{d}_1}$ and $\theta_{\f{d}_2}$ commute for $u_{\JJ_{\f{d}_1}} u_{\JJ_{\f{d}_2}} =0$, so joints arising as the intersections of such pairs of walls will be consistent.  Also, Lemma \ref{PentagonScatter} and \eqref{Parents} ensure consistency around joints arising as intersections of pairs of walls $\f{d}_1,\f{d}_2$ satisfying (i)-(iii) above for some $l$.  The only remaining joints are those equal to the boundary of a wall $\f{d}(\f{d}_1,\f{d}_2)$ as in \eqref{Parents}.  Consider such a joint $\f{j}$.  Using Lemma \ref{ScatTropBij}, consistency about $\f{j}$ is equivalent to checking that the tropical counts $\N^{\trop}_{\ww}(n_0; m)$ (for $m^{\perp}$ parallel to the supports of the walls with boundary containing $\f{j}$ and $n_0\in N_{\bb{R}}\setminus m^{\perp}$) are at least invariant as we translate the two-dimensional wall $B_{\out}$ (i.e., from one side of $\f{j}$ to the other side of $\f{j}$).  This is easily checked using the same techniques as in \S \ref{CPSproof} below.

Thus, we have $\f{D}_k^{\infty}= \scat(\f{D}_k^0)$, as desired.  Hence,
\begin{align*}
    \f{D}_{T_k}=\pi_*(\f{D}_k^{\infty})_{\as},
\end{align*}
where the $\pi_*$ means that we apply the homomorphism $\pi:\wt{T}_k\rar T'_k$ of \eqref{pi} to each $g_{\f{d}}$.

\subsection{Proofs of the main theorems}\label{MainProofs}

We will need a certain formula for relating the $t$- and $\?{u}$-variables.  For a weight vector $\ww$ as above, let $|\ww_i|:=\sum_{j=1}^{l_i} w_{ij}$, and let $t^{\ww}=\prod_{i,j}t_i^{w_{ij}} = \prod_i t_i^{|\ww_i|}$.  Note that
\begin{align}\label{tww}
t^{|\ww_i|}= |\ww_i|! \sum_{\substack{J_i\subset \{1,\ldots,k\} \\ \#J_i=|\ww_i|}}  \?{u}_{iJ_i}.
\end{align}

Given $\JJ\in \JJ_{\ww}$, let $J_i=\bigcup_j J_{ij}$.  Note that given the sets $J_i$, there are $\prod_i \frac{|\ww_i|!}{\prod_j w_{ij}!}$ possible refinements into the sets $J_{ij}$.  We thus find
\begin{align}\label{sum_uiJ}
    \sum_{\JJ\in \JJ_{\ww}} \?{u}_{I_{\JJ}} = \prod_{i\in I}\left(\prod_{j=1}^{l_i} \left(\frac{1}{w_{ij}!}\right) \sum_{\substack{J_i\subset \{1,\ldots,k\} \\ \#J_i=|\ww_i|}} |\ww_i|!\?{u}_{iJ_i}\right).
\end{align}
Now combining \eqref{tww} and \eqref{sum_uiJ} yields
\begin{align}\label{tww_sigma}
    t^{\ww} =  \sum_{\JJ\in \JJ_{\ww}}\left(\?{u}_{I_{\JJ}} \prod_{i,j} w_{ij}!\right).
\end{align}

\subsubsection{Proof of Theorem \ref{ScatTropDisks}}

Fix $n$, $m$, and $n_0$ as in the statement of the theorem.  Let $\f{D}_k^{\infty}(n,m)$ be the set of walls in $\f{D}_k^{\infty}$ of the form $(m_{\f{d}},\f{d},g_{\f{d}})$ with $m_{\f{d}}\in \bb{Z} m$ and $g_{\f{d}}\in \f{g}_{n}^{\parallel}$.

Recall that each $\JJ\in \JJ_{\ww}$ determines a set of walls $\f{D}_{k,\JJ}^0$, and in the reverse direction, each $\f{D}_{k,\JJ}^0$ determines an orbit of $\Aut(\ww)$ in $\JJ_{\ww}$.  Similarly, $u_{\JJ}$ uniquely determines and is determined by an orbit of $\Aut(\ww)$ in $\JJ_{\ww}$.  We see that the sum from the left-hand side of \eqref{ScatTropDisksFormula}, with $\f{D}(n,m)$ there replaced by $\f{D}_k^{\infty}(n,m)$, is equal to
\begin{align}\label{sumleftside}
\sum_{\substack{k>0  \\ \ww \in \s{W}(kn)}} \sum_{\JJ\in \JJ_{\ww}/\Aut(\ww)} \left( \sum_{\substack{\f{d}\in \f{D}_{k,\JJ}^{\infty} \\ m_{\f{d}} \in \bb{R} m}} \sign(m_{\f{d}}/m) g_{\f{d}}\right)u_{\JJ}.
\end{align}
Applying Lemma \ref{ScatTropBij}, this becomes
\begin{align}\label{uij}
    \sum_{\substack{k>0  \\ \ww \in \s{W}(kn)}} \sum_{\JJ\in \JJ_{\ww}/\Aut(\ww)}\left( \sum_{\substack{\Gamma \in \f{T}_{\Delta_{\ww}}(\A_{\ww,n_0,\JJ}) \\ m_{\Gamma} \in \bb{R} m}} \sign(m_{\Gamma}/m) g_{\Gamma}  \prod_{(i,j)\in I_{\JJ}} (w_{ij}!)\right)u_{\JJ}.
\end{align}
Now, note that for each $\ww$, $(\f{D}_k^0)_{\as}$ is symmetric with respect to permuting the elements of $\JJ_{\ww}$, i.e., for $\JJ_1,\JJ_2\in \JJ_{\ww}$, swapping the supports of $\f{d}_{iw\JJ_1}$ and $\f{d}_{iw\JJ_2}$ in \eqref{PertIn} does not affect $(\f{D}_k^0)_{\as}$.  Hence, the sum in the large parentheses of \eqref{sumleftside} must be independent of $\JJ$, so we can write $\f{T}_{\Delta_{\ww}}(\A_{\ww,n_0})$ in place of $\f{T}_{\Delta_{\ww}}(\A_{\ww,n_0,\JJ})$ in \eqref{uij}.\footnote{Here it is important that we are using $u_{\JJ}$ instead of $\?{u}_{\JJ}$.  We note that this step is really what gives us the invariance of $N_{\ww}^{\trop}(n_0,m)$ in Proposition \ref{TropInvariant}.}  Now, pulling the quotient by $\Aut(\ww)$ into the sum, applying $\pi:\wt{T}_k\rar T'_k$ as in \eqref{pi}, and utilizing \eqref{tww_sigma}, the expression \eqref{uij} becomes
\begin{align*}
\sum_{\substack{k>0  \\ \ww \in \s{W}(kn)}} \sum_{\substack{\Gamma \in \f{T}_{\Delta_{\ww}}(\A_{\ww,n_0})\\ m_{\Gamma} \in \bb{R}m}} \left(\frac{1}{|\Aut(\ww)|}\right) \sign(m_{\Gamma}/m) g_{\Gamma}  t^{\ww}.
\end{align*}
The claim follows.

\qed

\subsubsection{Proof of Theorem \ref{MainTheorem}}\label{MainProof}

Fix $r\in K\cap P$ and let $n=\varphi(p)+r$.  We wish to describe the degree $n$ part (for the $P$-grading) of $\vartheta_{p_1,Q}\cdots \vartheta_{p_s,Q}:=\vartheta_{p_1,Q}^{\f{D}_k^{\infty}}\cdots \vartheta_{p_s,Q}^{\f{D}_k^{\infty}}$, in terms of tropical curve counts (using $z^{p_i}\otimes 1\in A\otimes \wt{T}_K$ as our initial monomials).  We can assume that $Q$ is far enough from the origin for the degree $n$ part of the product over $\f{D}_k^{\infty}$ to agree with that over $(\f{D}_k^{\infty})_{\as} =\scat_k(\f{D}_{\In})$.

Specifically, we want to show that the degree $n$ part of $\vartheta_{p_1,Q}\cdots \vartheta_{p_s,Q}\in  A\otimes \wt{T}_k$ is
\begin{align}\label{Mainu}
\sum_{\ww\in \s{W}_{\pp}(n)} \left[\frac{\N^{\trop}_{\ww,\pp}(Q)}{|\Aut(\ww)|} \sum_{\JJ\in \JJ_{\ww}}\left(u_{\JJ} \prod_{i,j} w_{ij}!\right) \right]. 
\end{align}
Then applying $\pi$ and using \eqref{tww_sigma}, this becomes
\begin{align*}
\sum_{\ww\in \s{W}_{\pp}(n)}  \frac{\N^{\trop}_{\ww,\pp}(Q)}{|\Aut(\ww)|}t^{\ww},
\end{align*}
and setting $t_i=1$ for each $i$ yields the desired result.

Consider a collection of broken lines $\{\gamma_k\}_k$ for $\f{D}_{k}^{\infty}$ contributing non-trivially to $\vartheta_{p_1,Q}\cdots \vartheta_{p_s,Q}$ as in \eqref{alphaform}.  For any wall $\f{d}\in \f{D}_{k}^{\infty}$ along which some $\gamma_k$ bends at a point $Q_{\f{d}}$, we glue to $\gamma_k$ the truncation $h_{Q_{\f{d}}}$ at $Q_{\f{d}}$ of a tropical curve from Lemma \ref{ScatTropBij}, a so-called Maslov index $0$ tropical disk.  Note that  $h_{Q_{\f{d}}}$ together with $\gamma_k$ (weighted by the indexes of the degrees of the attached elements of $\f{g}$ and $A$) satisfies the balancing condition at $Q_{\f{d}}$, so repeating this for every bend of $\gamma_k$ results in a tropical disk $h_{\gamma_k}$ with $1$-valent point at $Q$.  One then concatenates these tropical disks $h_{\gamma_k}$ at $Q$ for each $k=1,\ldots,s$ to obtain a tropical disk in $\f{T}'_{\Delta_{\ww,\pp}}(\A_{\ww,\pp,Q,\JJ},s-2)$ for some $\ww$ and $\JJ\in \JJ_{\ww}$.  By design (and using Lemma \ref{SimpleBreak} and  \eqref{BreakCoeff}), the corresponding product of final elements $a_{\gamma_1} \cdots a_{\gamma_s}$ as in \eqref{alphaform}  is precisely $\Mult(h_{\gamma_i})$, times a factor of $\prod_{ij} w_{ij}!$ as in Lemma \ref{ScatTropBij}.  See Figure \ref{BLTC} for an example. 

\begin{figure}[htb]
\def\svgwidth{160pt}
    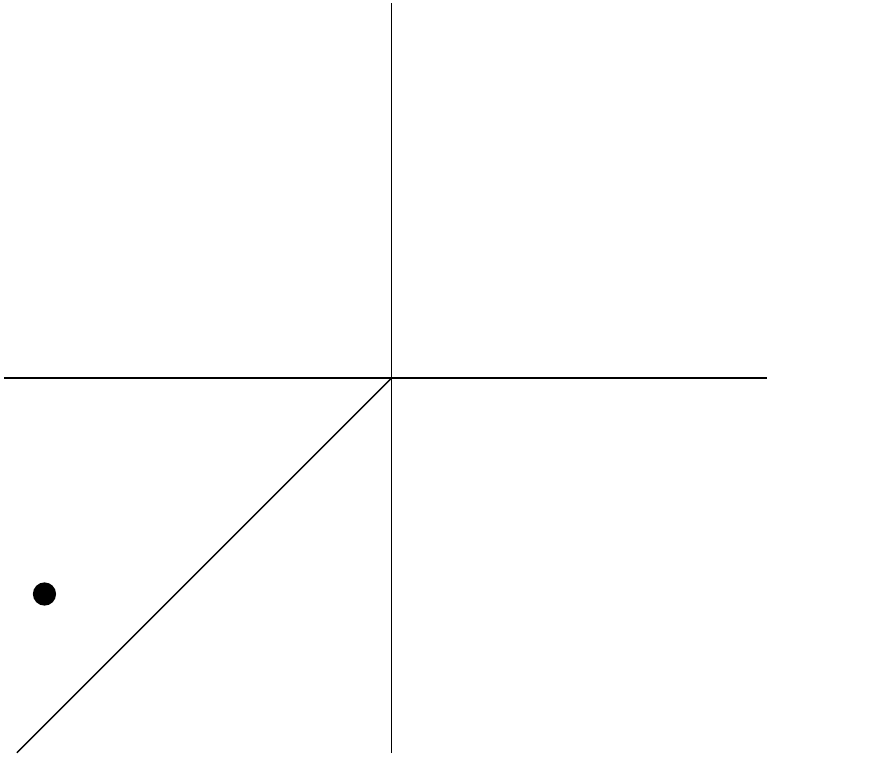
    \caption{A consistent scattering diagram over $\f{g}^{\omega}_q \otimes \wt{T}_1$ with $g_{\f{d}_1} = u_{11}z^{e_1}$, $g_{\f{d}_2} = u_{21}z^{e_2}$, and $g_{\f{d}_3}=[g_{\f{d}_1},g_{\f{d}_2}] = u_{11}u_{21}(q-q^{-1})z^{e_1+e_2}$.  The solid lines (both bold and unbold) are the supports of the walls.  The bold dashed lines are a pair of broken lines (one without any bends) contributing to the product $\vartheta_{e_1,Q}\vartheta_{e_2,Q}$.  The bold lines (dashed and undashed) form the tropical disk $\Gamma$ (which has weight $\ww=((1),(1))$) corresponding to this pair of broken lines.  The contribution of this pair of broken lines to the product is given by $[z^{e_1},[g_{\f{d}_1},g_{\f{d}_2}]]z^{e_2} = u_{11}u_{21}q^2(q-q^{-1})^2 z^{2e_1+2e_2}$.  If we view  $\Gamma$ as being in $\f{T}'_{\Delta_{\ww,\pp}}(\A_{\ww,\pp,Q},s-2)$ (using $\f{d}_1$ and $\f{d}_2$ as the incidence conditions for the legs they contain), then the corresponding contribution to $N_{\ww,\pp}^{\trop}(Q)$  is $\Mult(\Gamma)=[z^{e_1},[z^{e_1},z^{e_2}]]z^{e_2}=q^2(q-q^{-1})^2 z^{2e_1+2e_2}$.  Using Example \ref{MultExamples}(iii),  the coefficient $q^2(q-q^{-1})^2$ can be computed as a product of vertex-multiplicities: $\Mult_q(Q)=q^{\{2e_1+e_2,e_2\}}=q^2$, while the other two vertices, moving from top-right to bottom left, have multiplicities $[|\{e_1,e_2\}|]_q = [1]_q=q-q^{-1}$ and $[|\{e_1,e_1+e_2\}|]_q = [1]_q = q-q^{-1}$. \label{BLTC}}
\end{figure}

Now, for each $\ww\in \s{W}_{\pp}(n)$ and each $\JJ\in \JJ_{\ww}$, we can apply the above computation to all broken lines for $\f{D}_k^{\infty}$ whose corresponding tropical disk lives in $\f{T}'_{\Delta_{\ww,\pp}}(\A_{\ww,\pp,Q,\JJ},s-2)$.  Summing over all $\JJ\in \JJ_{\ww}$ and applying the same tricks as in the above proof of Theorem \ref{ScatTropDisks} (e.g., noting that for each $\ww$, the result must be symmetric with respect to permuting the $\JJ$'s in $\JJ_{\ww}$), one finds that the sum of the final monomials of all the relevant broken lines indeed yields \eqref{Mainu}.  

\qed

See Figure \ref{BLTC} for an example of the above construction with $k=1$.  We note that the complexity of the scattering diagram does increase quickly as soon as the $k$ in $\wt{T}_k$ is increased.  See \cite[Figure 1.3]{GPS} for an illustration when $k=2$.

\subsubsection{Proof of Theorem \ref{CPS} (the refined \cite{CPS} result)}\label{CPSproof}

We want to show that $\vartheta_{p,Q}$ for different generic values of $Q$ are related by path-ordered products.  Note that it suffices to prove this for the scattering diagrams $\f{D}_{k}^{\infty}$ as described in Lemma \ref{ScatTropBij}.  Recall from the proof of Theorem \ref{MainTheorem} in \S \ref{MainProof} above that the broken lines contributing to $\vartheta_{p,Q}$ (for the scattering diagram $\f{D}_{k}^{\infty}$) correspond bijectively to tropical disks in $\f{T}'_{\Delta_{\ww,p}}(\A_{\ww,p,Q,\JJ},-1)$ for some $\ww$ and some $\Aut(\ww)$-orbit of $\JJ$ in $\JJ_{\ww}$.  Furthermore, for a broken line $\gamma$ and corresponding tropical disk $h_{\gamma}$, we have $a_{\gamma}=\Mult(h_{\gamma})$. As we move $Q$, there are two issues that could result in changes to the types of broken lines contributing to $\vartheta_{p,Q}$.  The most obvious is that $Q$ may move across a wall $\f{d}$ of $\f{D}_{k}^{\infty}$, resulting in the possible gluing or losing of a Maslov index $0$ tropical disk associated to the wall.  By Lemma \ref{ScatTropBij} and Lemma \ref{SimpleBreak}, the resulting changes to the tropical disk counts are exactly accounted for by the wall-crossing automorphisms.

There is one other way that moving $Q$ might affect the types of tropical disks being enumerated.  Namely, as we translate $Q$ and correspondingly deform $\Gamma\in \f{T}'_{\Delta_{\ww,p}}(\A_{\ww,p,Q,\JJ},-1)$, an edge of $\Gamma$ might collapse to have length $0$, resulting in a $4$-valent vertex.  Let $Q_0$ be a point for which some $\Gamma_0$ in $\f{T}'_{\Delta_{\ww,p}}(\A_{\ww,p,Q,\JJ},-1)$ has a $4$-valent vertex, and assume $Q_0$ is generic among such points.  Then there is some neighborhood $U$ of $Q_0$ and some affine hyperplane $H$ containing $Q_0$ such that, for each $Q\in H\cap U$, there is a unique tropical disk of the same type as $\Gamma_0$ in $\f{T}'_{\Delta_{\ww,p}}(\A_{\ww,p,Q,\JJ},-1)$.

Let $E_1,E_2,E_3$ be the edges flowing into the $4$-valent vertex $V$ for $\Gamma_0$ as above, and let $E_4$ be the outward-flowing edge, flowing towards $Q_{\out}$, cf. Figure \ref{TropWall}(b).  For $Q$ in one component of $U\setminus (H\cap U)$, there is exactly one way to extend the $4$-valent vertex into a compact edge to yield a tropical disk $\Gamma \in \f{T}'_{\Delta_{\ww,p}}(\A_{\ww,p,Q,\JJ},-1)$, say, with $E_2$ and $E_3$ meeting first, cf. Figure \ref{TropWall}(a).  For $Q$ on the other side of $U\setminus (H\cap U)$, there are either one or two ways to insert a compact edge yielding tropical disks in $\f{T}'_{\Delta_{\ww,p}}(\A_{\ww,p,Q,\JJ},-1)$, either with $E_1$ and $E_2$ meeting first, or with $E_1$ and $E_3$ meeting first, cf. Figure \ref{TropWall}(c)(d).  We wish to show first that if one of the two tropical curve types of Figure \ref{TropWall}(c)(d) does not occur, then the tropical multiplicity associated to that type is $0$.  We will then show that the multiplicity associated to the tropical curve type of Figure \ref{TropWall}(a) is the sum of the multiplicities associated to the tropical curve types of Figure \ref{TropWall}(c)(d).\footnote{We note this strategy for proving invariance of tropical counts was first employed in the genus $0$ cases of \cite{GM}.}  See Figure \ref{BJ1} for an illustration of how this wall-crossing the tropical moduli space arises as a result of a broken line crossing a joint of $\f{D}_k^{\infty}$.

\begin{figure}[htb] 
    \begin{tabular}{c c c c}
    \def\svgwidth{90pt}
    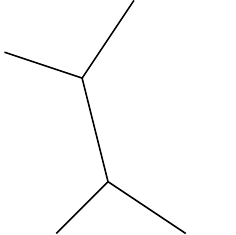
    & \def\svgwidth{80pt} 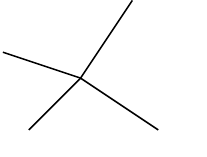 
    &  \def\svgwidth{120pt} 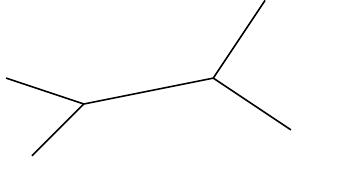 
    &  \def\svgwidth{84pt} 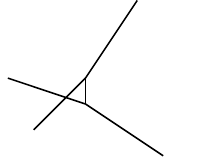 \\ \\
		(a) & (b)  & (c)  & (d)
    \end{tabular}
\caption{Tropical wall crossing.  Locally in the space of choices for the incidence conditions, there is a codimensions one ``wall'' of non-generic choices resulting in a $4$-valent vertex as in (b).  One one side of this wall, there is a single tropical curve type (a) satisfying the deformed conditions.  On the other side, there are up to two types (c) and (d).\label{TropWall}}
\end{figure}

\begin{figure}[htb]
    \begin{tabular}{c c c c}
\def\svgwidth{90pt}
    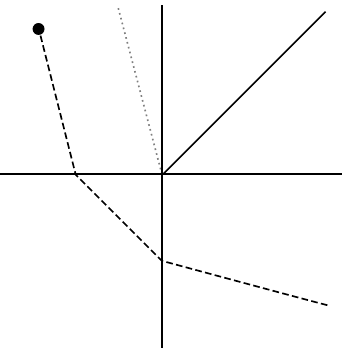 
    &\def\svgwidth{90 pt}
    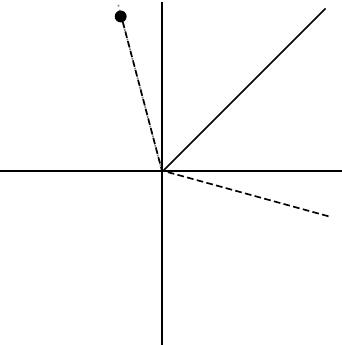 
    & \def\svgwidth{90pt} 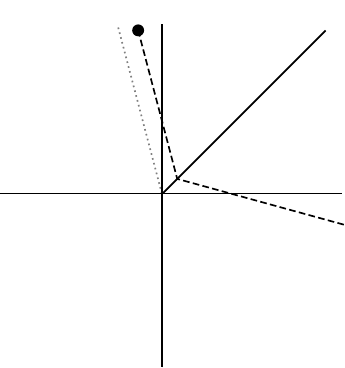
    &  \def\svgwidth{90 pt} 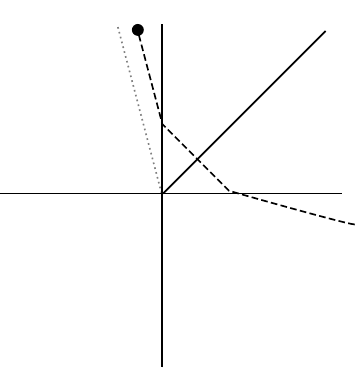 \\ \\
%    (a) \hspace{.25 in} & (b)  & \hspace{.25 in} (c) \\ \\ \\
    \def\svgwidth{90pt}
    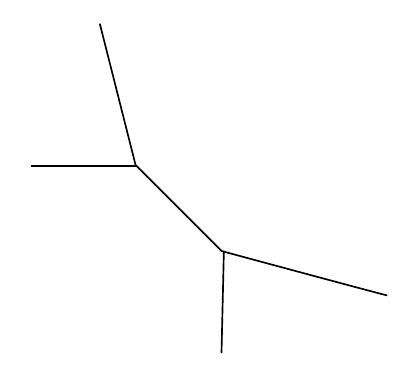
    & \def\svgwidth{90pt} 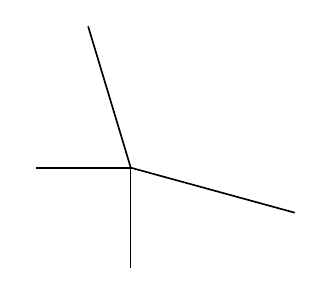 
    &\def\svgwidth{90pt}
    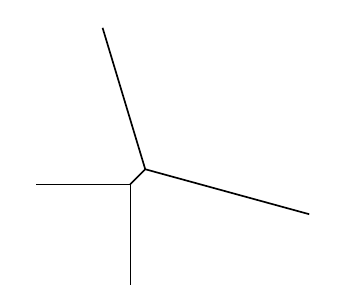
    &  \def\svgwidth{90pt} 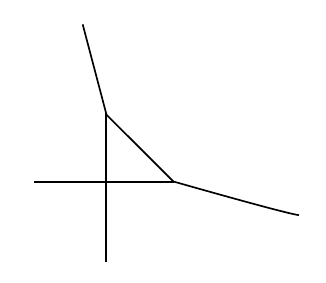 \\ \\
    (a)  & (b) &  (c) & (d)
    \end{tabular}
            \caption{Broken lines (the dashed segments) moving past a joint of a scattering diagram (the solid rays), and the corresponding transition in tropical disk types.  When $Q$ moves onto a certain local hyperplane $H$ (the opaque dotted ray), broken lines $\gamma$ ending at $Q$ collide with a joint (b, top), resulting in a tropical disk with a $4$-valent vertex (b, bottom).  For $Q$ on one side of $H$, there is one possibility for the additional straight segment of $\gamma$ (a, top), resulting in one tropical disk type (a, bottom).  On the other side of $H$, there are up to two possibilities for the new edge of the broken line (c, top) and (d, top), resulting in two corresponding tropical disk types (c, bottom) and (d, bottom), respectively.  Note that the bottom row here corresponds to the tropical disks of Figure \ref{TropWall}. \label{BJ1}}
\end{figure}

The situation in which one of the two types from Figure \ref{TropWall}(c)(d), say, the type from (d), does not occur arises under the following circumstances: consider the four tropical disk-types associated to the connected components of the tropical curve Figure \ref{TropWall}(b) with its vertex removed.  The incidence conditions on these components force each $E_i$ to live in some affine space $B_i$.  Then the tropical curve type from (d) does not occur if either $B_1$ and $B_3$ are not transverse, or $B_2$ and $B_4$ are not transverse.

For $i=1,2,3$, if $E_i$ has an element $g_{E_i}\in \f{g}^{\parallel}_{n_{E_i},m_{E_i}}$ associated to it in the definition of $\Mult(\Gamma)$, then $B_i=m_{E_i}^{\perp}$, while if $E_i$ has an element of $A$ associated to it, then $B_i$ is all of $N_{\bb{R}}$.  So the spaces $B_1$ and $B_3$ will automatically be transverse if either $E_1$ or $E_3$ has an element of $A$ associated to it.  On the other hand, if $B_1$ and $B_3$ are each associated with elements of $\f{g}^{\parallel}_{n_{E_i},m_{E_i}}$, then $B_1$ and $B_3$ will only fail to be transverse if $m_{E_1}$ and $m_{E_3}$ are parallel.  But then $n_{E_1}$ and $n_{E_3}$ are both contained in $m_{E_1}^{\perp}=m_{E_3}^{\perp}$, and so since $g_{E_i}\in \f{g}_{n_{E_i},m_{E_i}}^{\parallel}$, we have $[g_{E_1},g_{E_3}]=0$.  So then this missing type has multiplicity $0$ and does not affect the counts.

Now suppose that $B_2$ and $B_4$ fail to be transverse.  Let $E_0$ denote the compact edge in Figure \ref{TropWall}(d).  As above, it must be the case that $B_2$ has an element of $\f{g}_{n_{E_2},m_{E_2}}$ associated to it, not an element of $A$, and so $B_2$ is parallel to $m_{E_2}^{\perp}$.  On the other hand, $B_4$ is parallel to $\bb{R}n_{E_4}$, and so non-transversailty means $n_{E_4}\in m_{E_2}^{\perp}$.  But then the balancing condition forces $n_{E_0} \in m_{E_2^{\perp}}$ as well.  Since $g_{E_2} \in \f{g}_{n_{E_2},m_{E_2}}^{\parallel}$, we now have $g_{E_2}\cdot a_{g_{E_0}}=0$.  Thus, these missing tropical disk types have multiplicity $0$, as desired.

So now we may indeed assume that each of the $3$ possible tropical types of Figure \ref{TropWall}(a,c,d) occurs near the wall.  For convenience, let us now view the elements attached to the edges of $\Gamma$ not as living in $\f{g}$ or $A$, but instead as living in $\f{g}\oplus A$, always denoting the element associated to an edge $E$ by $g_E$.

For the side of $H$ associated to (a), the element $g_{E_4}\in \f{g}\oplus A$ is, up to sign, given by $[g_{E_1},[g_{E_2},g_{E_3}]]$.  For the other side of the wall, the $g_{E_4}$'s corresponding to the two possible types Figure \ref{TropWall} (c) and (d) are, up to signs, $[[g_{E_1},g_{E_2}],g_{E_3}]$ and $[g_{E_2},[g_{E_1},g_{E_3}]]$, respectively.  So equality of the tropical counts on the two sides of $H$ comes down to checking that
\begin{align}\label{Jacobi}
    \pm [g_{E_1},[g_{E_2},g_{E_3}]] = \pm [[g_{E_1},g_{E_2}],g_{E_3}] \pm [g_{E_2},[g_{E_1},g_{E_3}]],
\end{align}
where the signs have yet to be addressed.  If we can show that the signs of nonzero terms in \eqref{Jacobi} are either all positive or all negative, then the equality follows from the Jacobi identity.  We may of course assume that the terms of \eqref{Jacobi} are not all $0$, since this case is trivial.

Note that the signs in \eqref{Jacobi} are independent of the specific choice of $\f{g}$ and $A$, instead being determined entirely by the vectors $m_E$ and $n_E$ associated to the edges.  Thus, it suffices to check the case of the tropical vertex group $\f{h}$ as in Example \ref{gegs}(i).  In this case, Theorem \ref{CPS} is known to hold by \cite{CPS}, so all the signs of nonzero terms in \eqref{Jacobi} must be the same.

Now, in the tropical vertex group setting, for $i_1,i_2,i_3$ the distinct elements of $\{1,2,3\}$ in some order, we have that $[g_{E_{i_1}},[g_{E_{i_2}},g_{E_{i_3}}]]$ is nonzero if and only if $B_{i_2}$ and $B_{i_3}$ are transverse and $B_{i_1}$ and $B_4$ are transverse (here we use the assumption that not all terms of \eqref{Jacobi} are $0$ to ensure that the vanishing of powers of the $u_{ij}$'s does not cause $0$ multiplicity).  Furthermore, as we saw in our transversality arguments above, non-transversality of $B_{i_2}$ and $B_{i_3}$ or of $B_{i_1}$ and $B_4$ forces the bracket to be $0$ for any choice of $\f{g}$ and $A$.  Thus, the signs of all nonzero terms in \eqref{Jacobi} agreeing in the tropical vertex group setting is sufficient.  This completes the proof.

\qed

~

We note that the above proof used the fact that Theorem \ref{CPS} is known to hold over the tropical vertex group, but this can be avoided, either by tediously checking the signs of \eqref{Jacobi} in several different cases, or by using the results of \cite{MRudMult} to relate the multiplicities in the tropical vertex group setting to tropical Gromov-Witten counts that are known to be invariant.

\section{Cluster varieties and Frobenius maps}\label{ClusterSection}

In this section we briefly explain how to get the initial scattering diagrams used for constructing theta functions on cluster varieties, including both the classical and quantum versions.  We then use Theorem \ref{MainTheorem} to prove Fock and Goncharov's conjectures \cite[\S 4]{FG1} on symmetries of theta functions with respect to certain Frobenius automorphisms (not to be confused with Gross-Hacking-Keel's Frobenius structure conjecture).

\subsection{Seeds}\label{SeedsSection}

As in \cite[\S 1.2]{FG1}, a seed is a collection of data \begin{align}\label{S}
    S=\{L,I,E:=\{e_{i}\}_{i\in I},F,\{\cdot,\cdot\},\{d_{i}\}_{i\in I}\},
\end{align} where $L$ is a finitely generated free Abelian group, $I$ is a finite index set, $E$ is a basis for $N$ indexed by $I$, $F$ is a subset of $I$,  $\{\cdot,\cdot\}$ is a skew-symmetric $\bb{Q}$-valued bilinear form, and the $d_{i}$'s are positive rational numbers such that $d_j\{e_i,e_j\}$ is in $\bb{Z}$ whenever $i$ and $j$ are not both in $F$.  One considers the bilinear form $(\cdot,\cdot)$ defined by
\begin{align*}
    (e_{i},e_{j}) := d_j\{e_i,e_j \}.
\end{align*}
  One calls $e_i$ a frozen vector if $i \in F$.  We let $\pi_1$ and $\pi_2$ be the maps $L\rar L^*$ defined by $n\mapsto (n,\cdot)$ and $n\mapsto (\cdot,n)$, respectively.  The reader should notice the resemblance of this setup to that of Examples \ref{InExamples}.  If the seed $S$ is not clear from context, we may write the data with subscripts $S$ to clarify, e.g., $S=\{L_S,I_S,E_S=\{e_{S,i}\},F_S,\{\cdot,\cdot\}_S,\{d_{S,i}\}\}$.

Given $S$ as above, the Langland's dual seed $S^{\vee}$ has the same $L$, $I$, $E$, and $F$ as $S$, but $\{\cdot,\cdot\}$ is replaced with the form $\{\cdot,\cdot\}^{\vee}$ defined by $\{e_i,e_j\}^{\vee}:=d_id_j\{e_i,e_j\}$, and for each $i\in I$, $d_i$ is replaced by $d_i^{\vee}:=\frac{1}{d_i}$.  The main effect of this is that the form $(\cdot,\cdot)^{\vee}$ defined by $(e_i,e_j)^{\vee}=d_j^{\vee}\{e_i,e_j\}^{\vee}$ is the negative transpose of $(\cdot,\cdot)$, so $\pi_1^{\vee}=-\pi_2$ and $\pi_2^{\vee}=-\pi_1$.

We refer to \cite[\S 1.2]{FG1} for the definitions of the spaces $\s{A}$ and $\s{X}$ associated to the seed $S$.  For the quantum version $\s{X}_q$ of the $\s{X}$-space, cf. \cite[\S 3]{FG1}, and for the quantum version $\s{A}_q$ of $\s{A}$, cf. \cite{BZ} (alternatively, the reader may confer v2 of this article on arXiv).

Fix a seed $S$ as in \eqref{S}.  In the construction of the theta functions used in \cite{GHKK}, one  works not with $S$, but with the seed $S^{\prin}$ defined as follows:
\begin{itemize}[noitemsep]
\item $L_{S^{\prin}}:=L\oplus L^*$.
\item $I_{S^{\prin}}$ is the disjoint union of two copies of $I$.  We will call them $I_1$ and $I_2$ to distinguish between them.
\item $E_{S^{\prin}}:=\{(e_i,0)|i\in I_1\} \cup \{(0,e_i^*)|i \in I_2\}$
\item $F_{S^{\prin}}:=F_1\cup I_2$, where $F_1$ is $F$ viewed as a subset of $I_1$.
\item $\{ (n_1,m_1), (n_2,m_2)\}_{S^{\prin}} : = \{n_1,n_2\} + m_2(n_1) - m_1(n_2)$.
\item The $d_i$'s are the same as before (viewing $i$ in $I_1$ or $I_2$ as an element of $I$).
\end{itemize}

\subsection{The initial cluster scattering diagrams}\label{InClusterScat}

The theta functions in \cite{GHKK} are constructed first for $\s{A}^{\prin}$, and then certain restrictions of subsets of these theta functions yield the theta functions on $\s{A}$ and $\s{X}$ (cf. their Section 7.2).\footnote{In fact, since the theta functions are, in general, formal, they are more accurately defined only on various formal versions of these spaces.  We will ignore this issue here as it does not matter for our purposes.}  We will briefly give the initial scattering diagrams for directly constructing theta functions for $\s{X}$ and (if a ``compatible pair'' exists) for $\s{A}$.  Theta functions for $\s{A}^{\prin}$ can then be constructed by applying the $\s{A}$-case to the seed $S^{\prin}$.  Similarly, we will give the initial scattering diagrams for constructing the quantum theta functions on $\s{X}_q$ and $\s{A}_q$.

\subsubsection{Theta functions on $\s{X}$ and $\s{X}_q$}

The initial scattering diagram for constructing theta functions on $\s{X}$ is defined using Example \ref{InExamples}(ii) in the obvious way.  That is, we take $N=L$ with $E$, $I$, $F$, $\{\cdot,\cdot\}$, and $\{d_i\}$ as for the seed $S$.  Then, using the equivalence of Example \ref{equivD}(i), the resulting initial scattering diagram is
\begin{align*}
    \f{D}^{\s{X}}_{\In} &:= \{(\pi_2(e_i),\pi_2(e_i)^{\perp},\log(1+z^{e_i})\partial_{\pi_2(e_i)})\}_{i\in I\setminus F}.
    \end{align*}
We note a couple alternative ways to express this.  In terms of the Langland's dual seed $S^{\vee}$ and using the dilogarithm description of \eqref{Li2}, and applying the equivalence of Example \ref{equivD}(i) again, we can write the above scattering diagram as
\begin{align*}
    \f{D}^{\s{X}}_{\In} &= \{(\pi^{\vee}_1(e_i),\pi^{\vee}_1(e_i)^{\perp},-d_i\Li_2(-z^{e_i}))\}_{i\in I\setminus F}.
    \end{align*}
On the other hand, in terms of the version of scattering diagrams sketched in Remark \ref{RmkScatBar}(ii), we would write $\f{D}^{\s{X}}_{\In}$ as $\{(e_i,e_i^{\perp},\log(1+z^{e_i})\partial_{\pi_2(e_i)})\}_{i\in I\setminus F}$.

Similarly, the initial scattering diagram for the quantization $\s{X}_q$ is given as in \eqref{qIn} by
\begin{align*}
   \f{D}_{\In}^{\s{X}_q} := \{(\pi^{\vee}_1(e_i), \pi^{\vee}_1(e_i)^{\perp},-\Li_2(-z^{e_i};q^{1/d_i})\}.
\end{align*}
where we recall that $\Li_2(x;q):=\sum_{k=1}^{\infty} \frac{x^k}{k[k]_{q}}$ and $[k]_q:=q^k-q^{-k}$.
 We note that the construction of this quantum initial scattering diagram was outlined in \cite[arXiv v1, Construction 1.31]{GHKK}.

\subsubsection{Theta functions on $\s{A}$ and $\s{A}^{\prin}$, and on $\s{A}_q$ and $\s{A}_q^{\prin}$}
To construct the initial scattering diagram for $\s{A}$, we will use what \cite{BZ} calls a compatible pair, i.e., a skew-symmetric bilinear form $\Lambda$ on $L^*$ such that
\begin{align*}
\begin{array}{c l}
\Lambda(\pi_1(e_i),\cdot) = \frac{1}{d_i} e_i  &\mbox{ for each $i\in I\setminus F$.}
\end{array}
\end{align*}
(The other part of the ``pair'' is the data of the matrix $B$ for $(\cdot,\cdot)$ with respect to the basis $E$).  One sees that the existence of such a $\Lambda$ is equivalent to the condition that the restriction of $p_1$ to the span of $\{e_i\}_{i\in I\setminus F}$ is injective (called the ``Injectivity Assumption'' in \cite[\S 1]{GHKK}).  In particular, this is always the case for $S^{\prin}$ because $(\cdot,\cdot)_{\prin}$ is unimodular.

We now fix such a $\Lambda$, assuming one exists.  We then apply Example \ref{InExamples}(ii) to the data $N=L_S^*$, $I=I_S$, $F=F_S$, $E=\{\pi_1(e_{S,i})\}_{i\in I_S}$, $\omega=\Lambda$, and $d_i=d_{S,i}$ for each $i\in I$.  This yields the desired initial scattering diagram:
\begin{align*}
    \f{D}^{\s{A}}_{\In} = \{e_i,e_i^{\perp},\log(1+z^{\pi_1(e_i)})\partial_{e_i}\}.
\end{align*}

Similarly, the initial quantum scattering diagram is obtained by applying Example \ref{InExamples}(iii) to this data, thus yielding
\begin{align*}
    \f{D}^{\s{A}_q}_{\In} = \{e_i,e_i^{\perp},-\Li_2(-z^{\pi_1(e_i)};q^{1/d_i})\},
\end{align*}
Here, $-\Li_2(-z^{\pi_1(e_i)};q^{1/d_i})$ lives in the completion of the quantum torus algebra $\f{g}^{\Lambda}_q$ associated to $L^*$ and $\Lambda$ via the construction in Example \ref{gegs}(iii).

The initial scattering diagrams for $\s{A}^{\prin}$ and $\s{A}^{\prin}_q$ are constructed in the same way but using $S^{\prin}$ in place of $S$.

\subsection{The Frobenius maps}\label{Application}

Prior to the definition of the theta functions in \cite{GHKK}, \cite[\S4]{FG1} predicted their existence and conjectured several properties they should satisfy.  Among these properties are certain symmetries under a (quantum) Frobenius automorphism, predicted there for theta functions on the $\s{X}$-space, but proven here to also hold for the $\s{A}$-spaces.

First, we will need the following, which is little more than a restatement of \cite[Thm 1.13]{GHKK}.  
\begin{thm}[\cite{GHKK}, Thm 1.13]
Let $\f{D}_{\In}$ be an initial scattering diagram over a Poisson torus algebra as in \eqref{DInPoisson} (this includes each $\f{D}_{\In}^{\s{X}}$ and $\f{D}_{\In}^{\s{A}}$ of \S \ref{InClusterScat}).  Let $\f{D}:=\scat(\f{D}_{\In})$.  Then $\f{D}$ is equivalent to a scattering diagram $\f{D}'$ such that, for any wall $\f{d}\in \f{D}'$, and for any $u\in P$, crossing from the side of $\f{d}$ containing $u$ to the side not containing $u$ acts on $z^u$ via
\begin{align}\label{zu}
    z^u \mapsto z^u(1+z^n)^{cd_i|\{n,u\}|}
\end{align}
for some $n\in N^+$ and some positive integer $c$.  Consequently, every theta function constructed from broken lines for $\f{D}$ has non-negative integer coefficients.
\end{thm}

In particular, the integrality allows us to consider the coefficients modulo a prime $p$.  In \cite[\S 4.1, Equation 66]{FG1}, Fock and Goncharov predicted the $\s{X}$-space cases of the following theorem, which they called the Frobenius Conjecture:
\begin{thm}[Frobenius Conjecture, classical version]\label{FrobConjClass}
Consider $\f{D}_{\In}$ as in \eqref{DInPoisson} and $\f{D}=\scat(\f{D}_{\In})$. For any prime $p$ and any $u\in P$, the theta functions constructed from $\f{D}$ satisfy
\begin{align*}
    \vartheta_u^p \equiv \vartheta_{pu} \mbox{~(mod $p$)}.
\end{align*}
\end{thm}
\begin{proof}
We work with a representative $\f{D}'$ of the equivalence class of $\f{D}$ as in \eqref{zu}.  Consider broken lines with attached monomials $az^v$ and $az^{pv}$ ($a\in \bb{Z}$, $v\in P$) crossing a wall of $\f{D}'$ with associated wall-crossing automorphism $\nu$.  By \eqref{zu}, $\nu(az^v)=az^v(1+z^n)^k$ for some $n\in N^+$, $k\in \bb{Z}_{\geq 0}$.  Similarly, $\nu(az^{pv})=az^{pv}(1+z^n)^{pk}$.  By the freshman's dream and Fermat's little theorem, we see that $\nu(az^{pv}) \equiv \nu(az^v)^p$ (mod $p$).  It follows that the broken lines contributing to $\vartheta_{pu,Q}$ in characteristic $p$ are the same as the broken lines contributing to $\vartheta_{u,Q}$ in characteristic $p$, except that the attached monomials for broken lines contributing to $\vartheta_{pu,Q}$ are the $p$-th powers of the corresponding attached monomials for $\vartheta_{u,Q}$.  The result now follows by applying the freshman's dream to $\vartheta_u^p$.
\end{proof}

\cite{FG1} also predicted the following quantum version of the Frobenius Conjecture, their Conjecture 4.8.6.  First we introduce some notation.  Denote by $\vartheta_{u,Q}(z^n)=\sum c_nz^n\in \wh{A}=R_q\llb N^{\oplus}\rrb_P$ the Laurent series expansion of $\vartheta_{u,Q}$ in terms of monomials $z^n$, $n\in P$.  Then for $k\in \bb{Z}_{>0}$, denote $\vartheta_{u,Q}(z^{kn}):=\sum c_n z^{kn}$, the series obtained by multiplying each exponent by $k$.  When we want to specify that we are taking a certain limit for $q$, we will write this value in the subscript, as in $\vartheta_{u,Q,q}$.
\begin{thm}[Frobenius Conjecture, quantum version]\label{FrobConjQuant}
Consider theta functions with respect to $\f{D}=\scat(\f{D}_{\In})$ for $\f{D}_{\In}$ as in \eqref{qIn} (so this includes $\f{D}_{\In}$ equal to any $\f{D}_{\In}^{\s{X}_q}$ or $\f{D}_{\In}^{\s{A}_q}$).  Suppose $q$ and each $q^{1/d_i}$ are primitive $k$-th roots of unity for a positive odd integer $k$.  Then for any $u \in P$, we have
\begin{align*}
\vartheta_{ku,Q,q}(z^n) = \vartheta_{u,Q,1}(z^{kn})
\end{align*}
\end{thm}
The map $\vartheta_{u,Q,q}(z^n)\mapsto \vartheta_{u,Q,1}(z^{kn})$ is what \cite{FG1} calls the quantum Frobenius map.  The case of quantum cluster varieties from surfaces is \cite[Theorem 1.2.6]{AK}, assuming that their canonical bases turn out to equal the theta bases.  Since we do not have a version of \eqref{zu} in the quantum setting, the methods from the proof of Theorem \ref{FrobConjClass} will not be useful here.  Instead, we make use of Theorem \ref{MainTheorem}.
\begin{proof}
Consider a tropical disk making a nonzero contribution to \eqref{MainTrop} for $\vartheta_{ku,Q}$.  I.e., we consider a tropical disk $\Gamma$ contributing to some $\N^{\trop}_{\ww,\pp}(Q)$ in
\begin{align*}
\sum_{\ww\in \s{W}_{\pp}(ku)}  \frac{\N^{\trop}_{\ww,\pp}(Q)}{|\Aut(\ww)|}.
\end{align*}
  Let $\ww(\Gamma)$ denote the corresponding weight vector.  Using the description of $\Mult(\Gamma)$ given in \eqref{Multq}, we see that the contribution of $\Gamma$ is $z^{n_{\out}}$ times
\begin{align}\label{Contribution}
\left(\prod_{V\in \Gamma^{[0]}\setminus Q} [\Mult_{\Gamma}(V)]_q\right) \left(\prod_{w_{ij}\in \ww(\Gamma)} \frac{(-1)^{w_{ij}-1}}{w_{ij}[w_{ij}/d_i]_q} \right)\frac{1}{|\Aut(\ww(\Gamma))|}.
\end{align}
Here, each factor $\frac{(-1)^{w_{ij}-1}}{w_{ij}[w_{ij}/d_i]_q}$, which we will denote as $R_{w_{ij},d_i;q}$, arises as the $z^{w_{ij}e_i}$-coefficient in the quantum dilogarithm $-\Li_2(-z^{e_i};q^{1/d_i})$, so this product is the factor called $a_{\ww}$ in \eqref{Multq}.

The initial segment of the broken line corresponding to $\Gamma$ has weight a multiple of $k$.  We show by induction that the same is true for every edge of $\Gamma$.  Let $S$ be a maximal subset of $\Gamma\setminus Q_{\out}$ such that each edge $E\in S$ has weight a multiple of $k$ and the closure of $\Gamma\setminus S$ in $\Gamma$ is connected.  Suppose $S$ is not all of $\Gamma\setminus Q_{\out}$.  Then $S$ is a union of trees that each contain exactly $1$ univalent vertex, with the remainder of the vertices being trivalent.  To see this, note that there are no bivalent vertices in these trees because if two edges containing a vertex have weights a multiple of $k$, then the third does too.  Also, if there were more than one univalent vertex, then the closure of $\Gamma \setminus S$ would not be connected.  On the other hand, the vertex of a component of $S$ whose distance from $Q_{\out}$ is minimal must be univalant.

Now, the number of vertices of $S$ is equal to the number of undbounded edges in $S$.  Since $S$ contains the unbounded edge corresponding to the initial direction of the broken line, this means that $\Gamma$ has more vertices of multiplicity a multiple of $k$ than there are elements of $\ww(\Gamma)$ that are a multiple of $k$.  But for $\zeta$ a primitive $k$-th root of unity, $\lim_{q\rar \zeta} \frac{[a]_q}{[b]_q}=0$ if $a$ is a multiple of $k$ and $b$ is not, and the limit equals a finite nonzero number (see below) if both $a$ and $b$ are multiples of $k$.  Hence, the contribution of such a curve would be $0$.  So every edge of $\Gamma$ must have been weight a multiple of $k$.

We now see that a tropical curve contributes to $\vartheta_{ku,Q,q}$ if and only if it can be obtained by taking a tropical curve contributing to $\vartheta_{u,Q,1}$ and multiplying each weight by $k$.  This multiplication of each weight by $k$ takes each vertex multiplicity $[a]_q$ to $[k^2a]_q$, and each $R_{w_{ij},d_i;q}=\frac{(-1)^{w_{ij}-1}}{w_{ij}[w_{ij}/d_i]_q}$ to $R_{kw_{ij},d_i;q}=\frac{(-1)^{kw_{ij}-1}}{kw_{ij}[kw_{ij}/d_i]_q}$.  The number of trivalent vertices of $\Gamma$ is the same as the number of weights $w_{ij}$ in $\ww(\Gamma)$, so we can pair the trivalent vertices up with the $w_{ij}$'s and compute, for $\zeta$ a primitive $k$-th root of unity,
\begin{align*}
    \lim_{q\rar \zeta} [k^2a]_qR_{kw_{ij},d_i;q} &= \lim_{q\rar \zeta} \frac{(q^{k^2a}-q^{-k^2a})(-1)^{kw_{ij}-1}}{kw_{ij}(q^{kw_{ij}/d_i}-q^{-kw_{ij}/d_i})} \\
    &=\frac{(-1)^{kw_{ij}-1}}{kw_{ij}}\lim_{q\rar \zeta} q^{kw_{ij}/d_i-k^2a}\frac{(q^{2k^2a}-1)}{(q^{2kw_{ij}/d_i}-1)}.
\end{align*}
Since $q^{1/d_i}$ was also assumed to be a primitive $k$-th root of unity, $\lim_{q\rar \zeta} q^{kw_{ij}/d_i-k^2a} = 1$.  Using this and L'Hospital's rule, the above now further simplifies to
\begin{align*}
    \frac{(-1)^{kw_{ij}-1}}{kw_{ij}}\lim_{q\rar \zeta} \frac{2k^2aq^{2k^2a-1}}{(2kw_{ij}/d_i) q^{2kw_{ij}/d_i-1}} = 
    \frac{ad_i(-1)^{kw_{ij}-1}}{w_{ij}^2} = \frac{ad_i(-1)^{w_{ij}-1}}{w_{ij}^2},
\end{align*}
where the last equality used the assumption that $k$ is odd.  This is equal to $\lim_{q\rar 1} [a]_q R_{w_{ij},d_i;q}$, and the result follows from applying this to every such vertex-weight pair.
\end{proof}

\bibliographystyle{alpha}  % Here the bibliography 		     %
\bibliography{mandel}        % is inserted.			     %
\index{Bibliography@\emph{Bibliography}}%

\end{document}

%% file: m1m2.pdf_tex
%% Creator: Inkscape inkscape 0.92.1, www.inkscape.org
%% Accompanies image file 'm1m2.pdf' (pdf, eps, ps)
\begingroup%
  \makeatletter%
  \providecommand\color[2][]{%
    \errmessage{(Inkscape) Color is used for the text in Inkscape, but the package 'color.sty' is not loaded}%
    \renewcommand\color[2][]{}%
  }%
  \providecommand\transparent[1]{%
    \errmessage{(Inkscape) Transparency is used (non-zero) for the text in Inkscape, but the package 'transparent.sty' is not loaded}%
    \renewcommand\transparent[1]{}%
  }%
  \providecommand\rotatebox[2]{#2}%
  \ifx\svgwidth\undefined%
    \setlength{\unitlength}{267.40022095bp}%
    \ifx\svgscale\undefined%
      \relax%
    \else%
      \setlength{\unitlength}{\unitlength * \real{\svgscale}}%
    \fi%
  \else%
    \setlength{\unitlength}{\svgwidth}%
  \fi%
  \global\let\svgwidth\undefined%
  \global\let\svgscale\undefined%
  \makeatother%
  \begin{picture}(1,0.88008287)%
    \put(0,0){\includegraphics[width=\unitlength,page=1]{m1m2.pdf}}%
    \put(0.61644833,0.69077196){\color[rgb]{0,0,0}\makebox(0,0)[lt]{\begin{minipage}{0.42108787\unitlength}\raggedright $\gamma$\end{minipage}}}%
    \put(0.03510934,0.44509094){\color[rgb]{0,0,0}\makebox(0,0)[lb]{\smash{$0$}}}%
    \put(0.45740081,0.02030768){\color[rgb]{0,0,0}\makebox(0,0)[lb]{\smash{$0$}}}%
    \put(0,0){\includegraphics[width=\unitlength,page=2]{m1m2.pdf}}%
    \put(0.85185041,0.0502884){\color[rgb]{0,0,0}\makebox(0,0)[lt]{\begin{minipage}{0.17858589\unitlength}\raggedright $m_2$\end{minipage}}}%
    \put(-0.0037251,0.84598721){\color[rgb]{0,0,0}\makebox(0,0)[lb]{\smash{$m_1$}}}%
    \put(0.24357877,0.02766161){\color[rgb]{0,0,0}\makebox(0,0)[lb]{\smash{$-$}}}%
    \put(0.63439785,0.02205204){\color[rgb]{0,0,0}\makebox(0,0)[lb]{\smash{$+$}}}%
    \put(0.02558854,0.25143227){\color[rgb]{0,0,0}\makebox(0,0)[lb]{\smash{$-$}}}%
    \put(0.02258445,0.62955854){\color[rgb]{0,0,0}\makebox(0,0)[lb]{\smash{$+$}}}%
  \end{picture}%
\endgroup%

%% file: A2.pdf_tex
%% Creator: Inkscape inkscape 0.92.1, www.inkscape.org
%% Accompanies image file 'A2.pdf' (pdf, eps, ps)
\begingroup%
  \makeatletter%
  \providecommand\color[2][]{%
    \errmessage{(Inkscape) Color is used for the text in Inkscape, but the package 'color.sty' is not loaded}%
    \renewcommand\color[2][]{}%
  }%
  \providecommand\transparent[1]{%
    \errmessage{(Inkscape) Transparency is used (non-zero) for the text in Inkscape, but the package 'transparent.sty' is not loaded}%
    \renewcommand\transparent[1]{}%
  }%
  \providecommand\rotatebox[2]{#2}%
  \ifx\svgwidth\undefined%
    \setlength{\unitlength}{416.82762645bp}%
    \ifx\svgscale\undefined%
      \relax%
    \else%
      \setlength{\unitlength}{\unitlength * \real{\svgscale}}%
    \fi%
  \else%
    \setlength{\unitlength}{\svgwidth}%
  \fi%
  \global\let\svgwidth\undefined%
  \global\let\svgscale\undefined%
  \makeatother%
  \begin{picture}(1,0.9021422)%
    \put(0,0){\includegraphics[width=\unitlength,page=1]{A2.pdf}}%
    \put(0.86471374,0.53881845){\color[rgb]{0,0,0}\makebox(0,0)[lt]{\begin{minipage}{0.12595134\unitlength}\raggedright \end{minipage}}}%
    \put(0,0){\includegraphics[width=\unitlength,page=2]{A2.pdf}}%
    \put(0.45865696,0.8802694){\color[rgb]{0,0,0}\makebox(0,0)[lb]{\smash{$\f{d}_2$}}}%
    \put(0.66519477,0.88421236){\color[rgb]{0,0,0}\makebox(0,0)[lb]{\smash{}}}%
    \put(-0.0023897,0.00678956){\color[rgb]{0,0,0}\makebox(0,0)[lb]{\smash{$\f{d}_3$}}}%
    \put(0.01220261,0.16893351){\color[rgb]{0,0,0}\makebox(0,0)[lb]{\smash{$Q$}}}%
    \put(0,0){\includegraphics[width=\unitlength,page=3]{A2.pdf}}%
    \put(0.85823446,0.48067794){\color[rgb]{0,0,0}\makebox(0,0)[lb]{\smash{$\f{d}_1$}}}%
    \put(0,0){\includegraphics[width=\unitlength,page=3]{A2.pdf}}%
  \end{picture}%
\endgroup%

%% file: BrokenLineTropicalCurve.pdf_tex
%% Accompanies image file 'BrokenLineTropicalCurve.pdf' (pdf, eps, ps)
\begingroup%
  \makeatletter%
  \providecommand\color[2][]{%
    \errmessage{(Inkscape) Color is used for the text in Inkscape, but the package 'color.sty' is not loaded}%
    \renewcommand\color[2][]{}%
  }%
  \providecommand\transparent[1]{%
    \errmessage{(Inkscape) Transparency is used (non-zero) for the text in Inkscape, but the package 'transparent.sty' is not loaded}%
    \renewcommand\transparent[1]{}%
  }%
  \providecommand\rotatebox[2]{#2}%
  \newcommand*\fsize{\dimexpr\f@size pt\relax}%
  \newcommand*\lineheight[1]{\fontsize{\fsize}{#1\fsize}\selectfont}%
  \ifx\svgwidth\undefined%
    \setlength{\unitlength}{416.82762645bp}%
    \ifx\svgscale\undefined%
      \relax%
    \else%
      \setlength{\unitlength}{\unitlength * \real{\svgscale}}%
    \fi%
  \else%
    \setlength{\unitlength}{\svgwidth}%
  \fi%
  \global\let\svgwidth\undefined%
  \global\let\svgscale\undefined%
  \makeatother%
  \begin{picture}(1,0.9021422)%
    \lineheight{1}%
    \setlength\tabcolsep{0pt}%
    \put(0,0){\includegraphics[width=\unitlength,page=1]{BrokenLineTropicalCurve.pdf}}%
    \put(0.86471374,0.53881845){\color[rgb]{0,0,0}\makebox(0,0)[lt]{\begin{minipage}{0.12595134\unitlength}\raggedright \end{minipage}}}%
    \put(0,0){\includegraphics[width=\unitlength,page=2]{BrokenLineTropicalCurve.pdf}}%
    \put(0.46865696,0.8802694){\color[rgb]{0,0,0}\makebox(0,0)[lt]{\lineheight{1.25}\smash{\begin{tabular}[t]{l}$\f{d}_2$\end{tabular}}}}%
    \put(-0.0123897,0.00078956){\color[rgb]{0,0,0}\makebox(0,0)[lt]{\lineheight{1.25}\smash{\begin{tabular}[t]{l}$\f{d}_3$\end{tabular}}}}%
    \put(0.00220261,0.15893351){\color[rgb]{0,0,0}\makebox(0,0)[lt]{\lineheight{1.25}\smash{\begin{tabular}[t]{l}$Q$\end{tabular}}}}%
    \put(0,0){\includegraphics[width=\unitlength,page=3]{BrokenLineTropicalCurve.pdf}}%
    \put(0.85823446,0.48567794){\color[rgb]{0,0,0}\makebox(0,0)[lt]{\lineheight{1.25}\smash{\begin{tabular}[t]{l}$\f{d}_1$\end{tabular}}}}%
    \put(0,0){\includegraphics[width=\unitlength,page=1]{BrokenLineTropicalCurve.pdf}}%
  \end{picture}%
\endgroup%

%% file: Trop2.pdf_tex
%% Accompanies image file 'Trop2.pdf' (pdf, eps, ps)
\begingroup%
  \makeatletter%
  \providecommand\color[2][]{%
    \errmessage{(Inkscape) Color is used for the text in Inkscape, but the package 'color.sty' is not loaded}%
    \renewcommand\color[2][]{}%
  }%
  \providecommand\transparent[1]{%
    \errmessage{(Inkscape) Transparency is used (non-zero) for the text in Inkscape, but the package 'transparent.sty' is not loaded}%
    \renewcommand\transparent[1]{}%
  }%
  \providecommand\rotatebox[2]{#2}%
  \newcommand*\fsize{\dimexpr\f@size pt\relax}%
  \newcommand*\lineheight[1]{\fontsize{\fsize}{#1\fsize}\selectfont}%
  \ifx\svgwidth\undefined%
    \setlength{\unitlength}{112.89258525bp}%
    \ifx\svgscale\undefined%
      \relax%
    \else%
      \setlength{\unitlength}{\unitlength * \real{\svgscale}}%
    \fi%
  \else%
    \setlength{\unitlength}{\svgwidth}%
  \fi%
  \global\let\svgwidth\undefined%
  \global\let\svgscale\undefined%
  \makeatother%
  \begin{picture}(1,0.99543467)%
    \lineheight{1}%
    \setlength\tabcolsep{0pt}%
    \put(0,0){\includegraphics[width=\unitlength,page=1]{Trop2.pdf}}%
    \put(0.140655,0.15647843){\color[rgb]{0,0,0}\makebox(0,0)[lt]{\begin{minipage}{0.59791349\unitlength}\raggedright $E_2$\end{minipage}}}%
    \put(0.69542069,0.16810454){\color[rgb]{0,0,0}\makebox(0,0)[lt]{\begin{minipage}{0.59791349\unitlength}\raggedright $E_3$\end{minipage}}}%
    \put(-0.00882338,0.72447886){\color[rgb]{0,0,0}\makebox(0,0)[lt]{\begin{minipage}{0.59791349\unitlength}\raggedright $E_1$\end{minipage}}}%
    \put(0.55587264,0.95539206){\color[rgb]{0,0,0}\makebox(0,0)[lt]{\begin{minipage}{0.59791349\unitlength}\raggedright $E_4$\end{minipage}}}%
  \end{picture}%
\endgroup%

%% file: Trop0.pdf_tex
%% Accompanies image file 'Trop0.pdf' (pdf, eps, ps)
\begingroup%
  \makeatletter%
  \providecommand\color[2][]{%
    \errmessage{(Inkscape) Color is used for the text in Inkscape, but the package 'color.sty' is not loaded}%
    \renewcommand\color[2][]{}%
  }%
  \providecommand\transparent[1]{%
    \errmessage{(Inkscape) Transparency is used (non-zero) for the text in Inkscape, but the package 'transparent.sty' is not loaded}%
    \renewcommand\transparent[1]{}%
  }%
  \providecommand\rotatebox[2]{#2}%
  \newcommand*\fsize{\dimexpr\f@size pt\relax}%
  \newcommand*\lineheight[1]{\fontsize{\fsize}{#1\fsize}\selectfont}%
  \ifx\svgwidth\undefined%
    \setlength{\unitlength}{99.0175758bp}%
    \ifx\svgscale\undefined%
      \relax%
    \else%
      \setlength{\unitlength}{\unitlength * \real{\svgscale}}%
    \fi%
  \else%
    \setlength{\unitlength}{\svgwidth}%
  \fi%
  \global\let\svgwidth\undefined%
  \global\let\svgscale\undefined%
  \makeatother%
  \begin{picture}(1,0.75684413)%
    \lineheight{1}%
    \setlength\tabcolsep{0pt}%
    \put(0,0){\includegraphics[width=\unitlength,page=1]{Trop0.pdf}}%
    \put(-0.01005977,0.63844016){\color[rgb]{0,0,0}\makebox(0,0)[lt]{\begin{minipage}{0.68169715\unitlength}\raggedright $E_1$\end{minipage}}}%
    \put(0.0903012,0.1158057){\color[rgb]{0,0,0}\makebox(0,0)[lt]{\begin{minipage}{0.68169715\unitlength}\raggedright $E_2$\end{minipage}}}%
    \put(0.65133972,0.12163533){\color[rgb]{0,0,0}\makebox(0,0)[lt]{\begin{minipage}{0.68169715\unitlength}\raggedright $E_3$\end{minipage}}}%
    \put(0.61725489,0.71767386){\color[rgb]{0,0,0}\makebox(0,0)[lt]{\begin{minipage}{0.68169715\unitlength}\raggedright $E_4$\end{minipage}}}%
  \end{picture}%
\endgroup%

%% file: Trop1.pdf_tex
%% Accompanies image file 'Trop1.pdf' (pdf, eps, ps)
\begingroup%
  \makeatletter%
  \providecommand\color[2][]{%
    \errmessage{(Inkscape) Color is used for the text in Inkscape, but the package 'color.sty' is not loaded}%
    \renewcommand\color[2][]{}%
  }%
  \providecommand\transparent[1]{%
    \errmessage{(Inkscape) Transparency is used (non-zero) for the text in Inkscape, but the package 'transparent.sty' is not loaded}%
    \renewcommand\transparent[1]{}%
  }%
  \providecommand\rotatebox[2]{#2}%
  \newcommand*\fsize{\dimexpr\f@size pt\relax}%
  \newcommand*\lineheight[1]{\fontsize{\fsize}{#1\fsize}\selectfont}%
  \ifx\svgwidth\undefined%
    \setlength{\unitlength}{168.95507627bp}%
    \ifx\svgscale\undefined%
      \relax%
    \else%
      \setlength{\unitlength}{\unitlength * \real{\svgscale}}%
    \fi%
  \else%
    \setlength{\unitlength}{\svgwidth}%
  \fi%
  \global\let\svgwidth\undefined%
  \global\let\svgscale\undefined%
  \makeatother%
  \begin{picture}(1,0.51156763)%
    \lineheight{1}%
    \setlength\tabcolsep{0pt}%
    \put(0,0){\includegraphics[width=\unitlength,page=1]{Trop1.pdf}}%
    \put(0.03849488,0.0595901){\color[rgb]{0,0,0}\makebox(0,0)[lt]{\begin{minipage}{0.39951448\unitlength}\raggedright $E_2$\end{minipage}}}%
    \put(0.78980406,0.13062704){\color[rgb]{0,0,0}\makebox(0,0)[lt]{\begin{minipage}{0.39951448\unitlength}\raggedright $E_3$\end{minipage}}}%
    \put(-0.00589561,0.37977478){\color[rgb]{0,0,0}\makebox(0,0)[lt]{\begin{minipage}{0.39951448\unitlength}\raggedright $E_1$\end{minipage}}}%
    \put(0.74541357,0.51239746){\color[rgb]{0,0,0}\makebox(0,0)[lt]{\begin{minipage}{0.39951448\unitlength}\raggedright $E_4$\end{minipage}}}%
  \end{picture}%
\endgroup%

%% file: Trop3.pdf_tex
%% Accompanies image file 'Trop3.pdf' (pdf, eps, ps)
\begingroup%
  \makeatletter%
  \providecommand\color[2][]{%
    \errmessage{(Inkscape) Color is used for the text in Inkscape, but the package 'color.sty' is not loaded}%
    \renewcommand\color[2][]{}%
  }%
  \providecommand\transparent[1]{%
    \errmessage{(Inkscape) Transparency is used (non-zero) for the text in Inkscape, but the package 'transparent.sty' is not loaded}%
    \renewcommand\transparent[1]{}%
  }%
  \providecommand\rotatebox[2]{#2}%
  \newcommand*\fsize{\dimexpr\f@size pt\relax}%
  \newcommand*\lineheight[1]{\fontsize{\fsize}{#1\fsize}\selectfont}%
  \ifx\svgwidth\undefined%
    \setlength{\unitlength}{102.14437139bp}%
    \ifx\svgscale\undefined%
      \relax%
    \else%
      \setlength{\unitlength}{\unitlength * \real{\svgscale}}%
    \fi%
  \else%
    \setlength{\unitlength}{\svgwidth}%
  \fi%
  \global\let\svgwidth\undefined%
  \global\let\svgscale\undefined%
  \makeatother%
  \begin{picture}(1,0.73970768)%
    \lineheight{1}%
    \setlength\tabcolsep{0pt}%
    \put(0,0){\includegraphics[width=\unitlength,page=1]{Trop3.pdf}}%
    \put(0.10038668,0.12164848){\color[rgb]{0,0,0}\makebox(0,0)[lt]{\begin{minipage}{0.66082936\unitlength}\raggedright $E_2$\end{minipage}}}%
    \put(-0.00975182,0.52048864){\color[rgb]{0,0,0}\makebox(0,0)[lt]{\begin{minipage}{0.66082936\unitlength}\raggedright $E_1$\end{minipage}}}%
    \put(0.66231886,0.20579524){\color[rgb]{0,0,0}\makebox(0,0)[lt]{\begin{minipage}{0.66082936\unitlength}\raggedright $E_3$\end{minipage}}}%
    \put(0.60038634,0.68699125){\color[rgb]{0,0,0}\makebox(0,0)[lt]{\begin{minipage}{0.66082936\unitlength}\raggedright $E_4$\end{minipage}}}%
  \end{picture}%
\endgroup%

%% file: b1a.pdf_tex
%% Creator: Inkscape 0.48.2, www.inkscape.org
%% Accompanies image file 'b1a.pdf' (pdf, eps, ps)
\begingroup%
  \makeatletter%
  \providecommand\color[2][]{%
    \errmessage{(Inkscape) Color is used for the text in Inkscape, but the package 'color.sty' is not loaded}%
    \renewcommand\color[2][]{}%
  }%
  \providecommand\transparent[1]{%
    \errmessage{(Inkscape) Transparency is used (non-zero) for the text in Inkscape, but the package 'transparent.sty' is not loaded}%
    \renewcommand\transparent[1]{}%
  }%
  \providecommand\rotatebox[2]{#2}%
  \ifx\svgwidth\undefined%
    \setlength{\unitlength}{164bp}%
    \ifx\svgscale\undefined%
      \relax%
    \else%
      \setlength{\unitlength}{\unitlength * \real{\svgscale}}%
    \fi%
  \else%
    \setlength{\unitlength}{\svgwidth}%
  \fi%
  \global\let\svgwidth\undefined%
  \global\let\svgscale\undefined%
  \makeatother%
  \begin{picture}(1,1.01858733)%
    \put(0,0){\includegraphics[width=\unitlength]{b1a.pdf}}%
    \put(0.06090165,0.97282704){\color[rgb]{0,0,0}\makebox(0,0)[lb]{\smash{$Q$}}}%
  \end{picture}%
\endgroup%

%% file: b10.pdf_tex
%% Creator: Inkscape 0.48.2, www.inkscape.org
%% Accompanies image file 'b10.pdf' (pdf, eps, ps)
\begingroup%
  \makeatletter%
  \providecommand\color[2][]{%
    \errmessage{(Inkscape) Color is used for the text in Inkscape, but the package 'color.sty' is not loaded}%
    \renewcommand\color[2][]{}%
  }%
  \providecommand\transparent[1]{%
    \errmessage{(Inkscape) Transparency is used (non-zero) for the text in Inkscape, but the package 'transparent.sty' is not loaded}%
    \renewcommand\transparent[1]{}%
  }%
  \providecommand\rotatebox[2]{#2}%
  \ifx\svgwidth\undefined%
    \setlength{\unitlength}{164bp}%
    \ifx\svgscale\undefined%
      \relax%
    \else%
      \setlength{\unitlength}{\unitlength * \real{\svgscale}}%
    \fi%
  \else%
    \setlength{\unitlength}{\svgwidth}%
  \fi%
  \global\let\svgwidth\undefined%
  \global\let\svgscale\undefined%
  \makeatother%
  \begin{picture}(1,1.00900548)%
    \put(0,0){\includegraphics[width=\unitlength]{b10.pdf}}%
    \put(0.22553581,0.96324519){\color[rgb]{0,0,0}\makebox(0,0)[lb]{\smash{$Q$}}}%
  \end{picture}%
\endgroup%

%% file: b1c.pdf_tex
%% Creator: Inkscape 0.48.2, www.inkscape.org
%% Accompanies image file 'b1c.pdf' (pdf, eps, ps)
\begingroup%
  \makeatletter%
  \providecommand\color[2][]{%
    \errmessage{(Inkscape) Color is used for the text in Inkscape, but the package 'color.sty' is not loaded}%
    \renewcommand\color[2][]{}%
  }%
  \providecommand\transparent[1]{%
    \errmessage{(Inkscape) Transparency is used (non-zero) for the text in Inkscape, but the package 'transparent.sty' is not loaded}%
    \renewcommand\transparent[1]{}%
  }%
  \providecommand\rotatebox[2]{#2}%
  \ifx\svgwidth\undefined%
    \setlength{\unitlength}{165.15bp}%
    \ifx\svgscale\undefined%
      \relax%
    \else%
      \setlength{\unitlength}{\unitlength * \real{\svgscale}}%
    \fi%
  \else%
    \setlength{\unitlength}{\svgwidth}%
  \fi%
  \global\let\svgwidth\undefined%
  \global\let\svgscale\undefined%
  \makeatother%
  \begin{picture}(1,1.06757647)%
    \put(0,0){\includegraphics[width=\unitlength]{b1c.pdf}}%
    \put(0.36689628,1.02213483){\color[rgb]{0,0,0}\makebox(0,0)[lb]{\smash{$Q$}}}%
  \end{picture}%
\endgroup%

%% file: b1b.pdf_tex
%% Creator: Inkscape 0.48.2, www.inkscape.org
%% Accompanies image file 'b1b.pdf' (pdf, eps, ps)
\begingroup%
  \makeatletter%
  \providecommand\color[2][]{%
    \errmessage{(Inkscape) Color is used for the text in Inkscape, but the package 'color.sty' is not loaded}%
    \renewcommand\color[2][]{}%
  }%
  \providecommand\transparent[1]{%
    \errmessage{(Inkscape) Transparency is used (non-zero) for the text in Inkscape, but the package 'transparent.sty' is not loaded}%
    \renewcommand\transparent[1]{}%
  }%
  \providecommand\rotatebox[2]{#2}%
  \ifx\svgwidth\undefined%
    \setlength{\unitlength}{170.35bp}%
    \ifx\svgscale\undefined%
      \relax%
    \else%
      \setlength{\unitlength}{\unitlength * \real{\svgscale}}%
    \fi%
  \else%
    \setlength{\unitlength}{\svgwidth}%
  \fi%
  \global\let\svgwidth\undefined%
  \global\let\svgscale\undefined%
  \makeatother%
  \begin{picture}(1,1.03501795)%
    \put(0,0){\includegraphics[width=\unitlength]{b1b.pdf}}%
    \put(0.35864595,0.99096343){\color[rgb]{0,0,0}\makebox(0,0)[lb]{\smash{$Q$}}}%
  \end{picture}%
\endgroup%

%% file: t1a.pdf_tex
%% Creator: Inkscape 0.48.2, www.inkscape.org
%% Accompanies image file 't1a.pdf' (pdf, eps, ps)
\begingroup%
  \makeatletter%
  \providecommand\color[2][]{%
    \errmessage{(Inkscape) Color is used for the text in Inkscape, but the package 'color.sty' is not loaded}%
    \renewcommand\color[2][]{}%
  }%
  \providecommand\transparent[1]{%
    \errmessage{(Inkscape) Transparency is used (non-zero) for the text in Inkscape, but the package 'transparent.sty' is not loaded}%
    \renewcommand\transparent[1]{}%
  }%
  \providecommand\rotatebox[2]{#2}%
  \ifx\svgwidth\undefined%
    \setlength{\unitlength}{199.73404541bp}%
    \ifx\svgscale\undefined%
      \relax%
    \else%
      \setlength{\unitlength}{\unitlength * \real{\svgscale}}%
    \fi%
  \else%
    \setlength{\unitlength}{\svgwidth}%
  \fi%
  \global\let\svgwidth\undefined%
  \global\let\svgscale\undefined%
  \makeatother%
  \begin{picture}(1,0.90541506)%
    \put(0,0){\includegraphics[width=\unitlength]{t1a.pdf}}%
    \put(0.86315396,0.13636811){\color[rgb]{0,0,0}\makebox(0,0)[lb]{\smash{$E_3$}}}%
    \put(0.4455184,0.00955176){\color[rgb]{0,0,0}\makebox(0,0)[lb]{\smash{$E_2$}}}%
    \put(-0.00171322,0.45349465){\color[rgb]{0,0,0}\makebox(0,0)[lb]{\smash{$E_1$}}}%
    \put(0.14525523,0.86784165){\color[rgb]{0,0,0}\makebox(0,0)[lb]{\smash{$E_4$}}}%
  \end{picture}%
\endgroup%

%% file: t10.pdf_tex
%% Creator: Inkscape 0.48.2, www.inkscape.org
%% Accompanies image file 't10.pdf' (pdf, eps, ps)
\begingroup%
  \makeatletter%
  \providecommand\color[2][]{%
    \errmessage{(Inkscape) Color is used for the text in Inkscape, but the package 'color.sty' is not loaded}%
    \renewcommand\color[2][]{}%
  }%
  \providecommand\transparent[1]{%
    \errmessage{(Inkscape) Transparency is used (non-zero) for the text in Inkscape, but the package 'transparent.sty' is not loaded}%
    \renewcommand\transparent[1]{}%
  }%
  \providecommand\rotatebox[2]{#2}%
  \ifx\svgwidth\undefined%
    \setlength{\unitlength}{159.73288574bp}%
    \ifx\svgscale\undefined%
      \relax%
    \else%
      \setlength{\unitlength}{\unitlength * \real{\svgscale}}%
    \fi%
  \else%
    \setlength{\unitlength}{\svgwidth}%
  \fi%
  \global\let\svgwidth\undefined%
  \global\let\svgscale\undefined%
  \makeatother%
  \begin{picture}(1,0.8682476)%
    \put(0,0){\includegraphics[width=\unitlength]{t10.pdf}}%
    \put(0.82888425,0.15589105){\color[rgb]{0,0,0}\makebox(0,0)[lb]{\smash{$E_3$}}}%
    \put(0.30214428,0.01194377){\color[rgb]{0,0,0}\makebox(0,0)[lb]{\smash{$E_2$}}}%
    \put(-0.00214225,0.39024653){\color[rgb]{0,0,0}\makebox(0,0)[lb]{\smash{$E_1$}}}%
    \put(0.15748505,0.82126487){\color[rgb]{0,0,0}\makebox(0,0)[lb]{\smash{$E_4$}}}%
  \end{picture}%
\endgroup%

%% file: t1c.pdf_tex
%% Creator: Inkscape 0.48.2, www.inkscape.org
%% Accompanies image file 't1c.pdf' (pdf, eps, ps)
\begingroup%
  \makeatletter%
  \providecommand\color[2][]{%
    \errmessage{(Inkscape) Color is used for the text in Inkscape, but the package 'color.sty' is not loaded}%
    \renewcommand\color[2][]{}%
  }%
  \providecommand\transparent[1]{%
    \errmessage{(Inkscape) Transparency is used (non-zero) for the text in Inkscape, but the package 'transparent.sty' is not loaded}%
    \renewcommand\transparent[1]{}%
  }%
  \providecommand\rotatebox[2]{#2}%
  \ifx\svgwidth\undefined%
    \setlength{\unitlength}{166.13293457bp}%
    \ifx\svgscale\undefined%
      \relax%
    \else%
      \setlength{\unitlength}{\unitlength * \real{\svgscale}}%
    \fi%
  \else%
    \setlength{\unitlength}{\svgwidth}%
  \fi%
  \global\let\svgwidth\undefined%
  \global\let\svgscale\undefined%
  \makeatother%
  \begin{picture}(1,0.88295374)%
    \put(0,0){\includegraphics[width=\unitlength]{t1c.pdf}}%
    \put(0.83547626,0.19803977){\color[rgb]{0,0,0}\makebox(0,0)[lb]{\smash{$E_3$}}}%
    \put(0.29050487,0.01148365){\color[rgb]{0,0,0}\makebox(0,0)[lb]{\smash{$E_2$}}}%
    \put(-0.00205972,0.37521281){\color[rgb]{0,0,0}\makebox(0,0)[lb]{\smash{$E_1$}}}%
    \put(0.19957265,0.83778095){\color[rgb]{0,0,0}\makebox(0,0)[lb]{\smash{$E_4$}}}%
  \end{picture}%
\endgroup%

%% file: t1b.pdf_tex
%% Creator: Inkscape 0.48.2, www.inkscape.org
%% Accompanies image file 't1b.pdf' (pdf, eps, ps)
\begingroup%
  \makeatletter%
  \providecommand\color[2][]{%
    \errmessage{(Inkscape) Color is used for the text in Inkscape, but the package 'color.sty' is not loaded}%
    \renewcommand\color[2][]{}%
  }%
  \providecommand\transparent[1]{%
    \errmessage{(Inkscape) Transparency is used (non-zero) for the text in Inkscape, but the package 'transparent.sty' is not loaded}%
    \renewcommand\transparent[1]{}%
  }%
  \providecommand\rotatebox[2]{#2}%
  \ifx\svgwidth\undefined%
    \setlength{\unitlength}{156.44880371bp}%
    \ifx\svgscale\undefined%
      \relax%
    \else%
      \setlength{\unitlength}{\unitlength * \real{\svgscale}}%
    \fi%
  \else%
    \setlength{\unitlength}{\svgwidth}%
  \fi%
  \global\let\svgwidth\undefined%
  \global\let\svgscale\undefined%
  \makeatother%
  \begin{picture}(1,0.86623722)%
    \put(0,0){\includegraphics[width=\unitlength]{t1b.pdf}}%
    \put(0.82529229,0.14258795){\color[rgb]{0,0,0}\makebox(0,0)[lb]{\smash{$E_3$}}}%
    \put(0.22979419,0.01219448){\color[rgb]{0,0,0}\makebox(0,0)[lb]{\smash{$E_2$}}}%
    \put(-0.00218722,0.33546349){\color[rgb]{0,0,0}\makebox(0,0)[lb]{\smash{$E_1$}}}%
    \put(0.13560122,0.81826826){\color[rgb]{0,0,0}\makebox(0,0)[lb]{\smash{$E_4$}}}%
  \end{picture}%
\endgroup%